\documentclass[reqno,11pt]{amsart}
\usepackage[active]{srcltx}
\usepackage{color} 
\usepackage{epsfig, enumitem}
\usepackage{latexsym}
\usepackage{pdfpages}
\usepackage{url}
\usepackage{ulem}

\usepackage{ifpdf}

 \usepackage[usenames,dvipsnames]{pstricks}
 \usepackage{epsfig}
 \usepackage{pst-grad} 
 \usepackage{pst-plot} 
\def\qed{\hfill$\Box$ \vskip.3cm}

\def\1{\raisebox{2pt}{\rm{$\chi$}}}

\def\XXint#1#2#3{{\setbox0=\hbox{$#1{#2#3}{\int}$}
     \vcenter{\hbox{$#2#3$}}\kern-.5\wd0}}

\newcommand{\llceil}{\left\lceil}
\newcommand{\rrceil}{\right\rceil}

\newcommand{\threepartdef}[6]
{
	\left\{
		\begin{array}{lll}
			#1 & \mbox{if } #2 \\
			#3 & \mbox{if } #4 \\
			#5 & \mbox{if } #6
		\end{array}
	\right.
}

\newcommand{\mres}{\mathbin{\vrule height 1.6ex depth 0pt width
    0.13ex\vrule height 0.13ex depth 0pt width 1.3ex}}

\usepackage{mathtools}
\mathtoolsset{showonlyrefs}

\usepackage{pgfplots}
\pgfplotsset{compat=1.15}
\usepackage{mathrsfs}
\usetikzlibrary{arrows}

\newtheorem{theorem}{Theorem}[section]
\newtheorem{lemma}[theorem]{Lemma}
\newtheorem{definition}[theorem]{Definition}
\newtheorem{proposition}[theorem]{Proposition}
\newtheorem{corollary}[theorem]{Corollary}

\newtheorem*{theorem*}{\it Theorem}

\numberwithin{equation}{section}

\begin{document}

\title[Double bubbles on square lattice]{The double-bubble problem on the square lattice}


\author[M. Friedrich]{Manuel Friedrich}
\address[Manuel Friedrich]{Department of Mathematics, FAU Erlangen-N\"urnberg, Cauerstra\ss e 1, 91058 Erlangen, Germany}
\email{manuel.friedrich@uni-muenster.de}
\urladdr{\url{https://www.uni-muenster.de/AMM/en/Friedrich/}}

\author[W. G\'orny]{Wojciech G\'orny}
\address[Wojciech G\'orny]{Faculty of Mathematics, Informatics and
  Mechanics, University of Warsaw, Banacha 2, 02-097 Warsaw, Poland
  and Faculty of Mathematics, University of
  Vienna, Oskar-Morgenstern-Platz 1, A-1090 Vienna, Austria}
\email{wojciech.gorny@univie.ac.at}
\urladdr{\url{https://www.mat.univie.ac.at/~wgorny}}

\author[U. Stefanelli]{Ulisse Stefanelli}
\address[Ulisse Stefanelli]{Faculty of Mathematics, University of
  Vienna, Oskar-Morgenstern-Platz 1, A-1090 Vienna, Austria,
Vienna Research Platform on Accelerating
  Photoreaction Discovery, University of Vienna, W\"ahringerstra\ss e 17, 1090 Vienna, Austria,
 \& Istituto di
  Matematica Applicata e Tecnologie Informatiche {\it E. Magenes}, via
  Ferrata 1, I-27100 Pavia, Italy
}
\email{ulisse.stefanelli@univie.ac.at}
\urladdr{\url{http://www.mat.univie.ac.at/~stefanelli}}




\begin{abstract}

We investigate  minimal-perimeter configurations  of two finite sets of points on the square lattice. This corresponds to a
lattice version of the classical double-bubble problem. We  give  a detailed
description of the fine geometry of  minimisers  and, in some
parameter regime, we compute the optimal  perimeter  as a function of the
size of the point sets. Moreover,
we provide a sharp bound on the difference between two minimisers,
which are generally not unique, and use it to rigorously
identify their Wulff shape, as the size of the point sets  scales up.
\end{abstract}

\maketitle


\section{Introduction}

The classical double-bubble problem is concerned with  the shape of
two sets of given volume under minimisation of their surface area. In the
Euclidean space,  minimisers  are enclosed by three spherical
caps, intersecting at an angle of $2\pi/3$. The proof of this fact in
${\mathbb R}^2$ dates back to \cite{Foisy}, and has then been
extended to
${\mathbb R}^3$ \cite{Hutchings} and ${\mathbb R}^n$ for $n \geq
4$ \cite{Reichardt}.  See also \cite{Cicalese2} for a quantitative stability analysis
in two dimensions. A number of variants of the problem  has  also been
tackled, including double bubbles in spherical and hyperbolic spaces 
\cite{Corneli,Corneli2,Cotton,Masters},  hyperbolic surfaces \cite{Boyer}, cones
\cite{Lopez,Morgan}, the $3$-torus \cite{Carrion,Corneli0},  the
Gau\ss\ space \cite{Corneli2,Milman},  and in the anisotropic
Grushin plane \cite{Franceschi}. 

The aim of this paper is to tackle a  lattice  version of the
double-bubble problem. We restrict our attention to the 
square lattice ${\mathbb Z}^2$ and define the {\it lattice length} of
the interface
separating two disjoint sets $C,\,D \subset {\mathbb Z}^2$ as $Q(C,D)
= \#\{(c,d)\in C\times D  \colon \, |c-d|=1\}$,  where $|\cdot|$
is the Euclidean norm.  The {\it lattice
  double-bubble problem} consists in finding 
two distinct lattice subsets  
$A$ and $B$ of fixed sizes  $N_A,N_B \in \mathbb{N}$  solving
\begin{equation}
  \label{eq:dbp}
  \min\{P(A,B) \colon \ A,\, B \subset {\mathbb Z}^2, \ A \cap B =
  \emptyset, \ \#A = N_A, \ \#B = N_B\},
\end{equation}
where the {\it lattice perimeter} $P(A,B)$ is defined by
\begin{align}\label{eq:dbp3} P(A,B) &= Q(A,A^c) + Q(B,B^c) - 2\beta Q(A,B)\notag\\
  &= Q(A,A^c\setminus B) + Q(B,B^c\setminus A) +  (2-2\beta) Q(A,B).
\end{align}
The latter  definition  features the parameter $\beta\in (0,1)$.
Note that the classical double-bubble case corresponds to the choice
$\beta=1/2$. In the following, we allow for the more general
$\beta\in (0,1)$, for this will be  relevant  in connection
with applications,  see Section \ref{sec:equiv}.  In particular,
$\beta$ models the interaction between the two sets. The reader is
referred to
\cite{Futer} where cost-minimizing networks featuring different
interaction costs are considered. 

Analogously to the Euclidean case, we prove that minimisers $(A,B)$ of
\eqref{eq:dbp} are connected ($A$, $B$, and $A\cup B$ are connected in
the usual lattice sense, see below). Call {\it isoperimetric}
those subsets of the lattice which minimize $C \mapsto Q(C,C^c)$ under
given cardinality.  Without claiming completeness, the reader is
referred to the monograph \cite{Harper} and to \cite{Bezrukov0,Biskup,Bobkov,Bollobas,Wang} for a
minimal collection of results on discrete
isoperimetric inequalities, to \cite{Cicalese,Mainini,Mainini2} for sharp
fluctuation estimates, and to \cite{Barrett} for
some numerical approximation.   A second analogy with the Euclidean setting is that optimal pairs $(A,B)$ {\it do
  not} consist of the mere union of two isoperimetric sets $A$ and
$B$, for the onset of an interface between $A$ and $B$ influences
their shape.

\begin{figure}[h]

\definecolor{zzttqq}{rgb}{0.6,0.2,0.}
\begin{tikzpicture}[line cap=round,line join=round,>=triangle 45,x=2.0cm,y=2.0cm]
\clip(-2.3,-0.2) rectangle (1.3,2.2);
\fill[line width=2.pt,color=zzttqq,fill=zzttqq,fill opacity=0.1] (-0.5,2.) -- (1.,2.) -- (1.,0.) -- (-0.5,0.) -- cycle;
\draw[line width=0.5pt] (-2.,2.)-- (-0.5,2.);
\draw[line width=0.5pt] (-0.5,2.)-- (-0.5,0.);
\draw[line width=0.5pt] (-0.5,0.)-- (-2.,0.);
\draw[line width=0.5pt] (-2.,0.)-- (-2.,2.);
\draw[line width=0.5pt] (-0.5,2.)-- (1.,2.);
\draw[line width=0.5pt] (1.,2.)-- (1.,0.);
\draw[line width=0.5pt] (1.,0.)-- (-0.5,0.);
\end{tikzpicture}

\caption{A minimiser for $\beta=1/2$}
\label{fig:aspectratio}
  \end{figure}
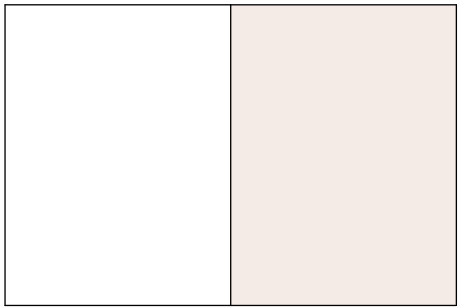

 Differently from  the Euclidean case, existence of minimisers for
\eqref{eq:dbp} is here obvious, for the minimisation problem is finite. Moreover,
the geometry of the intersection of interfaces is much simplified, as
effect of the discrete geometry of the underlying lattice. In
particular, all interfaces meet at multiples of $\pi/2$ angles.

At finite sizes $N_A, \,  N_B $, boundary effects are relevant and a
whole menagerie of minimisers  of \eqref{eq:dbp} may arise, depending
on the specific values of $N_A,  \, N_B $, and $\beta$. Indeed,
although uniqueness holds in some special cases, it cannot be expected
in general. We are however able to prove  an a priori estimate on the
symmetric distance of two minimisers, which differ  at most by  
$N^{1/2}_A=N^{1/2}_B$  points.

As size scales up, whereas properly rescaled isoperimetric sets approach the square, $A$ and $B$ converge to suitable rectangles. In the limit $N_A = N_B \to \infty$ (and for $\beta=1/2$), we prove that minimisers of \eqref{eq:dbp}
converge to the {\it Wulff shape} configuration of Figure
\ref{fig:aspectratio}. That is, uniqueness is restored in the Wulff
shape limit.  In fact, 
 in the crystalline-perimeter case, the 
double-bubble problem for
$\beta=1/2$ has been already tackled in \cite{Morgan98}, see also the recent \cite{Duncan0} for an elementary
proof of the existence of minimisers. The case $\beta\not = 1/2$ is
addressed in \cite{Wecht} instead. In particular, the different possible
  geometries of the Wulff shape, corresponding to different volume
  fractions of the two phases, have been identified.

Let us now present our main results. We start  by associating to each   $\mathcal{V} \subset
{\mathbb Z}^2$  
the corresponding {\it unit-disk graph}, namely the undirected simple graph  $G=(\mathcal{V},\mathcal{E})$,  where vertices are identified  with the points in  $\mathcal{V}$,  and the set  $\mathcal{E}\subset \mathcal{V} \times \mathcal{V}$  of edges contains one edge for each pair of points in  $\mathcal{V}$  at distance $1$. We say that a subset 
$ \mathcal{V}  \subset {\mathbb Z}^2$ is {\it connected} if the corresponding unit-disk
graph is connected. Moreover, we indicate by $R_z : = {\mathbb  Z}\times
\{z\}$ and $C_z=\{z\}\times {\mathbb  Z}$ rows and columns, for all $z \in {\mathbb  Z}$.

    Our main findings read as follows. 

\begin{theorem}\label{thm:main} 
  Let $(A,B)$ solve the double-bubble problem \eqref{eq:dbp}. Then,
  \begin{enumerate}
  \item[\rm i] {\rm (Connectedness)} The sets $A$, $B$, and $A \cup B$ are
    connected. Moreover, the sets $A\cap R_z$, $B\cap R_z$, $(A \cup B)\cap R_z$,
    $A\cap C_z$, $B\cap C_z$, and $(A \cup B)\cap C_z$  are connected
    (possibly being empty) for all $z\in {\mathbb
      Z}$;\\ 
  \item[\rm ii] {\rm (Separation)} If $\max\{x  \colon\, (x,z)\in A\} \leq \min\{x 
    \colon\, (x,z)\in B\}  - 1 $ {\rm for some} $z \in  {\mathbb  Z} $, then the
    same holds with equality {\rm for all}   $z \in  {\mathbb  Z} $
    (whenever not empty). An analogous
    statement is valid for columns,  possibly after  exchanging the role of $A$
    and $B$;\\
  \item[\rm iii] {\rm (Interface)} Let $I\subset {\mathbb R}^2$ be the
    set of midpoints of segments connecting points in $A$ with points in
    $B$ at distance $1$. Then, for all $x\in I$ there exists $y\in I\setminus\{x\}$
    with $|x-y| \in \lbrace   1/ \sqrt{2}, 1\rbrace $ and $I$ can be
     included in the image of a piecewise-affine curve  $\iota \colon [0,1] \to {\mathbb R}^2 $ with monotone components. 
\end{enumerate}
If $N_A=N_B= N $ and $\beta\leq 1/2$, we additionally have that 
\begin{enumerate}
\item[\rm iv]  {\rm (Minimal perimeter)}  
    \begin{align}\label{eq: periper}
    P(A,B)=\min_{h \in \mathbb{N}} \big(4\llceil N/h\rrceil + 2h( 2 - \beta )\big),
    \end{align}
where all minimisers $h$ satisfy $|h- \sqrt{2  N /(2-\beta)}| \le C_\beta N^{1/4}$ for some  constant $C_\beta$      only  depending on $\beta$. For $\beta \in \mathbb{R} \setminus \mathbb{Q}$, there exists a unique minimiser of \eqref{eq: periper}. 
\item[\rm v]    {\rm (Explicit solution)}  Let $h$ minimize
  \eqref{eq: periper} and $\ell
  \in {\mathbb N}$ and $0 \le r < h$  be given with $ N  = h\ell
  +r$. Then, letting 
  \begin{align*}
    A'&:=\{(x,y)\in {\mathbb Z}^2 \, \colon \, x \in[-\ell + 1 ,0], \, y \in
        [1,h] \ \text{or} \ x=-\ell,\, y \in [1,r]\},\\
     B'&:=\{(x,y)\in {\mathbb Z}^2 \, \colon \, x \in[1,\ell], \, y \in
        [1,h] \ \text{or} \ x=\ell+1,\, y \in [1,r]\},
\end{align*}
the pair $(A',B')$ solves the
    double-bubble problem \eqref{eq:dbp};\\
  \item[\rm vi] {\rm (Fluctuations)} There exists a constant $C_\beta$
     only  depending on $\beta$  and an isometry $T$ of ${\mathbb
       Z}^2$   such that 
\begin{align}\label{eq:fluct}
\# (A \triangle  T(A')) + \# (B
    \triangle T( B')) & \leq C_\beta N^{1/2} \quad \quad \text{if $\beta \in \mathbb{R}\setminus \mathbb{Q}$}, \notag \\
    \# (A \triangle  T(A')) + \# (B
    \triangle T( B')) & \leq C_\beta N^{3/4} \quad \quad \text{if $\beta \in \mathbb{Q}$}
\end{align}
 where the pair $(A',B')$ is defined in {\rm v}. 
 (See beginning of Section \ref{sec:law} for the definition of 
 isometry.) 
  \end{enumerate}
\end{theorem}

Theorem \ref{thm:main} is proved in subsequent steps along the paper, by carefully characterising the
geometry of optimal pairs $(A,B)$. In fact, our analysis reveals
additional geometrical details, so that the statements in the coming
sections are
often more precise and more general
in terms of conditions on the parameters $N_A$, $N_B$, and
$\beta$ with respect to Theorem~\ref{thm:main}. We prefer to postpone these details   in order not
to overburden the introduction.  

The connectedness of
optimal pairs $(A,B)$ is discussed in Section \ref{sec:algo} and Theorem \ref{thm:main}.i is
proved in Theorem \ref{thm:connected} and  Proposition  
\ref{cor:rows}. The separation property of Theorem \ref{thm:main}.ii follows from
Proposition~\ref{prop:algorithmdecreasesenergy} and
 Proposition  
\ref{cor:nomissingrows}-\ref{cor:allononeside}. The
geometry of the interface between $A$ and $B$, namely Theorem~\ref{thm:main}.iii,
is described by Corollary \ref{cor:interface}.

 In  Section \ref{sec:info} we present  a collection of examples,  illustrating 
the variety of optimal geometries. In particular, we show that optimal
pairs may be not unique and, in some specific
parameter range, present quite distinguished shapes. We then
classify different admissible pairs in Section \ref{sec:classi} by
introducing five 
distinct classes of configurations. 

The first of  these  classes, called Class $\mathcal I$ and corresponding
to Figure \ref{fig:aspectratio}, is indeed the reference one and is
studied in  detail in  
Section \ref{sec:reg1}.  In   Proposition
\ref{prop:classIregularisationstep2} we prove  the existence of optimal
pairs in Class $\mathcal I$, among which  there is  the explicit one of Theorem
\ref{thm:main}.v. The minimal
perimeter in Theorem~\ref{thm:main}.iv is then computed by referring to this
specific class in Theorem~\ref{thm:classIexact}. The remaining classes are studied in Section \ref{sec:reg2}.  We show that some of the classes cannot be optimal in the case $N_A = N_B$, and  that  the other ones can be modified  to a configuration in Class $\mathcal{I}$ by an explicit regularisation procedure. We also  observe  that for arbitrarily large $N$   solutions   may appear  which are not in Class $\mathcal{I}$, see Proposition~\ref{prop:largeminimisersiv}.  

Although optimal pairs $(A,B)$ are not unique,  by carefully inspecting
our constructions, we are able  to prove
 that,  in   some specific parameter regime,  two optimal pairs differ
by at most  $C_\beta N^{1/2}$ or $C_\beta N^{3/4}$  points,
respectively  depending on the irrationality or rationality of
$\beta$ and  up to isometries. This is studied  in Section~\ref{sec:law}, see
Theorem~\ref{thm:nonehalf}  which  proves Theorem
\ref{thm:main}.vi.  If $\beta $ is irrational, an output of our
construction is that the fluctuation bound  $C_\beta N^{1/2}$  is
sharp. In the case of a rational $\beta$, the sharpness of the
fluctuation bound will be proved  
in some 
  future work.   The $N^{1/2}$-scaling  in fluctuations is
specifically related to the presence of an interface between the two
sets $A$ and $B$. In fact, in case of a single set $A$,  optimal
configurations show fluctuations of  order $N^{3/4}$,  see
Subsection \ref{sec:opt} for details.

 Although  the setting of our paper is discrete, our results deliver some
understanding of the  continuous  case, as well.
This results by
considering the so-called {\it thermodynamic
  limit} as $ N  \to \infty$. For all $ V =\{x_1, \dots, x_{ N }\} \subset {\mathbb Z}^2$,
let $ \mu_V  = (\sum_{i=1}^{ N } \delta_{x_i/\sqrt{ N }})/ N $ be the corresponding
empirical measure on the plane and denote by $\mathcal L$ the
two-dimensional Lebesgue measure. We indicate by 
\begin{align}
  \mathcal
A:=\left(-\sqrt{\frac{2-\beta}{2}},0\right)\times \left(0, \sqrt{\frac{2}{2-\beta}}\right) \ \ \text{and} \ \ \mathcal
B:= \left(0, \sqrt{\frac{2-\beta}{2}} \right)\times \left(0,
  \sqrt{\frac{2}{2-\beta}}\right)\label{eq:wulff}
\end{align}
the continuous Wulff shapes, see Figure \ref{fig:aspectratio}. Note
that $\mathcal L(\mathcal A)=\mathcal L(\mathcal B) = 1$. By
combining the explicit construction of  Theorem \ref{thm:main}.v  and the
fluctuation estimate \eqref{eq:fluct} we have the following.

\begin{corollary}[Wulff shapes]\label{cor: wulff}
Let $\beta\leq 1/2$ and $(A_{ N },B_{ N })$ be
solutions of \eqref{eq:dbp} with $N_{A_{ N }} = N_{B_{ N
    }} =  N $, for all $ N \in \mathbb N$. Then, there exist
isometries $T_{ N }$ of ${\mathbb Z}^2$ such that
\begin{align}
\label{eq:meas}
\mu_{T_{ N } A_{ N }} \stackrel{\ast}{\rightharpoonup}  {\mathcal L} \mres {\mathcal A} \ \ \text{and} \ \ \mu_{T_{ N } B_{ N }} \stackrel{\ast}{\rightharpoonup}  {\mathcal
        L} \mres {\mathcal B},
    \end{align}
   as $N \to \infty$,   where the symbol $\stackrel{\ast}{\rightharpoonup}$  indicates the 
    weak-$\ast$ convergence of measures. 
  \end{corollary}
  Note that, by taking $\beta=1$  in \eqref{eq:wulff}
  (not covered by the corollary, though)  we have
  that $\mathcal A \cup \mathcal B$ form a single square with side
  $\sqrt{2}$ whereas for $\beta=0$ the Wulff shapes $\mathcal A$ and
  $\mathcal B$ are two squares of side $1$.

Our results also allow to solve the double-bubble problem in the
 continuous  setting of
${\mathbb R}^2$ with respect to a crystalline perimeter notion. More
precisely, for every set $D \subset {\mathbb R}^2$ of {\it finite
  perimeter}  we denote by $\partial^* D$ its {\it reduced boundary} \cite{Ambrosio-Fusco-Pallara,Maggi},
and define the {\it crystalline perimeter} and the {\it crystalline
  length} as
$${\rm Per} (D) = \int_{\partial^* D} \| \nu \|_1 \, {\rm d} \mathcal
H^1, \quad {\rm L} (\gamma) = \int_{\gamma} \| \nu \|_1 \, {\rm d} \mathcal
H^1,$$
where $\nu$ is the
outward pointing  unit  normal to  $\partial^*D$, $\|\nu\|_1 = |\nu_x|+|\nu_y|$,  $ \mathcal
H^1$ is the one-dimensional Hausdorff measure, and $\gamma \subset
\partial^*D$ is measurable.

The  continuous  analogue of
\eqref{eq:dbp} is the {\it crystalline double-bubble problem}
\begin{align}
 & \min\Big\{ {\rm Per} (A) + {\rm Per} (B) - 2\beta \, {\rm L} (\partial^*A
  \cap \partial^*B)\, \colon \ \\
  & \qquad \qquad A,\, B \subset {\mathbb R}^2 \ \text{of
   finite perimeter},    \quad  A \cap B =
  \emptyset, \ \mathcal L(A)=\mathcal L(B)=1\Big\}. 
  \label{eq:dbp2}
\end{align}

By combining Theorem \ref{thm:main}.v and \ref{thm:main}.vi we obtain the following.

\begin{corollary}[Crystalline double bubble]\label{cor: cryst-db} For all $\beta \leq 1/2$,
  the pair $(\mathcal A,\mathcal B)$ is a solution of
  \eqref{eq:dbp2}. The minimal energy is given by   $4
  \sqrt{4-2\beta}$. 
\end{corollary}

For the reference choice $\beta=1/2$, the solution of the crystalline
double-bubble problem~\eqref{eq:dbp2} is depicted in Figure
\ref{fig:aspectratio},  see also \cite{Duncan,Morgan98}. 
  Corollaries \ref{cor: wulff}
and \ref{cor: cryst-db} are proved in
Section \ref{sec:wulff}. 

 In the recent \cite{Duncan}, the
difference in energy between any properly rescaled optimal discrete
configuration and the Wulff shape is estimated. In case $N_A=N_B$
and $\beta\leq 1/2$ such an estimate can be recovered from the exact expressions in Theorem
\ref{thm:main}.iv and of Corollary \ref{cor: cryst-db}. Note however that  the
analysis in \cite{Duncan} covers the case
$N_A\not =N_B$ as well, although for $\beta=1/2$ only.

\section{Equivalent formulations of the double-bubble problem}\label{sec:equiv}

\subsection{Optimal particle configurations}\label{sec:opt} The double-bubble problem
\eqref{eq:dbp} can be equivalently recasted in terms of ground states
of  configurations of particles of two different types. 
Let $A =\{ x_1,\dots, x_{N_A} \}$ and  $B=\{ x_{N_A+1}, \dots,
x_{N_A+N_B}\}$ indicate the mutually distinct positions of particles
of two different particle
species and assume that $A,\,B \subset {\mathbb Z}^2$, which in turn
restricts the model to the description of zero-temperature situations.  
To the particle configuration $(A,B)$ we associate the {\it configurational energy} 
\begin{align}\label{eq: basic eneg}
E(A,B) = \frac12 \sum_{i, j=1}^{N_A+ N_B} V_{ \rm sticky}(x_i, x_j),
\end{align}
where  
\begin{equation*}
V_{ \rm sticky}(x_i, x_j) = \threepartdef{-1}{|x_i -  x_j|=
  1  \mbox{ and } x_i,x_j \in A \mbox{ or } x_i,x_j \in
  B,}{-\beta}{|x_i -  x_j| = 1  \mbox{ and }  x_i \in A, x_j \in B \mbox{ or } x_i \in B, x_j \in A,}{0}{|x_i -  x_j| \neq 1.}
\end{equation*}
The interaction density $V_{\rm sticky}(x_i, x_j) $ corresponds to the so-called
{\it sticky} or {\it Heitmann-Radin-type} potential 
\cite{Heitmann} and models the binding energy of
the two particles $x_i$ and $ x_j $. In particular, 
  only first-neighbor interactions contribute  to 
the energy, and {\it intraspecific} (namely, of type $A-A$ or $B-B$)
and {\it interspecific} (type $A-B$) interactions are 
quantified differently, with interspecific interactions being weaker as $\beta <1$.

The relation between the minimisation of $E$ and the double-bubble
problem \eqref{eq:dbp} is revealed by the equality
\begin{equation}
  E(A,B)+2N_A + 2N_B =
\frac12P(A,B).\label{eq:eq}
\end{equation}
This follows by analysing the contribution to $E$ and $P$ of each
point. In fact, one could decompose
$$E(A,B) = \sum_{i=1}^{N_A+N_B} e(x_i), \quad P(A,B) =\sum_{i=1}^{N_A+N_B}
p(x_i),  $$
where the single-point contribution to energy and perimeter is
quantified via
\begin{align*}
  e(x) = -\frac12 \# \{\text{same-species neighbors of $x$}\} - 
  \frac{\beta}{2}  \# \{\text{other-species neighbors of $x$}\}\\
 p(x) =  4 - \# \{\text{same-species neighbors of $x$}\} - 
  \beta  \# \{\text{other-species neighbors of $x$}\}.
 \end{align*}
The latter entail \eqref{eq:eq}, which in turn ensures that ground states of  $E$ and minimisers of
$P$ coincide, for all given sizes $N_A$ and $N_B$ of the sets $A$ and $B$.

The geometry of ground states of $E$ results from the competition
between intraspecific and interspecific interaction. 
In the extremal case $\beta=1$, intra- and interspecific
interaction are indistinguishable, and one can consider the whole
system $(A,B)$ as a single species. The minimisation of $E$ is then
the classical {\it edge-isoperimetric problem} \cite{Bezrukov,Harper}, 
namely the minimisation of $C  \mapsto  Q(C,C^c)$  under
prescribed size $\#C$. 
Ground states are isoperimetric sets, the ground-state energy is
known,    the possible distance between
two ground states scales as  $N^{3/4}$ where $N = \# C$,   and one could even directly prove
crystallization,  i.e., the periodicity of ground states,  under some stronger assumptions on the interaction
potentials~\cite{Mainini}. 

In the other extremal case $\beta=0$, no interspecific interaction is
accounted for, and  both  phases $A$ and $B$ are independent  isoperimetric
sets.  In particular, if $N_A$ and $N_B$ are
perfect squares (or for
$N_A, \, N_B \to \infty$ and up to rescaling), the phases $A$ and $B$
are squares.

 In the  intermediate case $\beta\in(0,1)$, which is hence the interesting
one, intraspecific and
interspecific interaction compete and neither $A$ or $B$ nor $A \cup B$
end up being  isoperimetric
sets.  The presence of interspecific interactions adds some level
of rigidity. This is revealed by the fact, which we prove, that the distance between different ground
states scales like $N^{1/2}$, in contrast with the purely
edge-isoperimetric case, where fluctuations are of order
$N^{3/4}$ \cite{Mainini}, see also
\cite{Cicalese,Davoli,Mainini2,Schmidt}.

Although we do not directly deal with crystallization here, for the
points $A$ and $B$ are {\it assumed} to be subset of the lattice
$\mathbb Z^2$, let us
mention that  a few rigorous crystallization  results   in
multispecies systems are available. At first, existence of quasiperiodic ground states in a
specific multicomponent two-dimensional system has been shown by
Radin \cite{Radin86}. One dimensional crystallization of {\it
  alternating} configurations of two-species has been investigated by  B\'etermin, Kn\"upfer, and Nolte
\cite{Betermin}, see also \cite{periodic} for some related
crystallization and noncyrstallization results. In the
two-dimensional, sticky interaction case, two crystallization results in
hexagonal and square geometries are  given  in \cite{kreutz,
  kreutz2}. Here, however, interspecific interactions favor  the
onset of alternating  phases.  

\subsection{Finite Ising model} The
double-bubble problem \eqref{eq:dbp} can also be equivalently seen as the ground-state problem for a {\it finite} Ising model with ferromagnetic
interactions. In particular, given $C=A \cup B\subset {\mathbb Z}^2$
one describes the state of the system by $u \colon C \to \pm 1$,
distinguishing the $+1$ and the $-1$ phase. The Ising-type energy of the
system is then given by
$$ F  (C,u) = -\frac{1-\beta}{ 4}\sum_{\substack{x,y \in C \\ |x-y|=1}} u(x)\,
u(y) - \frac{1+\beta}{ 4} \sum_{\substack{x,y \in C \\ |x-y|=1}}  
|u(x)\,
u(y))|. 
$$ 
The first term above is the classical ferromagnetic interaction
contribution, while the second sum gives  the total number of
interactions, irrespective of the phase. This second term is required  since in our model 
same-phase and different-phase interactions are both assumed to give negative contributions to the energy.  

Under the above provisions, minimisers of the problem
\begin{align*}
  &\min\Big\{ F  (C,u) \, \colon \, C \subset {\mathbb Z}^2, \ u \colon C \to \pm 1, \
  \\
  &\qquad \qquad \# \{x \in C \, \colon \, u(x)=1\}  = N_A, \  \# \{x
    \in C \, \colon \, u(x)=-1\}  = N_B \Big\}
\end{align*}
corresponds to solutions $(A,B)$ of the double-bubble problem
\eqref{eq:dbp}, under the
equivalence $A\equiv \{x \in C \, \colon \, u(x)=1\}$ and  $B\equiv \{x \in
C \, \colon \, u(x)=-1\}$. In fact, each pair of first neighbors  contributes  $-1$
to $ F $ if  it belongs  to the same phase and $-\beta$ if  it belongs  to
different phases, namely,  
$$ F  (C,u) = E(A,B).$$

The literature on the Ising model is vast and the reader is referred to \cite{Cerf,McCoy} for a
comprehensive collection of results. Ising models are usually
investigated from the point of view of their thermodynamic limit $\#C
\to \infty$ and at  positive  temperature. In particular, models are
usually formulated on the whole lattice or on a large box with constant
boundary states. Correspondingly, the analysis of Wulff shapes is concerned  with the study of a droplet of one phase in a sea of the other
one \cite{Cerf2}. 

Our setting is much 
different, for our system is finite and boundary effects
matter. To the best of our knowledge, we contribute here
the first characterisation of ferromagnetic Ising ground states, where
the location $C$ of the  
system is also unknown and results from minimisation.

Alternatively to the finite two-state setting above, one could
equivalently formulate the minimisation problem in the whole ${\mathbb
  Z}^2$ by allowing a third state, to be interpreted as interaction-neutral. In
particular, we could equivalently consider the minimisation problem 
\begin{align*}
  &\min\Big\{ F  ({\mathbb Z}^2,v) \, \colon \, v \colon {\mathbb
    Z}^2\to\{-1,0,1\}, \\
  &\qquad \qquad \# \{x \in {\mathbb Z}^2 \, \colon \, v(x)=1\}  = N_A, \  \#
    \{x \in {\mathbb Z}^2 \, \colon \, v(x)=-1\}  = N_B \Big\}.
    \end{align*}
The equivalence is of course given by setting $u=v$ on $C:=\{x\in {\mathbb
  Z}^2\, \colon \, v(x)\not = 0 \}$.

\subsection{Finite Heisenberg model} The three-state formulation of
the previous subsection can be easily reconciled within the 
frame of the classical Heisenberg model \cite{Tasaki}. In particular, we shall define
the vector-valued state function $s\colon M \to \{s_{-1}, s_0, s_1\}$ where
the box $M$ is given as $M:=[0,m]^2 \cap {\mathbb Z}^2$ for $m$
large. We choose the
three possible spins as
$$ s_0 =(-1,0), \ \ s_1 =\left(\beta, \sqrt{1-\beta^2}\right), \ \
s_{-1}=\left(\beta, -\sqrt{1-\beta^2}\right).$$ 
The energy of the system is defined as 
$$H(s) =  -\sum_{\substack{x,y \in M \\ |x-y|=1}} s(x)\cdot s(y).$$
For all $s\colon M \to \{s_{-1}, s_0, s_1\}$, let $A:=\{x \in M \, \colon \,
                s(x)=s_1\} $ and $B:=\{x \in M \, \colon \,
                s(x)=s_{-1}\}$. We are interested in the minimisation problem
\begin{align*}&\min\left\{H(s) \, \colon  \,  s\colon M \to \{s_{-1}, s_0, s_1\},  \ \#A=N_A , \ \#B=N_B \right\}.
\end{align*}
By letting $m$ be very large compared with $N_A$ and $N_B$, 
we can with no loss of generality  assume that  $s=s_0$ close to the boundary $\partial M$. 

Let us now show that the latter minimisation problem is indeed
equivalent to the double-bubble problem \eqref{eq:dbp}. To this aim,
we start by noting  that  the total number of first-neighbor interactions in
$M$ is $2m^2+2m$.  First-neighbor  interactions between identical states 
contribute  $-1$  to the energy, $s_0 - s_{1}$ and $s_0 - s_{-1}$
interactions contribute $-s_1\cdot s_0 = -s_{-1}\cdot s_0 = 
\beta$, and $s_1-   s_{-1}$ interactions contribute $-s_1\cdot  s_{-1}  = 
1-2{\beta^2} $. We hence have that 
\begin{align*}
  H(s) + (2m^2 +2m)   &=  (\beta+1) \left( Q (A,
  A^c\setminus B) + Q (B,
                      B^c\setminus A) \right)  +   (2-2\beta^2) Q(A,B) \\
  & = \left(\beta+1 \right) \left( Q (A,
  A^c\setminus B) + Q (B,
                      B^c\setminus A) +(2-2\beta)Q(A,B) \right)
    \nonumber\\
  &=
    \left(\beta+1\right) P(A,B),
\end{align*}
so that minimising $H$ is actually equivalent to solving \eqref{eq:dbp}.

\subsection{Minimum  balanced-separator  problem} One can rephrase the
double-bubble problem \eqref{eq:dbp}  as a minimum  balanced-separator  problem on an unknown
graph  as well.  Indeed, as interspecific contributions
are energetically less favored with respect to 
intraspecific ones, given the common occupancy ${\mathcal V}  =A \cup B$ of the two
phases, one is asked to part ${\mathcal V} $ into two regions $A$ and $B$ with given size in
such a way that the interface between $A$ and $B$ is minimal. This corresponds to a
minimum  balanced-separator  problem on the {\it unit-disk graph}  corresponding to 
${\mathcal V}$,   i.e.,  finding a disjunct
partition ${\mathcal V} = A \cup B$  solving  
$$\min \{ Q(A,B) \, : \, \#A =
N_A, \ \#B = N_B\}.$$
This is indeed a classical problem, with relevant applications in operations
research and computer science \cite{Nagamochi}.

Here, we generalize the above minimum  balanced-separator  problem by letting the underlying
graph also vary and by simultaneously optimising its perimeter.  In
particular, we consider 
$$\min \{P(A,B) \, \colon \, V=A \cup B, \ A \cap B =\emptyset, \ \#A =
N_A, \ \#B = N_B\},$$
where $({\mathcal V}, {\mathcal E})$  is again the unit graph related to $A \cup B\subset
{\mathbb Z}^2$.

Also in this setting, the competition between minimisation of the
interface and of the perimeter is evident. Recall  $P(A,B) =
Q(A,A^c \setminus B) + Q(B,B^c \setminus A) + (2-2\beta)  Q(A,B)$.
On
the one hand, a graph with few edges between $A$ and $B$ would 
 give  a short cut $Q(A,B)$, while necessarily having large
 $Q(A,A^c \setminus B) + Q(B,B^c \setminus A)$. On the other hand, a
 graph  with small   $Q(A,A^c \setminus B) + Q(B,B^c \setminus A)$ has $A \cup
 B$ close to be a square, and  for $N_A = N_B$  all possible cuts partitioning it in two
 are approximately as long as its side.

\section{Notation}

Let us collect here some notation, to be used throughout the
paper. 
For each pair  of    disjoint sets $A,B \subset \mathbb{Z}^2$ 
we  call the elements of $A$ and $B$ the {\it $A$-points} and {\it
  $B$-points}, respectively. We let $N_A = \# A$ and $N_B = \# B$. For
any point $p \in A \cup B$, we denote its first and second coordinate
by $p = (p_x, p_y)$. We say that two points are {\it connected by an
  edge} if their distance is equal to one. (Equivalently, we sometimes
use the words {\it bond} or {\it connection} in place of {\it edge}.)
We say that a set $S \subset A \cup B$ is {\it connected} if it is
connected as a graph with edges described above,  or   equivalently 
if the corresponding unit-disk graph is connected.   

 For the sake of definiteness, from here on, our notation is
adapted to the 
setting of Subsection~\ref{sec:opt}. In particular, we  say that a
configuration is {\it minimal} (or {\it optimal}) if it
minimises the energy $E$  given in \eqref{eq: basic eneg}  in the class of configurations with the same
number of $A$- and $B$-points.  Recall once more that minimisers
of $E$ and solutions of the double-bubble problem \eqref{eq:dbp}
coincide.

Since the number of points is finite, any configuration lies in a
bounded  square.  Suppose that a configuration $(A,B)$ has $N_{\rm
  row}$ rows (i.e., there are $N_{\rm row}$ rows in $\mathbb{Z}^2$
with at least one point from $A \cup B$). For $k=1,...,N_{\rm row}$,
denote by $R_k$ the $k$-th row (counting from the top). In a similar
fashion, $N_{\rm col}$ denotes the number of columns, and $C_k$
indicates the $k$-th column (counting from the left). To simplify the
notation, given a finite set $X \subset \mathbb{Z}^2$, we denote
$X^{\rm row}_k = X \cap R_k$ and $X_k^{\rm col} = X \cap C_k$. We will
typically apply this to the sets $A$, $B$, their union or some of
their subsets. Moreover, denote by $n_k^{\rm row}$ the number of
$A$-points in the row $R_k$ and by $m^{\rm row}_k$ the number of
$B$-points in the row $R_k$. In a similar fashion, $n_k^{\rm col}$ and
$m_k^{\rm col}$ denote the number of $A$- and  $B$-points in column
$C_k$, respectively. In the following, we will frequently modify
configurations. Not to overburden the notation, when we use the
notation  $n^{\rm row}_k$ and $m_k^{\rm row}$ (and similarly for
columns) we always refer to the configuration in the same sentence,
unless  otherwise  specified.

For two points $p,q \in A \cup B$, we say that $p$ {\it lies to the
  left} (respectively {\it right}) of $q$ if $p_y = q_y$ and $p_x <
q_x$ (respectively $p_x > q_x$). In other words, they are in the same
row, and the first coordinate of $p$ is smaller (respectively larger)
than the first coordinate of $q$. We say that $p$ lies directly to the
left (respectively right) of $q$ if additionally $p$ and $q$ are
connected by an edge. Similarly, we say that $p$ {\it lies above}
(respectively {\it below}) $q$ if $p_x = q_x$ and $p_y > q_y$
(respectively $p_y < q_y$). Again, we say that $p$  lies {\it directly} above (respectively below) $q$ if additionally these two points are connected by an edge.

 We   will also say that the set $A^{\rm
  row}_k$ {\it lies to the left} (respectively {\it right}) of $B^{\rm
  row}_k$ if for every $p \in A^{\rm row}_k$ and $q \in B^{\rm row}_k$
the point $p$ lies to the left (respectively right) of $q$. (Note that
by definition $A^{\rm row}_k$ and $B^{\rm row}_k$ are in the same
row.) We also say that $A^{\rm row}_k$ lies {\it directly} to the left
of $B^{\rm row}_k$ if additionally there is a connection between one
of the points in $A^{\rm row}_k$ and one of the points in $B^{\rm
  row}_k$. An analogous  notion  is used for columns.

Furthermore, we say that a number of points from different rows are
{\it aligned} if their first coordinates are equal. We also say that
two sets are {\it aligned to the right} (or {\it left}) if their rightmost (leftmost) points are aligned. The same  notion  is also used for columns.

Finally, given a finite set $X \subset \mathbb{Z}^2$, we denote by $X + (a,b)$ the set consisting of all points of $X$ shifted by the vector $(a,b) \in \mathbb{Z}^2$.

\section{Connectedness,  separation, and interface}\label{sec:algo}

In this section, we  introduce a procedure in order to modify  an
arbitrary configuration $(A,B)$ into  another  configuration $(\hat{A},
\hat{B})$ with  specific additional properties, without increasing
the energy.   In particular,  this will prove   that for a
minimal configuration the sets $A$, $B$,   and $A \cup B$   are connected.

\subsection{Description of the  procedure}

The goal of this subsection is to present  a procedure allowing to
modify a configuration, making it  more regular in the following sense: not only the sets $A$ and $B$ are connected, but also for any $k = 1,...,N_{\rm row}$ and any $l = 1,...,N_{\rm col}$ the sets $A^{\rm row}_k$, $B^{\rm row}_k$, $(A \cup B)_k^{\rm row}$,  $A^{\rm col}_l$, $B^{\rm col}_l$, and $(A \cup B)_l^{\rm col}$  are connected. We start with the following preliminary result.

\begin{proposition}
Let $(A,B)$ be a configuration in the sense described above. If there are any empty rows (or columns) between any two rows (or columns) in $(A,B)$, then there exists a configuration $(\hat{A},\hat{B})$ with strictly smaller energy.
\end{proposition}

\begin{proof}
Without restriction we present the argument for rows. Suppose that between rows $R_k$ and $R_{k+1}$ for some $k \in \{ 1,...,N_{\rm row} - 1 \}$ there are $l$ empty rows. Then, we can reduce  the  energy in the following way: denote by $(A',B')$ the configuration consisting of the top $k$ rows and by $(A'',B'')$ the configuration consisting of the bottom $N_{\rm row}-k$ rows. Then, we remove the empty rows, i.e., replace $(A'',B'')$ with $(A'',B'') + (0,l)$. Clearly, this does not increase the energy of the configuration $(A,B)$. If after this shift there is at least one connection between $A^{\rm row}_k \cup B^{\rm row}_k$ and $A^{\rm row}_{k+1} \cup B^{\rm row}_{k+1}$, the energy even decreases by at least $\beta$. Otherwise, if after this shift there are no connections between $A^{\rm row}_k \cup B^{\rm row}_k$ and $A^{\rm row}_{k+1} \cup B^{\rm row}_{k+1}$, we shift the configuration $(A',B')$ horizontally to make at least one connection. Again, the energy is decreased by at least $\beta$.
\end{proof}

Hence, in studying minimal configurations, we may assume that there
are no empty rows and columns. Now, we are ready to describe  a
modification procedure making  the configuration more regular. Notice that we may write the energy in the following way:
\begin{equation*}
E(A,B) = \sum_{k=1}^{N_{\rm row}} E^{\rm row}_k(A,B) + \sum_{k=1}^{N_{\rm row} -1} E_k^{\rm  inter }(A,B).      
\end{equation*}
Here, $E^{\rm row}_k(A,B)$ is the part of the energy given by interactions in the row $R_k$, namely 
\begin{equation}\label{eq: row energy}
E^{\rm row}_k(A,B) = \frac{1}{2} \sum_{ x_i,x_j  \in A^{\rm row}_k \cup B^{\rm row}_k} V_{ \rm sticky}(x_i, x_j),
\end{equation}
and $E_k^{\rm  inter }(A,B)$ is the part of the energy given by interactions between rows $R_k$ and $R_{k+1}$, namely
 \begin{equation*}
E_k^{\rm  inter }(A,B) = \sum_{ x_i  \in A^{\rm row}_k \cup B^{\rm row}_k, \,  x_j  \in A^{\rm row}_{k+1} \cup B^{\rm row}_{k+1}} V_{ \rm sticky }(x_i, x_j).
\end{equation*}
Now, let us see that we may bound $E^{\rm row}_k$ and $E_k^{\rm  inter }$ by expressions depending on $n^{\rm row}_k$ and $m^{\rm row}_k$. First, we estimate $E^{\rm row}_k$.

\begin{lemma}\label{lem:horizontal}
We have
\begin{equation*}
E^{\rm row}_k(A,B)  \geq \begin{cases} - (n^{\rm row}_k + m^{\rm row}_k) +2- \beta  & \text{if $n^{\rm row}_k > 0$, $m^{\rm row}_k > 0$,} \\
  -  (n^{\rm row}_k + m^{\rm row}_k) +1 & \text{else.}
  \end{cases} 
\end{equation*}
Moreover, this inequality is an equality if and only if the sets $A^{\rm row}_k$, $B^{\rm row}_k$, and $A^{\rm row}_k \cup B^{\rm row}_k$ are connected.
\end{lemma}

 This result is illustrated  in  Figure \ref{fig:singlerow}; assuming that both $A_k^{\rm row}$ and $B_k^{\rm row}$ consist of three points, we present three possible configurations. The configuration on top is optimal and is exactly of the form given in the statement of the lemma, while the other two configurations do not have the optimal energy. 

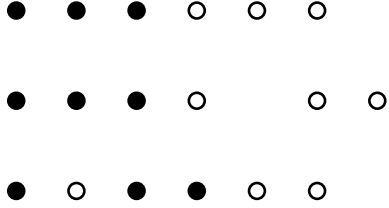
\begin{figure}[h!]

\begin{tikzpicture}[line cap=round,line join=round,>=triangle 45,x=0.8cm,y=0.6cm]
\clip(-3.5,-4.5) rectangle (3.5,0.5);
\begin{scriptsize}
\draw[line width=1pt,color=black] (0.,0.) circle (3pt);
\draw[line width=1pt,fill=black] (-3.,0.) circle (3pt);
\draw[line width=1pt,fill=black] (-2.,0.) circle (3pt);
\draw[line width=1pt,color=black] (1.,0.) circle (3pt);
\draw[line width=1pt,fill=black] (-1.,0.) circle (3pt);
\draw[line width=1pt,color=black] (3.,-2.) circle (3pt);
\draw[line width=1pt,color=black] (0.,-2.) circle (3pt);
\draw[line width=1pt,fill=black] (-1.,-2.) circle (3pt);
\draw[line width=1pt,fill=black] (-3.,-2.) circle (3pt);
\draw[line width=1pt,color=black] (2.,0.) circle (3pt);
\draw[line width=1pt,color=black] (2.,-2.) circle (3pt);
\draw[line width=1pt,fill=black] (-2.,-2.) circle (3pt);
\draw[line width=1pt,fill=black] (-3.,-4.) circle (3pt);
\draw[line width=1pt,fill=black] (-1.,-4.) circle (3pt);
\draw[line width=1pt,fill=black] (0.,-4.) circle (3pt);
\draw[line width=1pt,color=black] (-2.,-4.) circle (3pt);
\draw[line width=1pt,color=black] (1.,-4.) circle (3pt);
\draw[line width=1pt,color=black] (2.,-4.) circle (3pt);
\end{scriptsize}
\end{tikzpicture}

\caption{Different configurations inside a single row}
\label{fig:singlerow}
\end{figure}


\begin{proof}
We consider two cases. In the first case, we suppose that $m^{\rm row}_k = 0$ (a similar argument works if $n^{\rm row}_k = 0$): then, the desired inequality takes the form $E^{\rm row}_k(A,B) \geq -n^{\rm row}_k + 1$. Since $A^{\rm row}_k$ is a subset of a single row, $n^{\rm row}_k - 1$ is the maximum number of connections between points in $A^{\rm row}_k$ and it is achieved only if $A^{\rm row}_k$ is connected.

In the second case, we have $n^{\rm row}_k > 0$ and $m^{\rm row}_k > 0$. Since $A^{\rm row}_k \cup B^{\rm row}_k$ is a subset of a single row, the maximum number of connections (regardless of their type) is $n^{\rm row}_k + m^{\rm row}_k - 1$. It is achieved only if $(A \cup B)_k^{\rm row}$ is connected. Among these, at most $n^{\rm row}_k - 1$ are connections between points in $A^{\rm row}_k$ and at most $m^{\rm row}_k - 1$ are connections between points in $B^{\rm row}_k$. These numbers are achieved if and only if $A^{\rm row}_k$ and $B^{\rm row}_k$ are connected. Each of these connections contributes $-1$ to the energy and there can be at most $n^{\rm row}_k + m^{\rm row}_k - 2$ of them.  The remaining connections are between $A^{\rm row}_k$ and $B^{\rm row}_k$ contributing $-\beta$ to the energy. The fact that $\beta < 1$ yields the statement. 
\end{proof}

Now, we make a similar computation for $E_k^{\rm  inter }$.

\begin{lemma}\label{lem:vertical}
We have
\begin{align*}
E_k^{\rm  inter }(A,B)  \geq &-(1-\beta)\big(\min\lbrace n^{\rm row}_k,
                         n^{\rm row}_{k+1}\rbrace +\min\lbrace m^{\rm
                         row}_k, m^{\rm row}_{k+1}\rbrace \big) \\
  &-   \beta \min\lbrace n^{\rm row}_k + m^{\rm row}_k, n^{\rm row}_{k+1} + m^{\rm row}_{k+1}\rbrace. 
\end{align*}
Moreover,  equality is achieved if and only if the following conditions hold: \\
(1) There are $\min\lbrace n^{\rm row}_k,n^{\rm row}_{k+1}\rbrace$  points in $A^{\rm row}_k$ directly above points in $A^{\rm row}_{k+1}$; \\
(2) There are $\min\lbrace m^{\rm row}_k,m^{\rm row}_{k+1}\rbrace$ points in $B^{\rm row}_k$ directly above points in $B^{\rm row}_{k+1}$; \\
(3) Supposing that $n^{\rm row}_k + m^{\rm row}_k \geq n^{\rm row}_{k+1} + m^{\rm row}_{k+1}$, there is a point in $A^{\rm row}_k \cup B^{\rm row}_k$ directly above every point in $A^{\rm row}_{k+1} \cup B^{\rm row}_{k+1}$. Otherwise, if $n^{\rm row}_k + m^{\rm row}_k < n^{\rm row}_{k+1} + m^{\rm row}_{k+1}$, there is a point in $A^{\rm row}_{k+1} \cup B^{\rm row}_{k+1}$ directly below every point in $A^{\rm row}_k \cup B^{\rm row}_k$.
\end{lemma}

 This result is illustrated  in  Figure \ref{fig:doublerow}. We present three configurations consisting of two rows; each of them has the form prescribed by Lemma \ref{lem:horizontal} inside both rows, but the alignment of the two rows is different. Only the top configuration is optimal. 

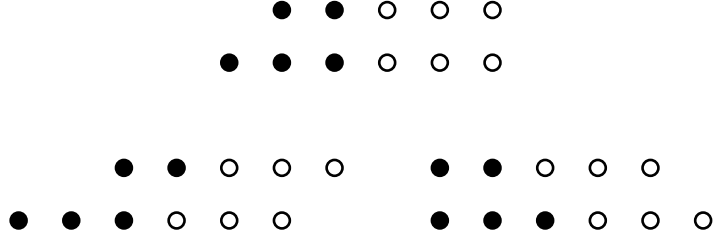
\begin{figure}[h!]

\begin{tikzpicture}[line cap=round,line join=round,>=triangle 45,x=0.7cm,y=0.7cm]
\clip(-7.3,-4.5) rectangle (6.3,0.5);
\begin{scriptsize}
\draw[line width=1pt,color=black] (0.,0.) circle (3pt);
\draw[line width=1pt,fill=black] (-2.,0.) circle (3pt);
\draw[line width=1pt,color=black] (1.,0.) circle (3pt);
\draw[line width=1pt,fill=black] (-1.,0.) circle (3pt);
\draw[line width=1pt,color=black] (2.,-1.) circle (3pt);
\draw[line width=1pt,color=black] (0.,-1.) circle (3pt);
\draw[line width=1pt,fill=black] (-1.,-1.) circle (3pt);
\draw[line width=1pt,color=black] (2.,0.) circle (3pt);
\draw[line width=1pt,color=black] (1.,-1.) circle (3pt);
\draw[line width=1pt,fill=black] (-2.,-1.) circle (3pt);
\draw[line width=1pt,fill=black] (1.,-3.) circle (3pt);
\draw[line width=1pt,fill=black] (-3.,-1.) circle (3pt);
\draw[line width=1pt,fill=black] (2.,-3.) circle (3pt);
\draw[line width=1pt,color=black] (5.,-3.) circle (3pt);
\draw[line width=1pt,color=black] (4.,-3.) circle (3pt);
\draw[line width=1pt,color=black] (3.,-3.) circle (3pt);
\draw[line width=1pt,fill=black] (1.,-4.) circle (3pt);
\draw[line width=1pt,fill=black] (2.,-4.) circle (3pt);
\draw[line width=1pt,fill=black] (3.,-4.) circle (3pt);
\draw[line width=1pt,color=black] (4.,-4.) circle (3pt);
\draw[line width=1pt,color=black] (5.,-4.) circle (3pt);
\draw[line width=1pt,color=black] (6.,-4.) circle (3pt);
\draw[line width=1pt,fill=black] (-5.,-3.) circle (3pt);
\draw[line width=1pt,fill=black] (-4.,-3.) circle (3pt);
\draw[line width=1pt,color=black] (-3.,-3.) circle (3pt);
\draw[line width=1pt,color=black] (-2.,-3.) circle (3pt);
\draw[line width=1pt,color=black] (-1.,-3.) circle (3pt);
\draw[line width=1pt,color=black] (-2.,-4.) circle (3pt);
\draw[line width=1pt,color=black] (-3.,-4.) circle (3pt);
\draw[line width=1pt,color=black] (-4.,-4.) circle (3pt);
\draw[line width=1pt,fill=black] (-5.,-4.) circle (3pt);
\draw[line width=1pt,fill=black] (-6.,-4.) circle (3pt);
\draw[line width=1pt,fill=black] (-7.,-4.) circle (3pt);
\end{scriptsize}
\end{tikzpicture}

\caption{Different alignments of adjacent rows}
\label{fig:doublerow}
\end{figure}

\begin{proof}
First, as there are   $n^{\rm row}_k + m^{\rm row}_k$ points in $A^{\rm row}_k \cup B^{\rm row}_k$ and $n^{\rm row}_{k+1} + m^{\rm row}_{k+1}$ points in $A^{\rm row}_{k+1} \cup B^{\rm row}_{k+1}$,  there are at most $\min\lbrace n^{\rm row}_k + m^{\rm row}_k, n^{\rm row}_{k+1} + m^{\rm row}_{k+1}\rbrace$ connections between points in $A^{\rm row}_k \cup B^{\rm row}_k$ and $A^{\rm row}_{k+1} \cup B^{\rm row}_{k+1}$, regardless of their type. Among these, we denote the number of connections between points in $A^{\rm row}_k$ and $A^{\rm row}_{k+1}$ by $\tilde{n}_k$ and the number of  connections between points in $B^{\rm row}_k$ and $B^{\rm row}_{k+1}$ by $\tilde{m}_k$. We have $\tilde{n}_k \le \min\lbrace n^{\rm row}_k, n^{\rm row}_{k+1}\rbrace$ and $\tilde{m}_k \le \min\lbrace m^{\rm row}_k, m^{\rm row}_{k+1}\rbrace$ with equality if this many points in $A^{\rm row}_{k+1}$ are placed directly under points in $A^{\rm row}_k$ (and similarly for $B^{\rm row}_k$ and $B^{\rm row}_{k+1}$). Each of these connections contributes $-1$ to the energy, i.e., a total contribution of $-\tilde{n}_k - \tilde{m}_k$. Then, there are at most  $\min\lbrace n^{\rm row}_k + m^{\rm row}_k, n^{\rm row}_{k+1} + m^{\rm row}_{k+1}\rbrace - (\tilde{n}_k + \tilde{m}_k)$ possible connections which need to be either connections between points in $A^{\rm row}_k$ and $B^{\rm row}_{k+1}$ or between points in $B^{\rm row}_k$ and $A^{\rm row}_{k+1}$. Either way, each of these connections contributes $-\beta$ to the energy. In conclusion, we obtain the desired inequality, with equality only if $\tilde{n}_k = \min\lbrace n^{\rm row}_k, n^{\rm row}_{k+1}\rbrace$, $\tilde{m}_k = \min\lbrace m^{\rm row}_k, m^{\rm row}_{k+1}\rbrace$, and if there are $\min\lbrace n^{\rm row}_k + m^{\rm row}_k, n^{\rm row}_{k+1} + m^{\rm row}_{k+1}\rbrace$ connections between  $A^{\rm row}_k \cup B^{\rm row}_k$ and $A^{\rm row}_{k+1} \cup B^{\rm row}_{k+1}$.
\end{proof}

In light of these estimates, we  describe a simple 
modification procedure making  any configuration more regular. For any configuration $(A,B)$, we construct a configuration $(\hat{A},\hat{B})$ having the same number of $A$- and $B$-points in each row as $(A,B)$ such that the energy is lower or equal and $(\hat{A},\hat{B})$ has some additional structure properties.

\textit{Step 0:} We start with the first row from the top. We let $\hat{A}_1$ be a connected set in a single row consisting of  $n^{\rm row}_1$  atoms and let $\hat{B}_1$ be the connected set in the same row with  $m^{\rm row}_1$   points right of $\hat{A}_1$, in such a way that there is a connection between $\hat{A}_1$ and $\hat{B}_1$. By Lemma~\ref{lem:horizontal}, we have $ E^{\rm row}_1 (\hat{A},\hat{B}) \leq  E^{\rm row}_1 (A,B)$.

\textit{Step $k$} (for $k = 1,..., N_{\rm row} -1$): We suppose that the sets in the previous steps have been constructed in such a way that $\hat{A}_{k}$, $\hat{B}_k$, and $\hat{A}_k \cup \hat{B}_k$ are connected, and $\hat{A}_k$ lies on the left of $\hat{B}_k$. We will now define $\hat{A}_{k+1}$ and $\hat{B}_{k+1}$. To this end, we distinguish four cases.

\textit{Case 1:} $n^{\rm row}_k \leq n^{\rm row}_{k+1}$ and $m^{\rm row}_k \leq m^{\rm row}_{k+1}$. We place $n^{\rm row}_k$ points of $\hat{A}_{k+1}$ directly below $\hat{A}_k$. Then, we put the remaining $n^{\rm row}_{k+1} - n^{\rm row}_k$ points to the left of the previously placed points, so that $\hat{A}_{k+1}$ is connected. Similarly, we place $m^{\rm row}_k$ points from $\hat{B}_{k+1}$ directly below $\hat{B}_k$ and the remaining $m^{\rm row}_{k+1} - m^{\rm row}_k$ points to the right of the previously placed points, so that $\hat{B}_{k+1}$ is connected. By Lemma \ref{lem:horizontal}, we have $E^{\rm row}_{k+1}(\hat{A},\hat{B}) \leq E^{\rm row}_{k+1}(A,B)$, and by Lemma \ref{lem:vertical}, we have $E_k^{\rm  inter }(\hat{A},\hat{B}) \leq E_k^{\rm  inter }(A,B)$.  The situation is presented in Figure \ref{fig:procedure} (the top configuration). 

\begin{figure}[h!]

\begin{tikzpicture}[line cap=round,line join=round,>=triangle 45,x=0.8cm,y=0.8cm]
\clip(-6.5,-4.5) rectangle (5.5,1.5);
\begin{scriptsize}
\draw[line width=1pt,color=black] (0.,0.) circle (3pt);
\draw[line width=1pt,fill=black] (-2.,1.) circle (3pt);
\draw[line width=1pt,color=black] (1.,1.) circle (3pt);
\draw[line width=1pt,fill=black] (-1.,1.) circle (3pt);
\draw[line width=1pt,color=black] (1.,0.) circle (3pt);
\draw[line width=1pt,color=black] (0.,1.) circle (3pt);
\draw[line width=1pt,fill=black] (-3.,0.) circle (3pt);
\draw[line width=1pt,color=black] (3.,0.) circle (3pt);
\draw[line width=1pt,color=black] (2.,0.) circle (3pt);
\draw[line width=1pt,fill=black] (-1.,0.) circle (3pt);
\draw[line width=1pt,fill=black] (1.,-3.) circle (3pt);
\draw[line width=1pt,fill=black] (-2.,0.) circle (3pt);
\draw[line width=1pt,fill=black] (2.,-3.) circle (3pt);
\draw[line width=1pt,color=black] (5.,-3.) circle (3pt);
\draw[line width=1pt,color=black] (4.,-3.) circle (3pt);
\draw[line width=1pt,color=black] (3.,-3.) circle (3pt);
\draw[line width=1pt,fill=black] (0.,-4.) circle (3pt);
\draw[line width=1pt,fill=black] (1.,-4.) circle (3pt);
\draw[line width=1pt,fill=black] (2.,-4.) circle (3pt);
\draw[line width=1pt,color=black] (4.,-4.) circle (3pt);
\draw[line width=1pt,color=black] (5.,-4.) circle (3pt);
\draw[line width=1pt,fill=black] (-6.,-3.) circle (3pt);
\draw[line width=1pt,fill=black] (-5.,-3.) circle (3pt);
\draw[line width=1pt,color=black] (-4.,-3.) circle (3pt);
\draw[line width=1pt,color=black] (-3.,-3.) circle (3pt);
\draw[line width=1pt,color=black] (-2.,-3.) circle (3pt);
\draw[line width=1pt,color=black] (-3.,-4.) circle (3pt);
\draw[line width=1pt,fill=black] (-4.,-4.) circle (3pt);
\draw[line width=1pt,fill=black] (-5.,-4.) circle (3pt);
\draw[line width=1pt,fill=black] (-6.,-4.) circle (3pt);
\draw[line width=1pt,fill=black] (3.,-4.) circle (3pt);
\end{scriptsize}
\end{tikzpicture}

\caption{Different cases of the regularisation procedure}
\label{fig:procedure}
\end{figure}
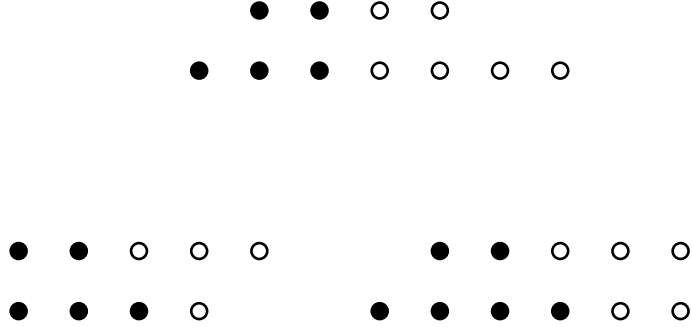

\textit{Case 2:} $n^{\rm row}_k > n^{\rm row}_{k+1}$ and $m^{\rm row}_k > m^{\rm row}_{k+1}$. We place all the points of $\hat{A}_{k+1}$ directly below $\hat{A}_k$, starting from the right. Then, we place all the points of $\hat{B}_{k+1}$ directly below $\hat{B}_k$, starting from the left. In this way, the sets $\hat{A}_{k+1}$, $\hat{B}_{k+1}$ and $\hat{A}_{k+1} \cup \hat{B}_{k+1}$ are connected. Again, by Lemma \ref{lem:horizontal} we have $E^{\rm row}_{k+1}(\hat{A},\hat{B}) \leq E^{\rm row}_{k+1}(A,B)$ and by Lemma \ref{lem:vertical} we have $E_k^{\rm  inter }(\hat{A},\hat{B}) \leq E_k^{\rm  inter }(A,B)$.  The situation (after exchanging the roles of the two rows) is presented in the top configuration in Figure~\ref{fig:procedure}. 

\textit{Case 3:} $n^{\rm row}_k \le  n^{\rm row}_{k+1}$ and $m^{\rm row}_k > m^{\rm row}_{k+1}$. First, we put $n^{\rm row}_k$ points of $\hat{A}_{k+1}$ directly below $\hat{A}_k$. Then, we consider two possibilities:

- If $n^{\rm row}_k + m^{\rm row}_k \geq n^{\rm row}_{k+1} + m^{\rm row}_{k+1}$, we place the remaining $n^{\rm row}_{k+1} - n^{\rm row}_k$ points of $\hat{A}_{k+1}$ under $\hat{B}_k$, starting from the left so that $\hat{A}_{k+1}$ is connected. Then, we place the $m^{\rm row}_{k+1}$ points of $\hat{B}_{k+1}$ to the right of the previously placed points, so that $\hat{B}_{k+1}$ and $\hat{A}_{k+1} \cup \hat{B}_{k+1}$ are connected.  The situation is presented in Figure~\ref{fig:procedure} (the left configuration). 

- If $n^{\rm row}_k + m^{\rm row}_k < n^{\rm row}_{k+1} + m^{\rm row}_{k+1}$, we place the $m^{\rm row}_{k+1}$ points of $\hat{B}_{k+1}$ below points in $\hat{B}_k$, starting from the right, so that $\hat{B}_{k+1}$ is connected. Then, we place $m^{\rm row}_k - m^{\rm row}_{k+1}$ points of $\hat{A}_{k+1}$ between the two sets of previously placed points. Finally, we place the remaining points of $\hat{A}_{k+1}$ to the left of all points placed so far, so that $\hat{A}_{k+1} \cup \hat{B}_{k+1}$ is connected.  The situation is presented in Figure~\ref{fig:procedure} (the right configuration). 

In both cases, by Lemma~\ref{lem:horizontal} we have $E^{\rm row}_{k+1}(\hat{A},\hat{B}) \leq E^{\rm row}_{k+1}(A,B)$ and by Lemma~\ref{lem:vertical} we get $E_k^{\rm  inter }(\hat{A},\hat{B}) \leq E_k^{\rm  inter }(A,B)$.

\textit{Case 4:} $n^{\rm row}_k > n^{\rm row}_{k+1}$ and $m^{\rm row}_k \le m^{\rm row}_{k+1}$. We proceed as in Case 3 with the roles of $A$ and $B$ interchanged, with 'left' and 'right' also interchanged. Again, by Lemma \ref{lem:horizontal} we have $E^{\rm row}_{k+1}(\hat{A},\hat{B}) \leq E^{\rm row}_{k+1}(A,B)$ and by Lemma \ref{lem:vertical} we have $E_k^{\rm  inter }(\hat{A},\hat{B}) \leq E_k^{\rm  inter }(A,B)$.  The situation is presented in Figure \ref{fig:procedure} (the two bottom configuration) after exchanging the roles of the two colors.

\begin{proposition}\label{prop:algorithmdecreasesenergy}
The  procedure  described above modifies a configuration $(A,B)$ into a configuration $(\hat{A},\hat{B})$ with $E(\hat{A},\hat{B}) \leq E(A,B)$. Moreover, if one of the sets $A^{\rm row}_k$, $B^{\rm row}_k$, or $(A \cup B)_k^{\rm row}$, for $k=1,\ldots,N_{\rm row}$, is not connected, or one of the properties {\rm (1)--(3)} in Lemma \ref{lem:vertical} is violated,  then $E(\hat{A},\hat{B}) < E(A,B)$.
\end{proposition}

\begin{proof}
The  construction  ensures that the configuration $(\hat{A},\hat{B})$ has the same number of rows as $(A,B)$. Hence, we compute
\begin{align*}
E(\hat{A},\hat{B}) & = \sum_{k=1}^{N_{\rm row} } E^{\rm row}_k(\hat{A},\hat{B}) + \sum_{k=1}^{N_{\rm row}-1} E_k^{\rm  inter }(\hat{A},\hat{B}) \leq \sum_{k=1}^{N_{\rm row}} E^{\rm row}_k(A,B) + \sum_{k=1}^{N_{\rm row}-1} E_k^{\rm  inter }(A,B) \\ & = E(A,B).
\end{align*}
In view of Lemma \ref{lem:horizontal}, we obtain strict inequality if one of the sets $A^{\rm row}_k$, $B^{\rm row}_k$ or $(A \cup B)_k^{\rm row}$ is not connected. In a similar fashion, we get  strict inequality whenever one of the properties (1)--(3) in Lemma \ref{lem:vertical} does not hold.
\end{proof}

In particular, for optimal configurations $(A,B)$, all sets $A^{\rm
  row}_k$, $B^{\rm row}_k$, and $(A \cup B)_k^{\rm row}$ are
connected. In other words, inside any row we have first all points of
one type and then all points of the other type without any gaps in
between. Moreover, we may make use of this  procedure  (and prove an analogue of  Lemma \ref{lem:horizontal}--Proposition \ref{prop:algorithmdecreasesenergy})  for columns in place of rows. Hence, the sets $A^{\rm col}_k$, $B^{\rm col}_k$, and $(A \cup B)_k^{\rm col}$ are connected. In other words, given an optimal configuration, in each column there are first all points of one type and then all points of the other type without any gaps in between. In particular, as a consequence, we get an important property of any minimising configuration.

\begin{theorem}\label{thm:connected}
Suppose that $(A,B)$ is an optimal configuration. Then $A$ and $B$ are connected.
\end{theorem}

\begin{proof}
Suppose by contradiction that $A$ is not connected (we proceed similarly for $B$). First of all, let us notice that for each $k = 1,...,N_{\rm row}$ the set $A^{\rm row}_k$ is connected. Otherwise, by Proposition \ref{prop:algorithmdecreasesenergy} we find that $(A,B)$ was not an optimal configuration.

 Let us first suppose  that $n^{\rm row}_k > 0$ for all $k \in 1,...,N_{\rm row}$ (i.e., $A^{\rm row}_k \neq \emptyset$). Since every $A^{\rm row}_k$ is connected, if $A$ is not connected, it means that there is no connection between $A^{\rm row}_k$ and $A^{\rm row}_{k+1}$ for some choice of $k$. In this case, by Lemma~\ref{lem:vertical} and by Proposition \ref{prop:algorithmdecreasesenergy}  we find that $(A,B)$ was not an optimal configuration. 

Hence, the only remaining possibility that $A$ is not connected is
that there exist $k_1 < k_2 < k_3$ such that $n^{\rm row}_{k_1},
n^{\rm row}_{k_3} > 0$ and $n^{\rm row}_{k_2} = 0$ (i.e., $A^{\rm
  row}_{k_1}, A^{\rm row}_{k_3} \neq \emptyset$ and $A^{\rm row}_{k_2}
= \emptyset$). Without loss of generality, we may require that for
every $k = k_1 + 1,...,k_3 - 1$ the set $A^{\rm row}_{k}$ is
empty. Let us apply the reorganisation $(A,B) \rightarrow
(\hat{A},\hat{B})$ using the  procedure  described above. Clearly, $(\hat{A},\hat{B})$ is still optimal by Proposition \ref{prop:algorithmdecreasesenergy}.  Then, for $k = k_1$ we are either in Case 2 or in Case 4 of the  procedure.  We distinguish these two cases.

In the first one, suppose that for $k = k_1$ Case~2 of the  procedure
applies. Then, the leftmost point of $\hat{B}_{k_1+1}$ lies directly
below the leftmost point of $\hat{B}_{k_1}$. Then, since for every $k
= k_1 + 1,...,k_3 - 1$ the set $A^{\rm row}_{k}$ is empty, Case 1 or 3
of the  procedure  shows that also the leftmost point of $\hat{B}_k$
lies below the leftmost point of $\hat{B}_{k_1+1}$ (hence below the
leftmost point of $\hat{B}_{k_1}$). Now, for $k = k_3 - 1$, when we
place the sets $\hat{A}_{k_3}$ and $\hat{B}_{k_3}$, we either fall
into Case 1 or Case 3 in the description of the  procedure.  In Case
1, the leftmost point of $\hat{B}_{k_3}$ is again placed below the
leftmost point of $\hat{B}_{k_1}$. Then, the rightmost point of
$\hat{A}_{k_3}$ is placed below the rightmost point of
$\hat{A}_{k_1}$.  Now, one reaches a contradiction by following
the same construction of
Proposition~\ref{prop:algorithmdecreasesenergy} by exchanging the role
of rows and columns.  
In Case 3,  we either have that a point of $\hat{A}_{k_3}$ is placed below a point of $\hat{A}_{k_1}$, which as above is a contradiction to Proposition~\ref{prop:algorithmdecreasesenergy}, or  
the leftmost point of $\hat{A}_{k_3}$ is placed below the leftmost point of $\hat{B}_{k_1}$.  In particular, the leftmost point of $\hat{A}_{k_3}$ is placed   one point to the right of the rightmost point of $\hat{A}_{k_1}$.  Then, by  Lemma~\ref{lem:vertical}(1)  and  Proposition~\ref{prop:algorithmdecreasesenergy}  for columns in place of rows  we again see that the energy of $(\hat{A},\hat{B})$  was not minimal, a contradiction.  The situation is presented in the top line of Figure \ref{fig:threerows} in a simplified setting with $k_1 = 1$ and $k_3 = 3$. 

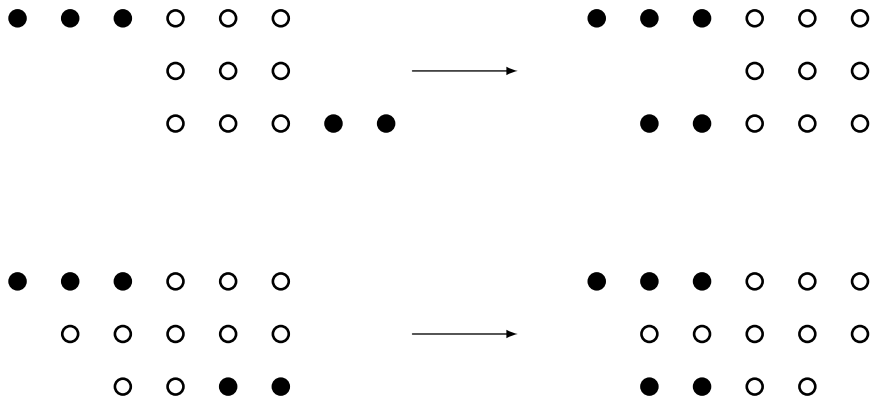
\begin{figure}[h!]

\begin{tikzpicture}[line cap=round,line join=round,>=triangle 45,x=0.7cm,y=0.7cm]
\clip(-8.3,-5.5) rectangle (8.3,2.5);
\draw[-latex,line width=0.5pt] (-0.5,1.) -- (1.5,1.);
\draw[-latex,line width=0.5pt] (-0.5,-4.) -- (1.5,-4.);
\begin{scriptsize}
\draw[line width=1pt,fill=black] (-6.,2.) circle (3pt);
\draw[line width=1pt,color=black] (-4.,2.) circle (3pt);
\draw[line width=1pt,fill=black] (-8.,2.) circle (3pt);
\draw[line width=1pt,color=black] (-3.,1.) circle (3pt);
\draw[line width=1pt,color=black] (-4.,1.) circle (3pt);
\draw[line width=1pt,fill=black] (-1.,0.) circle (3pt);
\draw[line width=1pt,color=black] (-3.,2.) circle (3pt);
\draw[line width=1pt,color=black] (-5.,2.) circle (3pt);
\draw[line width=1pt,fill=black] (-2.,0.) circle (3pt);
\draw[line width=1pt,fill=black] (-3.,-5.) circle (3pt);
\draw[line width=1pt,fill=black] (-4.,-5.) circle (3pt);
\draw[line width=1pt,color=black] (-6.,-5.) circle (3pt);
\draw[line width=1pt,color=black] (-3.,-4.) circle (3pt);
\draw[line width=1pt,color=black] (-7.,-4.) circle (3pt);
\draw[line width=1pt,fill=black] (3.,-3.) circle (3pt);
\draw[line width=1pt,fill=black] (4.,-3.) circle (3pt);
\draw[line width=1pt,fill=black] (5.,-3.) circle (3pt);
\draw[line width=1pt,color=black] (-5.,-5.) circle (3pt);
\draw[line width=1pt,color=black] (6.,-3.) circle (3pt);
\draw[line width=1pt,color=black] (7.,-3.) circle (3pt);
\draw[line width=1pt,fill=black] (-8.,-3.) circle (3pt);
\draw[line width=1pt,fill=black] (-7.,-3.) circle (3pt);
\draw[line width=1pt,color=black] (-5.,-3.) circle (3pt);
\draw[line width=1pt,color=black] (-4.,-3.) circle (3pt);
\draw[line width=1pt,color=black] (-3.,-3.) circle (3pt);
\draw[line width=1pt,color=black] (-4.,-4.) circle (3pt);
\draw[line width=1pt,color=black] (-5.,-4.) circle (3pt);
\draw[line width=1pt,color=black] (-6.,-4.) circle (3pt);
\draw[line width=1pt,fill=black] (5.,-5.) circle (3pt);
\draw[line width=1pt,fill=black] (-6.,-3.) circle (3pt);
\draw[line width=1pt,fill=black] (-7.,2.) circle (3pt);
\draw[line width=1pt,color=black] (-5.,1.) circle (3pt);
\draw[line width=1pt,color=black] (-4.,0.) circle (3pt);
\draw[line width=1pt,color=black] (-5.,0.) circle (3pt);
\draw[line width=1pt,color=black] (7.,2.) circle (3pt);
\draw[line width=1pt,color=black] (7.,0.) circle (3pt);
\draw[line width=1pt,fill=black] (3.,2.) circle (3pt);
\draw[line width=1pt,fill=black] (4.,2.) circle (3pt);
\draw[line width=1pt,fill=black] (5.,2.) circle (3pt);
\draw[line width=1pt,color=black] (8.,2.) circle (3pt);
\draw[line width=1pt,color=black] (6.,1.) circle (3pt);
\draw[line width=1pt,color=black] (7.,1.) circle (3pt);
\draw[line width=1pt,color=black] (8.,1.) circle (3pt);
\draw[line width=1pt,color=black] (6.,0.) circle (3pt);
\draw[line width=1pt,color=black] (-3.,0.) circle (3pt);
\draw[line width=1pt,color=black] (6.,2.) circle (3pt);
\draw[line width=1pt,color=black] (8.,0.) circle (3pt);
\draw[line width=1pt,fill=black] (5.,0.) circle (3pt);
\draw[line width=1pt,fill=black] (4.,0.) circle (3pt);
\draw[line width=1pt,color=black] (8.,-4.) circle (3pt);
\draw[line width=1pt,color=black] (6.,-4.) circle (3pt);
\draw[line width=1pt,color=black] (7.,-4.) circle (3pt);
\draw[line width=1pt,color=black] (8.,-3.) circle (3pt);
\draw[line width=1pt,color=black] (5.,-4.) circle (3pt);
\draw[line width=1pt,color=black] (4.,-4.) circle (3pt);
\draw[line width=1pt,color=black] (7.,-5.) circle (3pt);
\draw[line width=1pt,color=black] (6.,-5.) circle (3pt);
\draw[line width=1pt,fill=black] (4.,-5.) circle (3pt);
\end{scriptsize}
\end{tikzpicture}

\caption{Different cases of the regularisation procedure}
\label{fig:threerows}
\end{figure}

In the second case, we have that for $k = k_1$ Case 4 of the algorithm applies. Then, the leftmost point of $\hat{B}_{k_1+1}$ does not lie directly below the leftmost point of $\hat{B}_{k_1}$, but it lies to its left (but no further than the leftmost point of $\hat{A}_{k_1}$). Again, for every $k = k_1 + 1,...,k_3 - 1$ the leftmost point of $\hat{B}_k$ lies below the leftmost point of $\hat{B}_{k_1+1}$. Again, when we place the sets $\hat{A}_{k_3}$ and $\hat{B}_{k_3}$, Case~1 or Case~3 of the  procedure  applies. In Case 1, the leftmost point of $\hat{B}_{k_3}$ is again placed below the leftmost point of $\hat{B}_{k_1+1}$. Hence, the rightmost point of $\hat{A}_{k_3}$ is placed either below a point in $\hat{A}_{k_1}$ or, in view of the definition of  $\hat{B}_{k_1+1}$, one point to the left from the leftmost point of $\hat{A}_{k_1}$.   As before, by Lemma~\ref{lem:vertical}(1) and  Proposition \ref{prop:algorithmdecreasesenergy}   for  columns in place of rows, we see that the energy of $(\hat{A},\hat{B})$ was not minimal,  a contradiction. In Case 3,  a point  of $\hat{A}_{k_3}$ is placed below the leftmost point of  $\hat{B}_{k_3-1}$. This shows that the leftmost point of $\hat{A}_{k_3}$ is placed either below a point in $\hat{A}_{k_1}$ or one point to the right from the rightmost point of $\hat{A}_{k_1}$.  The situation is presented in the bottom line of Figure \ref{fig:threerows} in a simplified setting with $k_1 = 1$ and $k_3 = 3$.  As before, we obtain a contradiction to the minimality of $(\hat{A},\hat{B})$, and the proof is concluded. 
\end{proof}

A careful inspection of the proofs of Proposition
\ref{prop:algorithmdecreasesenergy} and Theorem \ref{thm:connected}
provides  some more information about the structure of any minimising
configuration, collected in the following   statements. 

\begin{proposition}\label{cor:rows}
Let $(A,B)$ be an optimal configuration. Then, for any row $R_k$, the sets $A^{\rm row}_k$, $B^{\rm row}_k$ and $(A \cup B)_k^{\rm row}$ are connected. The same claim holds for columns. \qed
\end{proposition}

\begin{proposition}\label{cor:nomissingrows}
Let $(A,B)$ be an optimal configuration. If for some $1 \leq k_1 < k_2 \leq N_{\rm row}$ we have $A^{\rm row}_{k_1}, A^{\rm row}_{k_2} \neq \emptyset$, then also $A^{\rm row}_k \neq \emptyset$ for all $k_1 \leq k \leq k_2$. The same claim holds for columns and the set $B$. \qed
\end{proposition}

\begin{proposition}\label{cor:allononeside}
Let $(A,B)$ be an optimal configuration. Suppose that there exists a row $R_{k_0}$ such that $A^{\rm row}_{k_0}, B^{\rm row}_{k_0} \neq \emptyset$ and $A^{\rm row}_{k_0}$ lies to the left of $B^{\rm row}_{k_0}$. Then, for every row $R_k$ either $A^{\rm row}_k$ lies to the left of $B^{\rm row}_k$ or one of these sets is empty. The same claim holds for columns  and if we interchange the roles of $A$ and $B$. \qed
\end{proposition}

 We observe that  Theorem \ref{thm:connected} and  Proposition
\ref{cor:rows} imply Theorem \ref{thm:main}.i and that Theorem \ref{thm:main}.ii follows from
Proposition~\ref{prop:algorithmdecreasesenergy} and
 Propositions  
\ref{cor:nomissingrows}--\ref{cor:allononeside}. Corollary \ref{cor:interface} implies   Theorem~\ref{thm:main}.iii and  will be crucial for our later considerations.  To this end, we introduce the following definition.

\begin{definition}
The interface $I_{AB}$ (between $A$ and $B$) is the set of midpoints of edges connecting a point in $A$ with a point in $B$. We say that there is an edge between two points $p,q \in I_{AB}$ if $|p - q| \in \lbrace 1/\sqrt{2}, 1\rbrace$ and the line segment  between $p$ and $q$   does not intersect any point in $\mathbb{Z}^2$. We say that the interface is connected if it is connected as a graph.
\end{definition}

In other words, a point $p \in \mathbb{R}^2$ lies in the interface $I_{AB}$ between $A$ and $B$ if there exist points $p_1 \in A$ and $p_2 \in B$ such that $|p - p_1| = |p - p_2| = 1/2$. Necessarily, the interface is a subset of the lattice $\{ (k + \frac{1}{2},l)\colon k,l \in \mathbb{Z} \} \cup \{ (k, l + \frac{1}{2})\colon k,l \in \mathbb{Z} \}$.  An example is presented in Figure \ref{fig:interface}. 

\begin{figure}[h]

\begin{tikzpicture}[line cap=round,line join=round,>=triangle 45,x=0.8cm,y=0.8cm]
\clip(1.5,1.7) rectangle (11.5,8.3);
\draw[line width=0.5pt] (3.5,3.)-- (4.,3.5);
\draw[line width=0.5pt] (4.,3.5)-- (5.,3.5);
\draw[line width=0.5pt] (5.,3.5)-- (5.5,4.);
\draw[line width=0.5pt] (5.5,4.)-- (5.5,5.);
\draw[line width=0.5pt] (5.5,5.)-- (6.,5.5);
\draw[line width=0.5pt] (6.,5.5)-- (7.,5.5);
\draw[line width=0.5pt] (7.,5.5)-- (7.5,6.);
\draw[line width=0.5pt] (7.5,6.)-- (8.,6.5);
\draw[line width=0.5pt] (8.,6.5)-- (9.,6.5);
\begin{scriptsize}
\draw[line width=1pt,color=black] (5.,7.) circle (2.5pt);
\draw[line width=1pt,color=black] (6.,6.) circle (2.5pt);
\draw[line width=1pt,color=black] (5.,5.) circle (2.5pt);
\draw[line width=1pt,color=black] (4.,6.) circle (2.5pt);
\draw[line width=1pt,color=black] (4.,5.) circle (2.5pt);
\draw[line width=1pt,fill=black] (7.,3.) circle (2.5pt);
\draw[line width=1pt,fill=black] (6.,5.) circle (2.5pt);
\draw[line width=1pt,fill=black] (7.,5.) circle (2.5pt);
\draw[line width=1pt,fill=black] (9.,5.) circle (2.5pt);
\draw[line width=1pt,fill=black] (8.,6.) circle (2.5pt);
\draw[line width=1pt,color=black] (3.,5.) circle (2.5pt);
\draw[line width=1pt,color=black] (5.,8.) circle (2.5pt);
\draw[line width=1pt,color=black] (4.,7.) circle (2.5pt);
\draw[line width=1pt,color=black] (2.,3.) circle (2.5pt);
\draw[line width=1pt,color=black] (2.,4.) circle (2.5pt);
\draw[line width=1pt,fill=black] (6.,3.) circle (2.5pt);
\draw[line width=1pt,fill=black] (4.,2.) circle (2.5pt);
\draw[line width=1pt,fill=black] (5.,3.) circle (2.5pt);
\draw[line width=1pt,fill=black] (8.,4.) circle (2.5pt);
\draw[line width=1pt,fill=black] (8.,5.) circle (2.5pt);
\draw[line width=1pt,color=black] (3.,6.) circle (2.5pt);
\draw[line width=1pt,color=black] (7.,8.) circle (2.5pt);
\draw[line width=1pt,color=black] (9.,7.) circle (2.5pt);
\draw[line width=1pt,color=black] (7.,6.) circle (2.5pt);
\draw[line width=1pt,color=black] (8.,7.) circle (2.5pt);
\draw[line width=1pt,fill=black] (11.,6.) circle (2.5pt);
\draw[line width=1pt,fill=black] (11.,5.) circle (2.5pt);
\draw[line width=1pt,fill=black] (6.,2.) circle (2.5pt);
\draw[line width=1pt,fill=black] (9.,3.) circle (2.5pt);
\draw[line width=1pt,fill=black] (10.,4.) circle (2.5pt);
\draw[line width=1pt,fill=black] (10.,6.) circle (2.5pt);
\draw[line width=1pt,fill=black] (9.,6.) circle (2.5pt);
\draw[line width=1pt,fill=black] (9.,4.) circle (2.5pt);
\draw[line width=1pt,fill=black] (10.,5.) circle (2.5pt);
\draw[line width=1pt,fill=black] (8.,3.) circle (2.5pt);
\draw[line width=1pt,color=black] (7.,7.) circle (2.5pt);
\draw[line width=1pt,color=black] (5.,4.) circle (2.5pt);
\draw[line width=1pt,color=black] (6.,7.) circle (2.5pt);
\draw[line width=1pt,color=black] (6.,8.) circle (2.5pt);
\draw[line width=1pt,color=black] (4.,4.) circle (2.5pt);
\draw[line width=1pt,color=black] (3.,4.) circle (2.5pt);
\draw[line width=1pt,color=black] (3.,3.) circle (2.5pt);
\draw[line width=1pt,fill=black] (5.,2.) circle (2.5pt);
\draw[line width=1pt,fill=black] (6.,4.) circle (2.5pt);
\draw[line width=1pt,fill=black] (7.,4.) circle (2.5pt);
\draw[line width=1pt,fill=black] (4.,3.) circle (2.5pt);
\draw[line width=1pt,color=black] (5.,6.) circle (2.5pt);
\draw [color=black] (3.5,3.)-- ++(-3pt,-3pt) -- ++(7.0pt,7.0pt) ++(-7.0pt,0) -- ++(7.0pt,-7.0pt);
\draw [color=black] (4.,3.5)-- ++(-3pt,-3pt) -- ++(7.0pt,7.0pt) ++(-7.0pt,0) -- ++(7.0pt,-7.0pt);
\draw [color=black] (5.,3.5)-- ++(-3pt,-3pt) -- ++(7.0pt,7.0pt) ++(-7.0pt,0) -- ++(7.0pt,-7.0pt);
\draw [color=black] (5.5,4.)-- ++(-3pt,-3pt) -- ++(7.0pt,7.0pt) ++(-7.0pt,0) -- ++(7.0pt,-7.0pt);
\draw [color=black] (5.5,5.)-- ++(-3pt,-3pt) -- ++(7.0pt,7.0pt) ++(-7.0pt,0) -- ++(7.0pt,-7.0pt);
\draw [color=black] (6.,5.5)-- ++(-3pt,-3pt) -- ++(7.0pt,7.0pt) ++(-7.0pt,0) -- ++(7.0pt,-7.0pt);
\draw [color=black] (7.,5.5)-- ++(-3pt,-3pt) -- ++(7.0pt,7.0pt) ++(-7.0pt,0) -- ++(7.0pt,-7.0pt);
\draw [color=black] (7.5,6.)-- ++(-3pt,-3pt) -- ++(7.0pt,7.0pt) ++(-7.0pt,0) -- ++(7.0pt,-7.0pt);
\draw [color=black] (8.,6.5)-- ++(-3pt,-3pt) -- ++(7.0pt,7.0pt) ++(-7.0pt,0) -- ++(7.0pt,-7.0pt);
\draw [color=black] (9.,6.5)-- ++(-3pt,-3pt) -- ++(7.0pt,7.0pt) ++(-7.0pt,0) -- ++(7.0pt,-7.0pt);
\end{scriptsize}
\end{tikzpicture}

\caption{Definition of the interface}
\label{fig:interface}
\end{figure}
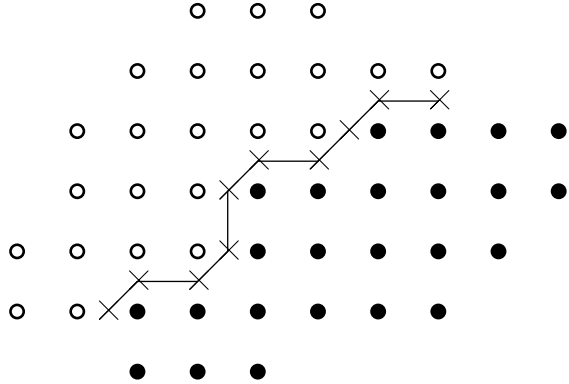

Notice that  Proposition  
\ref{cor:rows} implies that there is at most one point in $I_{AB}$ which is a midpoint of an edge between a point in $A^{\rm row}_k$ and a point in $B^{\rm row}_k$. Similarly, there is at most one point in $I_{AB}$ which is a midpoint of an edge between a point in $A^{\rm col}_k$ and a point in $B^{\rm col}_k$. Hence, we get the following result.

\begin{corollary}\label{cor:interface}
For any optimal configuration $(A,B)$, the interface $I_{AB}$ is connected. Moreover, it is monotone: up to reflections, it goes only upwards and to the right, i.e., given $p,q \in I_{AB}$, if $p_1 > q_1$, then $p_2 \ge q_2$.  \qed
\end{corollary}

We will use this result to study the minimal configurations in the
following way: we will identify all possible shapes of the interface,
collected in different classes. Analysing the different classes in
detail, we will show that there always exists an optimal configuration
in the most natural class (called Class~$\mathcal{I}$). For this
class, we are able to  directly compute the minimal energy,
explicitly exhibit a minimiser,  and provide a sharp estimate of the possible mismatch of ground
states in terms of their size, see \eqref{eq:fluct}.   

Let us also note that the introduction of $I_{AB}$ enables us to write a convenient formula for the energy associated to an optimal configuration $(A,B)$. Namely, denote by $E_A$ the energy inside $A$, i.e.,  minus the number of bonds between $A$-points. In a similar fashion, we define $E_B$.   Eventually, by $E_{AB} := - \# I_{AB} \beta$ we denote the interfacial energy, i.e., minus the number of bonds between $A$- and $B$-points weighted by the coefficient $\beta$. Then,
\begin{equation}\label{eq:formulafortheenergy}
E(A,B) = E_A + E_B + E_{AB}.
\end{equation}
This simple formula has a very important consequence. Namely, if we separate the sets $A$ and $B$ and reattach them in a different way (i.e., apply an isometry to one or both sets), then $E_A$ and $E_B$ do not change, but $E_{AB}$ possibly might. Therefore, if a configuration is optimal, it has the longest possible interface with respect to this operation. We will use variants of this argument on multiple occasions in Section \ref{sec:regularisation}.

\section{A collection of  examples}\label{sec:info}

In this short section, we consider a few examples of minimisers that
will serve as a motivation for the discussion about possible shapes of
the interface in the next section. By Theorem~\ref{thm:connected}, for
any optimal configuration, both sets $A$ and $B$ are connected. The
properties of an optimal configuration are further restricted by
 Propositions  
\ref{cor:rows}--\ref{cor:interface}.  For different
choices of $N_A, N_B > 0$ and $\beta \in (0,1)$, we provide here a
complete account of optimal configurations. Note that the limited
number of points involved allows a direct exhaustive analysis. Even
though some optimal configuration is   irregular, the main effort in this paper will be to prove that actually for $N_A = N_B$ and $\beta \le  1/2$ one may find an optimal configuration which is very regular, in the sense that they roughly consist of two rectangles as given in Theorem \ref{thm:main}.v.

The first example consists of only three points: we have $N_A = 2$, $N_B = 1$, for any $\beta \in (0,1)$. Even then, the minimiser may fail to be unique: up to isometries, we have two minimisers, both presented in Figure \ref{fig:threepointexample}. 

\begin{figure}[h!]
\centering

\begin{tikzpicture}[line cap=round,line join=round,>=triangle 45,x=1.0cm,y=1.0cm]
\clip(-1,-0.5) rectangle (8,1.5);
\begin{scriptsize}
\draw[line width=1pt,color=black] (0.,0.) circle (3pt);
\draw[line width=1pt,fill=black] (1.,0.) circle (3pt);
\draw[line width=1pt,color=black] (0.,1.) circle (3pt);
\draw[line width=1pt,color=black] (5.,0.) circle (3pt);
\draw[line width=1pt,color=black] (6.,0.) circle (3pt);
\draw[line width=1pt,fill=black] (7.,0.) circle (3pt);
\end{scriptsize}
\end{tikzpicture}
\caption{Minimisers for $N_A = 2, N_B = 1$}
\label{fig:threepointexample}
\end{figure}
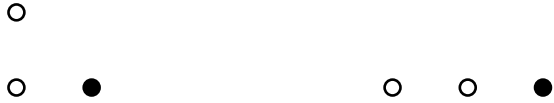


The second example consists of six points: we have $N_A = N_B = 3$ for any $\beta \in (0,1)$. The numbers of $A$- and $B$-points are equal. The minimiser may fail to be unique: up to isometries, we have two minimisers, both presented in Figure \ref{fig:sixpointexample}. Note that the interface is not necessarily straight. However, there is a minimiser which has a straight interface. 

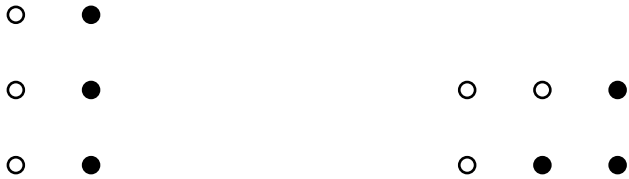
\begin{figure}[h!]
\begin{tikzpicture}[line cap=round,line join=round,>=triangle 45,x=1.0cm,y=1.0cm]
\clip(-1,-0.5) rectangle (9,2.5);
\begin{scriptsize}
\draw[line width=1pt,color=black] (0.,0.) circle (3pt);
\draw[line width=1pt,fill=black] (1.,0.) circle (3pt);
\draw[line width=1pt,fill=black] (1.,1.) circle (3pt);
\draw[line width=1pt,color=black] (0.,1.) circle (3pt);
\draw[line width=1pt,color=black] (0.,2.) circle (3pt);
\draw[line width=1pt,fill=black] (1.,2.) circle (3pt);
\draw[line width=1pt,color=black] (6.,0.) circle (3pt);
\draw[line width=1pt,color=black] (6.,1.) circle (3pt);
\draw[line width=1pt,fill=black] (7.,0.) circle (3pt);
\draw[line width=1pt,color=black] (7.,1.) circle (3pt);
\draw[line width=1pt,fill=black] (8.,1.) circle (3pt);
\draw[line width=1pt,fill=black] (8.,0.) circle (3pt);
\end{scriptsize}
\end{tikzpicture}
\caption{Minimisers for $N_A = 3, N_B = 3$}
\label{fig:sixpointexample}
\end{figure}


The third example consists of eight points: we have $N_A = N_B = 4$ for any $\beta \in (0,1)$. In this case, the minimiser is unique. Up to isometries, the only solution is presented in Figure \ref{fig:fourpointexample}. Note that the interface is straight and both rectangles are ``full''. This situation is very special, and in a generic case we do not expect uniqueness. 

\begin{figure}[h!]

\begin{tikzpicture}[line cap=round,line join=round,>=triangle 45,x=1.0cm,y=1.0cm]
\clip(-2,-0.5) rectangle (3,1.5);
\begin{scriptsize}
\draw[line width=1pt,color=black] (0.,0.) circle (3pt);
\draw[line width=1pt,fill=black] (1.,0.) circle (3pt);
\draw[line width=1pt,color=black] (0.,1.) circle (3pt);
\draw[line width=1pt,color=black] (-1.,0.) circle (3pt);
\draw[line width=1pt,fill=black] (2.,0.) circle (3pt);
\draw[line width=1pt,color=black] (-1.,1.) circle (3pt);
\draw[line width=1pt,fill=black] (1.,1.) circle (3pt);
\draw[line width=1pt,fill=black] (2.,1.) circle (3pt);
\end{scriptsize}
\end{tikzpicture}

\caption{Unique minimiser for $N_A = 4, N_B = 4$}
\label{fig:fourpointexample}
\end{figure}
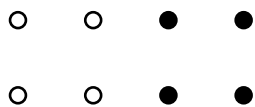

The fourth example consists of seven points: we have $N_A = 3$, $N_B = 4$, for any $\beta \in (0,1)$. Up to isometries, we have three minimisers, presented in Figure \ref{fig:sevenpointexample}. As in the second example of Figure~\ref{fig:sixpointexample}, in the configuration on the right the interface is ``L-shaped''. 

\begin{figure}[h!]

\begin{tikzpicture}[line cap=round,line join=round,>=triangle 45,x=0.9cm,y=0.9cm]
\clip(-1,-0.5) rectangle (14,2.5);
\begin{scriptsize}
\draw[line width=1pt,color=black] (0.,0.) circle (3pt);
\draw[line width=1pt,color=black] (0.,1.) circle (3pt);
\draw[line width=1pt,color=black] (0.,2.) circle (3pt);
\draw[line width=1pt,fill=black] (1.,1.) circle (3pt);
\draw[line width=1pt,fill=black] (1.,0.) circle (3pt);
\draw[line width=1pt,fill=black] (2.,0.) circle (3pt);
\draw[line width=1pt,fill=black] (2.,1.) circle (3pt);
\draw[line width=1pt,color=black] (5.,0.) circle (3pt);
\draw[line width=1pt,color=black] (6.,0.) circle (3pt);
\draw[line width=1pt,color=black] (6.,1.) circle (3pt);
\draw[line width=1pt,fill=black] (7.,1.) circle (3pt);
\draw[line width=1pt,fill=black] (7.,0.) circle (3pt);
\draw[line width=1pt,fill=black] (8.,0.) circle (3pt);
\draw[line width=1pt,fill=black] (8.,1.) circle (3pt);
\draw[line width=1pt,color=black] (11.,1.) circle (3pt);
\draw[line width=1pt,color=black] (11.,2.) circle (3pt);
\draw[line width=1pt,color=black] (12.,2.) circle (3pt);
\draw[line width=1pt,fill=black] (12.,1.) circle (3pt);
\draw[line width=1pt,fill=black] (12.,0.) circle (3pt);
\draw[line width=1pt,fill=black] (13.,1.) circle (3pt);
\draw[line width=1pt,fill=black] (13.,0.) circle (3pt);
\end{scriptsize}
\end{tikzpicture}

\caption{Minimisers for $N_A = 3, N_B = 4$}
\label{fig:sevenpointexample}
\end{figure}
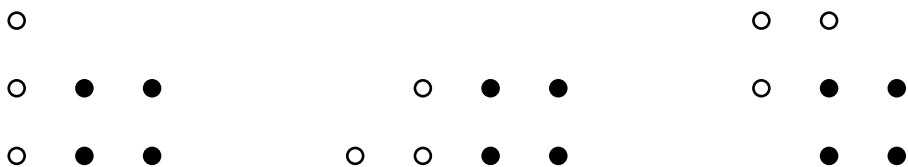

The fifth example consists of ten points: we have $N_A = N_B = 5$ for any $\beta \in (0,1)$. Up to isometries, we have five possible minimisers, presented in Figure \ref{fig:fivepoints}. Notice that the heights of the two types may differ and that the interface may fail to be straight. Furthermore, the two configurations on the left differ even though the interface is straight.

\begin{figure}[h!]
    \centering
    
\begin{tikzpicture}[line cap=round,line join=round,>=triangle 45,x=0.8cm,y=0.8cm]
\clip(3.7,-1) rectangle (19.3,7.5);
\begin{scriptsize}
\draw[line width=1pt,color=black] (5.,7.) circle (3pt);
\draw[line width=1pt,color=black] (5.,6.) circle (3pt);
\draw[line width=1pt,color=black] (5.,5.) circle (3pt);
\draw[line width=1pt,color=black] (4.,6.) circle (3pt);
\draw[line width=1pt,color=black] (4.,5.) circle (3pt);
\draw[line width=1pt,fill=black] (6.,6.) circle (3pt);
\draw[line width=1pt,fill=black] (6.,5.) circle (3pt);
\draw[line width=1pt,fill=black] (7.,5.) circle (3pt);
\draw[line width=1pt,fill=black] (7.,6.) circle (3pt);
\draw[line width=1pt,fill=black] (6.,7.) circle (3pt);
\draw[line width=1pt,color=black] (10.,7.) circle (3pt);
\draw[line width=1pt,color=black] (11.,7.) circle (3pt);
\draw[line width=1pt,color=black] (12.,7.) circle (3pt);
\draw[line width=1pt,color=black] (11.,6.) circle (3pt);
\draw[line width=1pt,color=black] (10.,6.) circle (3pt);
\draw[line width=1pt,fill=black] (12.,6.) circle (3pt);
\draw[line width=1pt,fill=black] (12.,5.) circle (3pt);
\draw[line width=1pt,fill=black] (11.,5.) circle (3pt);
\draw[line width=1pt,fill=black] (13.,6.) circle (3pt);
\draw[line width=1pt,fill=black] (13.,5.) circle (3pt);
\draw[line width=1pt,color=black] (15.,3.) circle (3pt);
\draw[line width=1pt,color=black] (17.,4.) circle (3pt);
\draw[line width=1pt,color=black] (16.,3.) circle (3pt);
\draw[line width=1pt,color=black] (15.,4.) circle (3pt);
\draw[line width=1pt,color=black] (16.,4.) circle (3pt);
\draw[line width=1pt,fill=black] (19.,4.) circle (3pt);
\draw[line width=1pt,fill=black] (17.,3.) circle (3pt);
\draw[line width=1pt,fill=black] (18.,3.) circle (3pt);
\draw[line width=1pt,fill=black] (19.,3.) circle (3pt);
\draw[line width=1pt,fill=black] (18.,4.) circle (3pt);
\draw[line width=1pt,fill=black] (12.,0.) circle (3pt);
\draw[line width=1pt,fill=black] (13.,1.) circle (3pt);
\draw[line width=1pt,fill=black] (13.,2.) circle (3pt);
\draw[line width=1pt,fill=black] (13.,0.) circle (3pt);
\draw[line width=1pt,fill=black] (12.,1.) circle (3pt);
\draw[line width=1pt,color=black] (11.,1.) circle (3pt);
\draw[line width=1pt,color=black] (12.,2.) circle (3pt);
\draw[line width=1pt,color=black] (10.,1.) circle (3pt);
\draw[line width=1pt,color=black] (11.,2.) circle (3pt);
\draw[line width=1pt,color=black] (5.,2.) circle (3pt);
\draw[line width=1pt,color=black] (5.,1.) circle (3pt);
\draw[line width=1pt,color=black] (5.,0.) circle (3pt);
\draw[line width=1pt,color=black] (4.,0.) circle (3pt);
\draw[line width=1pt,color=black] (4.,1.) circle (3pt);
\draw[line width=1pt,fill=black] (6.,0.) circle (3pt);
\draw[line width=1pt,fill=black] (6.,1.) circle (3pt);
\draw[line width=1pt,fill=black] (6.,2.) circle (3pt);
\draw[line width=1pt,fill=black] (7.,2.) circle (3pt);
\draw[line width=1pt,fill=black] (7.,1.) circle (3pt);
\draw[line width=1pt,color=black] (10.,2.) circle (3pt);
\end{scriptsize}
\end{tikzpicture}    

\caption{Minimisers for $N_A = 5, N_B = 5$}
\label{fig:fivepoints}
\end{figure}
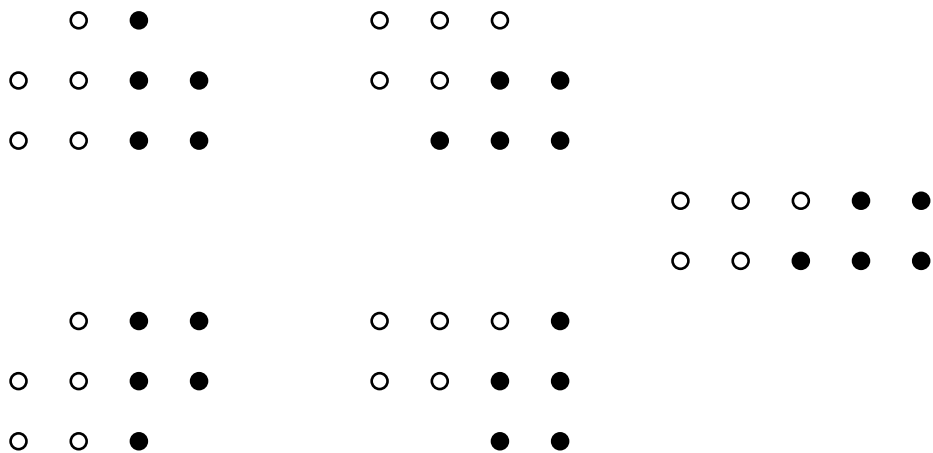

The final example consists of sixteen points: we have $N_A = 12$ and $N_B = 4$. Then, the situation may differ with $\beta$. For $\beta \in (1/2,1)$, up to isometries we have two possible minimisers (with energy $-20-4\beta$), presented in Figure \ref{fig:sixteenpointexample}. In one case, we have a straight interface, while in the other it is L-shaped. 

\begin{figure}[h!]

\begin{tikzpicture}[line cap=round,line join=round,>=triangle 45,x=1.0cm,y=1.0cm]
\clip(-1,-0.5) rectangle (11,3.5);
\begin{scriptsize}
\draw[line width=1pt,color=black] (0.,0.) circle (3pt);
\draw[line width=1pt,color=black] (0.,1.) circle (3pt);
\draw[line width=1pt,color=black] (0.,2.) circle (3pt);
\draw[line width=1pt,color=black] (0.,3.) circle (3pt);
\draw[line width=1pt,color=black] (1.,3.) circle (3pt);
\draw[line width=1pt,color=black] (2.,3.) circle (3pt);
\draw[line width=1pt,color=black] (3.,3.) circle (3pt);
\draw[line width=1pt,color=black] (3.,2.) circle (3pt);
\draw[line width=1pt,fill=black] (3.,1.) circle (3pt);
\draw[line width=1pt,fill=black] (3.,0.) circle (3pt);
\draw[line width=1pt,fill=black] (2.,0.) circle (3pt);
\draw[line width=1pt,color=black] (1.,0.) circle (3pt);
\draw[line width=1pt,color=black] (1.,1.) circle (3pt);
\draw[line width=1pt,color=black] (1.,2.) circle (3pt);
\draw[line width=1pt,color=black] (2.,2.) circle (3pt);
\draw[line width=1pt,fill=black] (2.,1.) circle (3pt);
\draw[line width=1pt,color=black] (7.,3.) circle (3pt);
\draw[line width=1pt,color=black] (7.,2.) circle (3pt);
\draw[line width=1pt,color=black] (7.,1.) circle (3pt);
\draw[line width=1pt,color=black] (7.,0.) circle (3pt);
\draw[line width=1pt,color=black] (8.,0.) circle (3pt);
\draw[line width=1pt,color=black] (8.,1.) circle (3pt);
\draw[line width=1pt,color=black] (8.,2.) circle (3pt);
\draw[line width=1pt,color=black] (8.,3.) circle (3pt);
\draw[line width=1pt,color=black] (9.,3.) circle (3pt);
\draw[line width=1pt,color=black] (9.,2.) circle (3pt);
\draw[line width=1pt,color=black] (9.,1.) circle (3pt);
\draw[line width=1pt,color=black] (9.,0.) circle (3pt);
\draw[line width=1pt,fill=black] (10.,0.) circle (3pt);
\draw[line width=1pt,fill=black] (10.,1.) circle (3pt);
\draw[line width=1pt,fill=black] (10.,2.) circle (3pt);
\draw[line width=1pt,fill=black] (10.,3.) circle (3pt);
\end{scriptsize}
\end{tikzpicture}

\caption{Minimisers for $N_A = 12, N_B = 4$,  large $\beta$}
\label{fig:sixteenpointexample}
\end{figure}
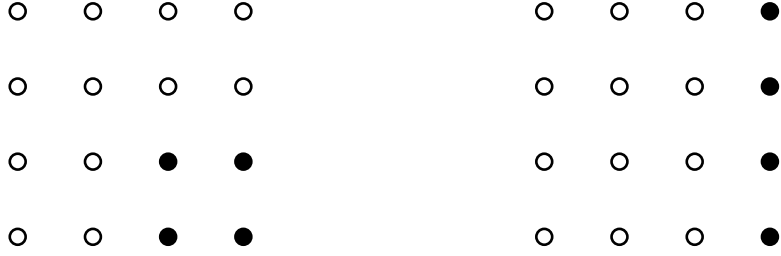

For $\beta \in (0,1/2)$, up to isometries, we have three possible minimisers (with energy $-21-2\beta$), presented in Figure \ref{fig:sixteenpointexamplebigalpha}. Here, the structure of sets $A$ and $B$ is fixed, but we may attach them in a few different ways.

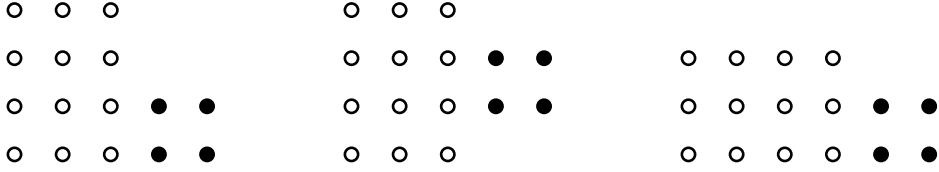
\begin{figure}[h!]

\begin{tikzpicture}[line cap=round,line join=round,>=triangle 45,x=0.64cm,y=0.64cm]
\clip(-0.3,-0.5) rectangle (19.3,3.5);
\begin{scriptsize}
\draw[line width=1pt,color=black] (0.,0.) circle (2.5pt);
\draw[line width=1pt,color=black] (0.,1.) circle (2.5pt);
\draw[line width=1pt,color=black] (0.,2.) circle (2.5pt);
\draw[line width=1pt,color=black] (0.,3.) circle (2.5pt);
\draw[line width=1pt,color=black] (1.,3.) circle (2.5pt);
\draw[line width=1pt,color=black] (2.,3.) circle (2.5pt);
\draw[line width=1pt,fill=black] (4.,0.) circle (2.5pt);
\draw[line width=1pt,fill=black] (4.,1.) circle (2.5pt);
\draw[line width=1pt,fill=black] (3.,1.) circle (2.5pt);
\draw[line width=1pt,fill=black] (3.,0.) circle (2.5pt);
\draw[line width=1pt,color=black] (2.,0.) circle (2.5pt);
\draw[line width=1pt,color=black] (1.,0.) circle (2.5pt);
\draw[line width=1pt,color=black] (1.,1.) circle (2.5pt);
\draw[line width=1pt,color=black] (1.,2.) circle (2.5pt);
\draw[line width=1pt,color=black] (2.,2.) circle (2.5pt);
\draw[line width=1pt,color=black] (2.,1.) circle (2.5pt);
\draw[line width=1pt,color=black] (7.,3.) circle (2.5pt);
\draw[line width=1pt,color=black] (7.,2.) circle (2.5pt);
\draw[line width=1pt,color=black] (7.,1.) circle (2.5pt);
\draw[line width=1pt,color=black] (7.,0.) circle (2.5pt);
\draw[line width=1pt,color=black] (8.,0.) circle (2.5pt);
\draw[line width=1pt,color=black] (8.,1.) circle (2.5pt);
\draw[line width=1pt,color=black] (8.,2.) circle (2.5pt);
\draw[line width=1pt,color=black] (8.,3.) circle (2.5pt);
\draw[line width=1pt,color=black] (9.,3.) circle (2.5pt);
\draw[line width=1pt,color=black] (9.,2.) circle (2.5pt);
\draw[line width=1pt,color=black] (9.,1.) circle (2.5pt);
\draw[line width=1pt,color=black] (9.,0.) circle (2.5pt);
\draw[line width=1pt,fill=black] (11.,1.) circle (2.5pt);
\draw[line width=1pt,fill=black] (10.,1.) circle (2.5pt);
\draw[line width=1pt,fill=black] (10.,2.) circle (2.5pt);
\draw[line width=1pt,fill=black] (11.,2.) circle (2.5pt);
\draw[line width=1pt,color=black] (14.,0.) circle (2.5pt);
\draw[line width=1pt,color=black] (14.,1.) circle (2.5pt);
\draw[line width=1pt,color=black] (14.,2.) circle (2.5pt);
\draw[line width=1pt,color=black] (15.,2.) circle (2.5pt);
\draw[line width=1pt,color=black] (15.,1.) circle (2.5pt);
\draw[line width=1pt,color=black] (15.,0.) circle (2.5pt);
\draw[line width=1pt,color=black] (16.,0.) circle (2.5pt);
\draw[line width=1pt,color=black] (16.,1.) circle (2.5pt);
\draw[line width=1pt,color=black] (16.,2.) circle (2.5pt);
\draw[line width=1pt,color=black] (17.,2.) circle (2.5pt);
\draw[line width=1pt,color=black] (17.,1.) circle (2.5pt);
\draw[line width=1pt,color=black] (17.,0.) circle (2.5pt);
\draw[line width=1pt,fill=black] (18.,1.) circle (2.5pt);
\draw[line width=1pt,fill=black] (18.,0.) circle (2.5pt);
\draw[line width=1pt,fill=black] (19.,1.) circle (2.5pt);
\draw[line width=1pt,fill=black] (19.,0.) circle (2.5pt);
\end{scriptsize}
\end{tikzpicture}

\caption{Minimisers for $N_A = 12, N_B = 4$,  small  $\beta$}
\label{fig:sixteenpointexamplebigalpha}
\end{figure}

For $\beta  = 1/2$,  all  configurations presented in Figures \ref{fig:sixteenpointexample} and \ref{fig:sixteenpointexamplebigalpha} are minimal.

\section{Classification of admissible configurations}\label{sec:classi}

For simplicity, we will call the configurations which satisfy the
statement of Theorem \ref{thm:connected} and  of  the corollaries
below it {\it admissible}. In particular, these results show that
optimal configurations are admissible. In this section, we collect
admissible configurations in different classes. These classes will be
analysed in more detail in the subsequent sections. The starting point
is the observation that by   Proposition  
\ref{cor:nomissingrows} we have that   there cannot be a row $R_{k_0}$ such that $n^{\rm row}_k > 0$ above and below this row (for some $k > k_0$ and some other $k < k_0$), while $n^{\rm row}_{k_0} = 0$. The same result holds for columns. Therefore, we may cluster the minimisers into several classes which are easier to handle and are described using this property.

Let us start from the top and suppose that $ n^{\rm row}_1  > 0$ (otherwise, we exchange the roles of the two types). Denote by $R_{k_0}$ the last row such that $n^{\rm row}_{k_0} > 0$. Then, we have the two possibilities 
\begin{equation}\label{eq:k0=Nr}
{\rm (i)} \ \ k_0 = N_{\rm row} \quad \quad \quad \text{and} \quad \quad \quad  {\rm (ii)} \ \ k_0 < N_{\rm row}. 
\end{equation}
In case (i), we  distinguish three  possibilities, depending on whether $ B_1^{\rm row} $
and $B^{\rm row}_{N_{\rm row}}$ are empty or not: if $ m^{\rm row}_1,
m^{\rm row}_{N_{\rm row}} > 0$, then each row contains points from
both types.  This case corresponds to class~$\mathcal{I}$.  If
$ m^{\rm row}_1$ or $m^{\rm row}_{N_{\rm row}}$ equals zero, then the
$B$-part of the  configuration has a smaller height.  This
corresponds to either class~$\mathcal{II}$ or~$\mathcal{III}$. 
In case (ii),  we have  $n^{\rm row}_{N_{\rm row}}=0$ and $m^{\rm row}_{N_{\rm row}}>0$. We distinguish two possibilities: if the last column $ B_{N_{\rm row}}^{\rm col}$ is not empty, i.e.\   $m^{\rm col}_{N_{\rm col}} >0$,   the configuration is in class~$\mathcal{IV}$. The case $m^{\rm col}_{N_{\rm col}} =0$ instead corresponds to class~$\mathcal{V}$.

%

By performing the same analysis for columns, and recalling the
corollaries after Theorem \ref{thm:connected}, we end up with a number
of possibilities which we list below, where without restriction we
assume that $ n_1^{\rm col} >0$. This list is complete up to isometries and
changing roles of the types. For the sake of the presentation, by
applying Corollary~\ref{cor:interface} we can without restriction
(possibly up to isometry and changing the roles of the types) assume
that the interface is going upwards and to the right. We divide  all admissible configurations 
into five main \textit{classes}, the first three being quite regular
and the last two a bit more difficult to handle. In this section, we
list all classes and introduce appropriate notation for each of
them. In the next section  we advance  a regularisation procedure for all
configurations.  This has the aim of proving that   for $N_A = N_B$
and $\beta \le 1/2$ all minimal configurations  belong to  Class
$\mathcal{I}$, $\mathcal{IV}$,  or  $ \mathcal{V}$,  as
well as checking some fine geometrical properties of such minimisers.

\subsection{Class $\mathcal{I}$}

The first possibility is the reference case: we say that an admissible
configuration $(A,B)$ belongs to Class $\mathcal{I}$ if for each $k =
1,...,N_{\rm row}$ we have $n^{\rm row}_k > 0$ and $m^{\rm row}_k >
0$. In other words, \eqref{eq:k0=Nr}(i) holds with  $m^{\rm row}_1, m^{\rm
  row}_{N_{\rm row}} > 0$.  The situation is presented in Figure
\ref{fig:ClassIbefore}. Examples of optimal configurations in Class
$\mathcal{I}$ can be found in Figure \ref{fig:threepointexample} (on
the right), in Figure \ref{fig:sixpointexample} (both), in Figure
\ref{fig:fourpointexample}, in Figure \ref{fig:sevenpointexample} (in
the middle), in Figure \ref{fig:fivepoints} (all but the two middle
ones), and  in Figure \ref{fig:sixteenpointexample} (on the
right). The abundance of examples in Class $\mathcal{I}$ is in some
sense expected. Indeed, we will prove that for many choices of $N_A$,
$N_B$, and  $\beta$  existence of an optimal configuration in Class $\mathcal{I}$ is guaranteed.

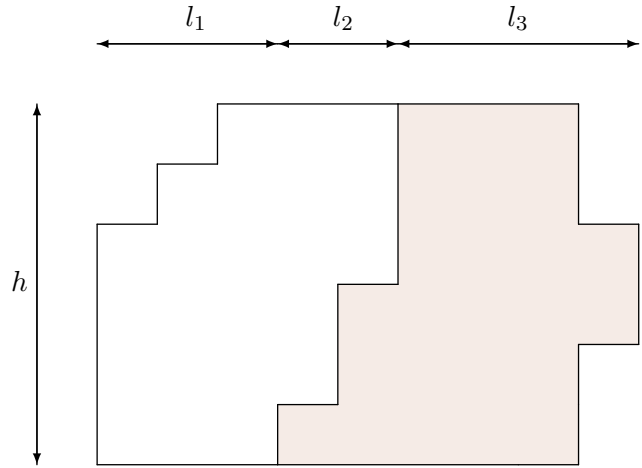
\begin{figure}[h]

\definecolor{zzttqq}{rgb}{0.6,0.2,0.}
\begin{tikzpicture}[line cap=round,line join=round,>=triangle 45,x=0.8cm,y=0.8cm]
\clip(-1.6,-0.1) rectangle (9.9,7.89);
\fill[line width=2.pt,color=zzttqq,fill=zzttqq,fill opacity=0.1] (5.,6.) -- (8.,6.) -- (8.,4.) -- (9.,4.) -- (9.,2.) -- (8.,2.) -- (8.,0.) -- (3.,0.) -- (3.,1.) -- (4.,1.) -- (4.,3.) -- (5.,3.) -- cycle;
\draw[line width=0.5pt] (0.,0.)-- (0.,4.);
\draw[line width=0.5pt] (0.,4.)-- (1.,4.);
\draw[line width=0.5pt] (1.,4.)-- (1.,5.);
\draw[line width=0.5pt] (1.,5.)-- (2.,5.);
\draw[line width=0.5pt] (2.,5.)-- (2.,6.);
\draw[line width=0.5pt] (2.,6.)-- (5.,6.);
\draw[line width=0.5pt] (5.,6.)-- (5.,3.);
\draw[line width=0.5pt] (5.,3.)-- (4.,3.);
\draw[line width=0.5pt] (4.,3.)-- (4.,1.);
\draw[line width=0.5pt] (4.,1.)-- (3.,1.);
\draw[line width=0.5pt] (3.,1.)-- (3.,0.);
\draw[line width=0.5pt] (3.,0.)-- (0.,0.);
\draw[line width=0.5pt] (5.,6.)-- (8.,6.);
\draw[line width=0.5pt] (8.,6.)-- (8.,4.);
\draw[line width=0.5pt] (8.,4.)-- (9.,4.);
\draw[line width=0.5pt] (9.,4.)-- (9.,2.);
\draw[line width=0.5pt] (8.,2.)-- (9.,2.);
\draw[line width=0.5pt] (8.,2.)-- (8.,1.);
\draw[line width=0.5pt] (8.,1.)-- (8.,0.);
\draw[line width=0.5pt] (8.,0.)-- (7.,0.);
\draw[line width=0.5pt] (7.,0.)-- (3.,0.);
\draw[-latex,line width=0.5pt] (0.,7.) -- (3.,7.);
\draw[-latex,line width=0.5pt] (3.,7.) -- (0.,7.);
\draw[-latex,line width=0.5pt] (3.,7.) -- (5.,7.);
\draw[-latex,line width=0.5pt] (5.,7.) -- (3.,7.);
\draw[-latex,line width=0.5pt] (5.,7.) -- (9.,7.);
\draw[-latex,line width=0.5pt] (9.,7.) -- (5.,7.);
\draw[-latex,line width=0.5pt] (-1.,6.) -- (-1.,0.);
\draw[-latex,line width=0.5pt] (-1.,0.) -- (-1.,6.);
\draw (1.3,7.8) node[anchor=north west] {$l_1$};
\draw (3.75,7.8) node[anchor=north west] {$l_2$};
\draw (6.65,7.8) node[anchor=north west] {$l_3$};
\draw (-1.6,3.4) node[anchor=north west] {$h$};
\end{tikzpicture}

\caption{Class $\mathcal{I}$}
\label{fig:ClassIbefore}
\end{figure}

Let us introduce the following notation. Let $h$ denote the number of rows (which in this case corresponds to the number of rows of both $A$ and $B$). Let $l_1$ denote the number of columns such that $A^{\rm col}_k \neq \emptyset$ and $B^{\rm col}_k = \emptyset$. Let $l_2$ denote the number of columns such that $A^{\rm col}_k \neq \emptyset$ and $B^{\rm col}_k \neq \emptyset$. Finally, let $l_3$ denote the number of columns such that $A^{\rm col}_k = \emptyset$ and $B^{\rm col}_k \neq \emptyset$. This notation is also presented in Figure \ref{fig:ClassIbefore}. Then,  in view of \eqref{eq:eq},   the energy  \eqref{eq: basic eneg}  may be expressed as  
\begin{equation}\label{eq:formulaforenergyclassI}
E(A,B) = - 2 (N_A + N_B) + (l_1 + l_2 + l_3) + h +   (1-\beta) (l_2 + h).
\end{equation}
In particular, the energy splits into the \textit{bulk energy} $-
2 (N_A + N_B)$ and, up to a factor $1/2$, into the \textit{lattice
  perimeter} introduced in  \eqref{eq:dbp3}. Clearly, only the latter
is relevant for identifying optimal configurations.  For convenience,
we will frequently refer to it as the surface energy. 

In the next section, we will simplify the structure of configurations in Class $\mathcal{I}$, without increasing the energy, in order to compute the minimal energy in this class. After such regularisation, it will turn out that we have two possibilities: either $l_2 = 0$ or $l_2 = 1$, i.e., either the interface is a straight line or it has one horizontal jump, see Proposition \ref{prop:classIregularisationstep1}.

\subsection{Class $\mathcal{II}$}
We say that an admissible configuration $(A,B)$ belongs to Class~$\mathcal{II}$ if there  exists  a column $C_{k_0}$ such that for all $k \leq k_0$ we have $n^{\rm col}_k > 0$ and $m^{\rm col}_k = 0$, for all $k > k_0$ we have $n^{\rm col}_k = 0$ and $m^{\rm col}_k > 0$, and $(A,B)$  does not lie  in Class~$\mathcal{I}$. In other words, the interface is a straight vertical line, and there exists at least one row which contains only one type (as otherwise $(A,B) \in \mathcal{I}$). Examples of optimal configurations in this class can be found in Figure \ref{fig:threepointexample} (on the left), in Figure \ref{fig:sevenpointexample} (on the left), and in Figure \ref{fig:sixteenpointexamplebigalpha} (all of them). Notice that in all these examples we have $N_A \neq N_B$. Indeed, in Section \ref{sec:regularisation} we will show that, if $N_A$ and $N_B$ are equal, such a configuration cannot be optimal.

A priori, this set of configurations may arise from both cases in \eqref{eq:k0=Nr}.  Up to changing the roles the two types, however, we may assume that we are in situation \eqref{eq:k0=Nr}(i), as we can see in the following simple observation.

\begin{lemma}\label{lem:classIIregularisation}
Fix $N_A, N_B > 0$ and $ \beta  \in (0,1)$. Suppose that $(A,B) \in \mathcal{II}$ is a minimal configuration. Then, there exists a minimal configuration $(\hat{A},\hat{B}) \in \mathcal{II}$ such that the last rows align, i.e., $n^{\rm row}_{N_{\rm row}} > 0$ and $m^{\rm row}_{N_{\rm row}} > 0$.
\end{lemma}

\begin{proof}
Without loss of generality, suppose that $n^{\rm row}_{N_{\rm row}} > 0$ and that $r_0 < N_{\rm row}$ is the biggest number such that $m^{\rm row}_{r_0} > 0$. Notice that, since the interface is a straight line, we may move the set $B$ by the vector $(0,r_0 - N_{\rm row})$ so that the last two rows align and this procedure does not increase the energy. The resulting configuration $(\hat{A},\hat{B})$ also lies in Class $\mathcal{II}$: if after this procedure we had also $n^{\rm row}_{1} > 0$ and $m^{\rm row}_{1} > 0$, i.e.,  $(\hat{A},\hat{B})$ lies in Class $\mathcal{I}$, then we would have added at least one bond. This induces a drop in the energy, a contradiction to the fact that $(A,B)$ is a minimal configuration.
\end{proof}

After applying this regularisation argument, we introduce the following notation. Up to reflection along the (straight) interface and interchanging the roles of the types, we may assume that $A$ is on the left-hand side and that it has more nonempty rows than $B$. Then, let $h_1$ denote the number of rows such that $A^{\rm row}_k \neq \emptyset$ and $B^{\rm row}_k = \emptyset$, and let $h_2$ be the number of rows such that $A^{\rm row}_k \neq \emptyset$ and $B^{\rm row}_k \neq \emptyset$. Moreover, let $l_1$ denote the number of columns such that $A^{\rm col}_k \neq \emptyset$ and $l_3$ denote the number of columns such that $B^{\rm col}_k \neq \emptyset$ (the notation $l_2$ is omitted on purpose to simplify some later regularisation arguments). Then, arguing as in the justification of formula \eqref{eq:formulaforenergyclassI},  see also \eqref{eq:eq},   the energy  \eqref{eq: basic eneg}  may be expressed as  
\begin{equation}\label{eq: energy,class2}
E(A,B) = - 2(N_A + N_B) + (l_1 + l_3) + (h_1 + h_2) + 
(1-\beta)  h_2.
\end{equation}
 The situation is presented in Figure \ref{fig:classIInotation}.

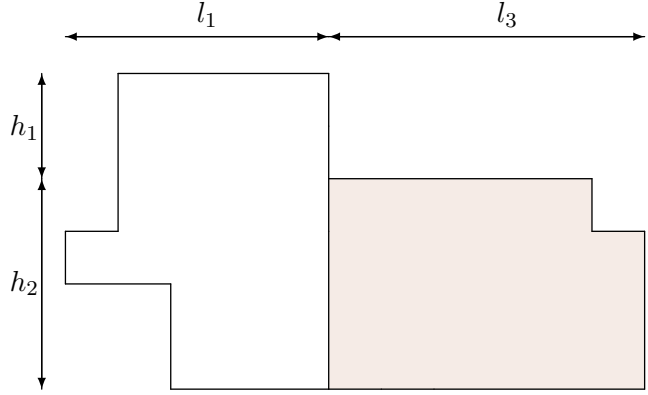
\begin{figure}[h]

\definecolor{zzttqq}{rgb}{0.6,0.2,0.}
\begin{tikzpicture}[line cap=round,line join=round,>=triangle 45,x=0.7cm,y=0.7cm]
\clip(-1.2,-3.5) rectangle (11.2,4.8);
\fill[line width=2.pt,color=zzttqq,fill=zzttqq,fill opacity=0.1] (5.,1.) -- (10.,1.) -- (10.,0.) -- (11.,0.) -- (11.,-3.) -- (5.,-3.) -- cycle;
\draw[line width=0.5pt] (5.,-3.)-- (2.,-3.);
\draw[line width=0.5pt] (2.,-1.)-- (2.,-3.);
\draw[line width=0.5pt] (2.,-1.)-- (0.,-1.);
\draw[line width=0.5pt] (0.,-1.)-- (0.,0.);
\draw[line width=0.5pt] (1.,0.)-- (0.,0.);
\draw[line width=0.5pt] (1.,0.)-- (1.,3.);
\draw[line width=0.5pt] (1.,3.)-- (5.,3.);
\draw[line width=0.5pt] (5.,0.)-- (5.,1.);
\draw[line width=0.5pt] (5.,2.)-- (5.,1.);
\draw[line width=0.5pt] (5.,0.)-- (5.,-3.);
\draw[line width=0.5pt] (5.,1.)-- (10.,1.);
\draw[line width=0.5pt] (10.,1.)-- (10.,0.);
\draw[line width=0.5pt] (10.,0.)-- (11.,0.);
\draw[line width=0.5pt] (11.,0.)-- (11.,-3.);
\draw[line width=0.5pt] (11.,-3.)-- (7.,-3.);
\draw[line width=0.5pt] (6.,-3.)-- (5.,-3.);
\draw[-latex,line width=0.5pt] (-0.45,3.) -- (-0.45,1.);
\draw[-latex,line width=0.5pt] (-0.45,1.) -- (-0.45,3.);
\draw[-latex,line width=0.5pt] (-0.45,1.) -- (-0.45,-3.);
\draw[-latex,line width=0.5pt] (-0.45,-3.) -- (-0.45,1.);
\draw (2.3,4.55) node[anchor=north west] {$l_1$};
\draw (-1.25,2.4) node[anchor=north west] {$h_1$};
\draw (8.,4.55) node[anchor=north west] {$l_3$};
\draw (-1.25,-0.55) node[anchor=north west] {$h_2$};
\draw[line width=0.5pt] (5.,2.)-- (5.,3.);
\draw[-latex,line width=0.5pt] (0.,3.7) -- (5.,3.7);
\draw[-latex,line width=0.5pt] (5.,3.7) -- (0.,3.7);
\draw[-latex,line width=0.5pt] (5.,3.7) -- (11.,3.7);
\draw[-latex,line width=0.5pt] (11.,3.7) -- (5.,3.7);
\draw[line width=0.5pt] (7.,-3.)-- (6.,-3.);
\end{tikzpicture}

\caption{Class $\mathcal{II}$}
\label{fig:classIInotation}
\end{figure}

\subsection{Class $\mathcal{III}$}

We say that an admissible configuration $(A,B)$ belongs to Class~$\mathcal{III}$ if for each $k = 1,...,N_{\rm row}$ we have $n^{\rm row}_k > 0$ and for each $l = 1,...,N_{\rm col}$ we have $n^{\rm col}_l > 0$. In other words, each row and each column of $(A,B)$ contains at least one $A$-point (or equivalently, for every $B$-point there is a $A$-point above it and another one to its left). An example of an optimal configuration in this class can be found in Figure~\ref{fig:sixteenpointexample}. Note that in this example the ratio $N_A / N_B$ is far away from $1$. Indeed, in Section \ref{sec:regularisation} we will show that for $N_A = N_B$ configurations in this class cannot be optimal.

Counting from the left, let $l_1$ denote the number of columns such that $A^{\rm col}_k \neq \emptyset$ and $B^{\rm col}_k = \emptyset$, let $l_2$ denote the number of columns such that $A^{\rm col}_k \neq \emptyset$ and $B^{\rm col}_k \neq \emptyset$, and let $l_3$ be the number of columns such that $A^{\rm col}_k \neq \emptyset$ and $B^{\rm col}_k = \emptyset$. Similarly, counting from the top, denote by $h_1$ the number of rows such that $A^{\rm row}_k \neq \emptyset$ and $B^{\rm row}_k = \emptyset$, let $h_2$ be the number of rows such that $A^{\rm row}_k \neq \emptyset$ and $B^{\rm row}_k \neq \emptyset$, and finally let $h_3$ be the number of rows such that $A^{\rm row}_k \neq \emptyset$ and $B^{\rm row}_k = \emptyset$. Similarly to previous classes, the energy may be expressed as 
\begin{equation}\label{eq: nerg3}
E(A,B) = - 2(N_A + N_B) + (l_1 + l_2 + l_3) + (h_1 + h_2 + h_3) +
 (1-\beta) (l_2 + h_2).
\end{equation}
 The situation is presented in Figure \ref{fig:classIII}. 

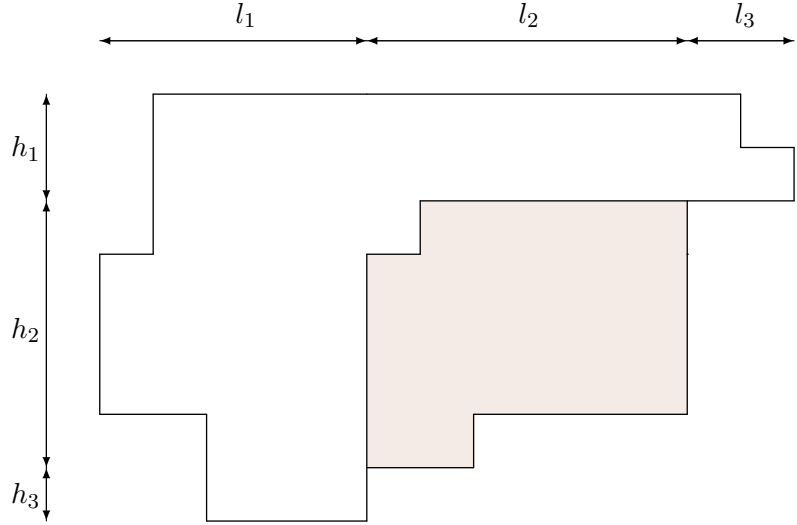
\begin{figure}[h]

\definecolor{zzttqq}{rgb}{0.6,0.2,0.}
\begin{tikzpicture}[line cap=round,line join=round,>=triangle 45,x=0.71cm,y=0.71cm]
\clip(-1.9,-5.1) rectangle (13.1,4.9);
\fill[line width=0.5pt,color=zzttqq,fill=zzttqq,fill opacity=0.1] (6.,1.) -- (11.,1.) -- (11.,-3.) -- (7.,-3.) -- (7.,-4.) -- (5.,-4.) -- (5.,0.) -- (6.,0.) -- cycle;
\draw[line width=0.5pt] (5.,-5.)-- (2.,-5.);
\draw[line width=0.5pt] (2.,-3.)-- (2.,-5.);
\draw[line width=0.5pt] (2.,-3.)-- (0.,-3.);
\draw[line width=0.5pt] (0.,-3.)-- (0.,0.);
\draw[line width=0.5pt] (1.,0.)-- (0.,0.);
\draw[line width=0.5pt] (1.,0.)-- (1.,3.);
\draw[line width=0.5pt] (1.,3.)-- (5.,3.);
\draw[line width=0.5pt] (5.,3.)-- (12.,3.);
\draw[line width=0.5pt] (12.,2.)-- (12.,3.);
\draw[line width=0.5pt] (12.,2.)-- (13.,2.);
\draw[line width=0.5pt] (13.,2.)-- (13.,1.);
\draw[line width=0.5pt] (13.,1.)-- (6.,1.);
\draw[line width=0.5pt] (6.,1.)-- (6.,0.);
\draw[line width=0.5pt] (6.,0.)-- (5.,0.);
\draw[line width=0.5pt] (5.,0.)-- (5.,-5.);
\draw[line width=0.5pt] (11.,1.)-- (11.,0.);
\draw[line width=0.5pt] (11.,0.)-- (11.,0.);
\draw[line width=0.5pt] (11.,0.)-- (11.,-3.);
\draw[line width=0.5pt] (11.,-3.)-- (7.,-3.);
\draw[line width=0.5pt] (7.,-3.)-- (7.,-4.);
\draw[line width=0.5pt] (7.,-4.)-- (5.,-4.);
\draw[-latex,thin] (0.,4.) -- (5.,4.);
\draw[-latex,thin] (5.,4.) -- (0.,4.);
\draw[-latex,thin] (5.,4.) -- (11.,4.);
\draw[-latex,thin] (11.,4.) -- (5.,4.);
\draw[-latex,thin] (11.,4.) -- (13.,4.);
\draw[-latex,thin] (13.,4.) -- (11.,4.);
\draw[-latex,thin] (-1.,3.) -- (-1.,1.);
\draw[-latex,thin] (-1.,1.) -- (-1.,3.);
\draw[-latex,thin] (-1.,1.) -- (-1.,-4.);
\draw[-latex,thin] (-1.,-4.) -- (-1.,1.);
\draw[-latex,thin] (-1.,-4.) -- (-1.,-5.);
\draw[-latex,thin] (-1.,-5.) -- (-1.,-4.);
\draw (2.35,4.9) node[anchor=north west] {$l_1$};
\draw (7.65,4.9) node[anchor=north west] {$l_2$};
\draw (-1.85,2.4) node[anchor=north west] {$h_1$};
\draw (11.7,4.9) node[anchor=north west] {$l_3$};
\draw (-1.85,-1.) node[anchor=north west] {$h_2$};
\draw (-1.85,-4.1) node[anchor=north west] {$h_3$};
\end{tikzpicture}

\caption{Class $\mathcal{III}$}
\label{fig:classIII}
\end{figure}

\subsection{Class $\mathcal{IV}$}
We say that an admissible configuration $(A,B)$ belongs to
Class~$\mathcal{IV}$ if there exist \linebreak $l_1, l_2, h_1, h_2 > 0$ such that $N_{\rm row} + N_{\rm col} - (l_1+l_2+h_1+h_2)>0$ and the following conditions hold: for each $k = 1,...,l_1$ we have $n^{\rm col}_k > 0$ and $m^{\rm col}_k = 0$. For each $k = l_1+1,...,l_1 + l_2$ we have $n^{\rm col}_k > 0$ and $m^{\rm col}_k > 0$. Finally, for all $k = l_1+l_2+1,...,N_{\rm row}$ (this may possibly be empty) we have $n^{\rm col}_k = 0$ and $m^{\rm col}_k > 0$. Similarly, for each $l = 1,...,h_1$ we have $n^{\rm row}_l > 0$ and $m^{\rm row}_l = 0$. For each $l = h_1+1,...,h_1+h_2$ we have $n^{\rm row}_l > 0$ and $m^{\rm row}_l > 0$. Finally, for all $l = h_1+h_2+1,...,N_{\rm col}$ (this may possibly be empty) we have $n^{\rm row}_l = 0$ and $m^{\rm row}_l > 0$. Setting $l_3 = N_{\rm col} - l_1 - l_2$ and $h_3 = N_{\rm row} - h_1 - h_2$ we observe $l_3>0$ or $h_3>0$, i.e., the configuration does not lie in Class $\mathcal{III}$. The energy may be expressed as 
\begin{equation}\label{eq:classIVformula}
E(A,B) = - 2(N_A + N_B) + (l_1 + l_2 + l_3) + (h_1 + h_2 + h_3) +
 (1-\beta)  (l_2 + h_2).
\end{equation} 
The situation is presented in Figure \ref{fig:ClassIVbefore}. Examples of optimal configurations in this class can be found in Figure~\ref{fig:sevenpointexample} (on the right) and in Figure \ref{fig:fivepoints} (both in the middle). 

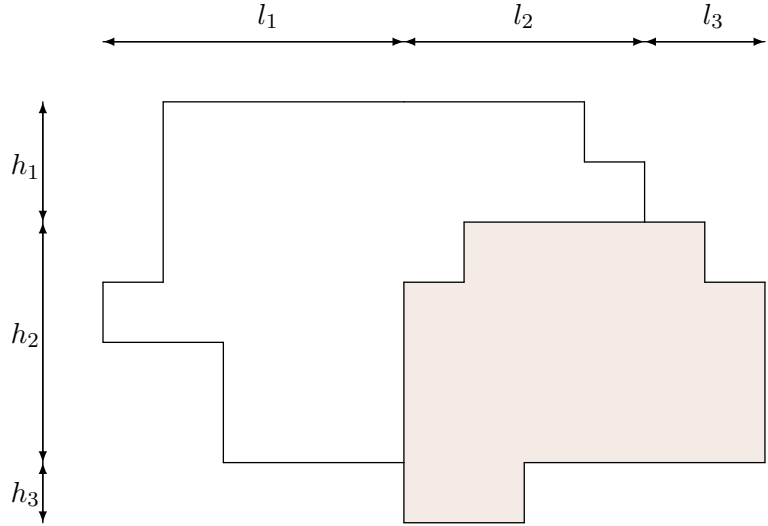
\begin{figure}[h]

\definecolor{zzttqq}{rgb}{0.6,0.2,0.}
\begin{tikzpicture}[line cap=round,line join=round,>=triangle 45,x=0.8cm,y=0.8cm]
\clip(-1.8,-4.05) rectangle (11.2,4.7);
\fill[line width=2.pt,color=zzttqq,fill=zzttqq,fill opacity=0.1] (6.,1.) -- (10.,1.) -- (10.,0.) -- (11.,0.) -- (11.,-3.) -- (7.,-3.) -- (7.,-4.) -- (5.,-4.) -- (5.,0.) -- (6.,0.) -- cycle;
\draw[line width=0.5pt] (5.,-3.)-- (2.,-3.);
\draw[line width=0.5pt] (2.,-1.)-- (2.,-3.);
\draw[line width=0.5pt] (2.,-1.)-- (0.,-1.);
\draw[line width=0.5pt] (0.,-1.)-- (0.,0.);
\draw[line width=0.5pt] (1.,0.)-- (0.,0.);
\draw[line width=0.5pt] (1.,0.)-- (1.,3.);
\draw[line width=0.5pt] (1.,3.)-- (5.,3.);
\draw[line width=0.5pt] (5.,3.)-- (8.,3.);
\draw[line width=0.5pt] (8.,2.)-- (8.,3.);
\draw[line width=0.5pt] (8.,2.)-- (9.,2.);
\draw[line width=0.5pt] (9.,2.)-- (9.,1.);
\draw[line width=0.5pt] (9.,1.)-- (6.,1.);
\draw[line width=0.5pt] (6.,1.)-- (6.,0.);
\draw[line width=0.5pt] (6.,0.)-- (5.,0.);
\draw[line width=0.5pt] (5.,0.)-- (5.,-3.);
\draw[line width=0.5pt] (9.,1.)-- (10.,1.);
\draw[line width=0.5pt] (10.,1.)-- (10.,0.);
\draw[line width=0.5pt] (10.,0.)-- (11.,0.);
\draw[line width=0.5pt] (11.,0.)-- (11.,-3.);
\draw[line width=0.5pt] (11.,-3.)-- (7.,-3.);
\draw[line width=0.5pt] (7.,-3.)-- (7.,-4.);
\draw[line width=0.5pt] (7.,-4.)-- (5.,-4.);
\draw[line width=0.5pt] (5.,-4.)-- (5.,-3.);
\draw[-latex,line width=0.5pt] (0.,4.) -- (5.,4.);
\draw[-latex,line width=0.5pt] (5.,4.) -- (0.,4.);
\draw[-latex,line width=0.5pt] (5.,4.) -- (9.,4.);
\draw[-latex,line width=0.5pt] (9.,4.) -- (5.,4.);
\draw[-latex,line width=0.5pt] (9.,4.) -- (11.,4.);
\draw[-latex,line width=0.5pt] (11.,4.) -- (9.,4.);
\draw[-latex,line width=0.5pt] (-1.,3.) -- (-1.,1.);
\draw[-latex,line width=0.5pt] (-1.,1.) -- (-1.,3.);
\draw[-latex,line width=0.5pt] (-1.,1.) -- (-1.,-3.);
\draw[-latex,line width=0.5pt] (-1.,-3.) -- (-1.,1.);
\draw[-latex,line width=0.5pt] (-1.,-3.) -- (-1.,-4.);
\draw[-latex,line width=0.5pt] (-1.,-4.) -- (-1.,-3.);
\draw (2.4,4.8) node[anchor=north west] {$l_1$};
\draw (6.65,4.8) node[anchor=north west] {$l_2$};
\draw (-1.7,2.3) node[anchor=north west] {$h_1$};
\draw (9.8,4.8) node[anchor=north west] {$l_3$};
\draw (-1.7,-0.5) node[anchor=north west] {$h_2$};
\draw (-1.7,-3.15) node[anchor=north west] {$h_3$};
\end{tikzpicture}

\caption{Class $\mathcal{IV}$}
\label{fig:ClassIVbefore}
\end{figure}

\subsection{Class $\mathcal{V}$}

We say that an admissible configuration $(A,B)$ belongs to Class~$\mathcal{V}$ if there exist $l_1, l_2, l_3, h_1, h_2, h_3 > 0$ such that $l_1 + l_2 + l_3 = N_{\rm col}$, $h_1+h_2+h_3 = N_{\rm row}$ and the following conditions hold: for each $k = 1,...,l_1$ we have $n^{\rm col}_k > 0$ and $m^{\rm col}_k = 0$. For each $k = l_1+1,...,l_1 + l_2$ we have $n^{\rm col}_k > 0$ and $m^{\rm col}_k > 0$. Finally, for all $k = l_1+l_2+1,...,N_{\rm row}$  we have $n^{\rm col}_k > 0$ and $m^{\rm col}_k = 0$. On the other hand, for each $l = 1,...,h_1$ we have $n^{\rm row}_l > 0$ and $m^{\rm row}_l = 0$. For each $l = h_1+1,...,h_1+h_2$ we have $n^{\rm row}_l > 0$ and $m^{\rm row}_l > 0$. Finally, for all $l = h_1+h_2+1,...,N_{\rm col}$ we have $n^{\rm row}_l = 0$ and $m^{\rm row}_l > 0$. The energy may be expressed as 
\begin{equation*}
E(A,B) = - 2(N_A + N_B) + (l_1 + l_2 + l_3) + (h_1 + h_2 + h_3) +  (1-\beta) (l_2 + h_2).
\end{equation*} 
The situation is presented in Figure \ref{fig:classVbefore}.


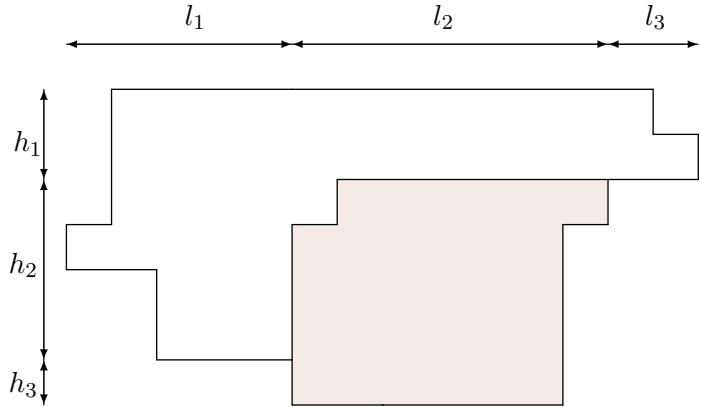
\begin{figure}[h]

\definecolor{zzttqq}{rgb}{0.6,0.2,0.}
\begin{tikzpicture}[line cap=round,line join=round,>=triangle 45,x=0.6cm,y=0.6cm]
\clip(-1.5,-4.9) rectangle (14.1,5.2);
\fill[line width=0.5pt,color=zzttqq,fill=zzttqq,fill opacity=0.1] (6.,1.) -- (12.,1.) -- (12.,0.) -- (11.,0.) -- (11.,-4.) -- (5.,-4.) -- (5.,0.) -- (6.,0.) -- cycle;
\draw[line width=0.5pt] (5.,-3.)-- (2.,-3.);
\draw[line width=0.5pt] (2.,-1.)-- (2.,-3.);
\draw[line width=0.5pt] (2.,-1.)-- (0.,-1.);
\draw[line width=0.5pt] (0.,-1.)-- (0.,0.);
\draw[line width=0.5pt] (1.,0.)-- (0.,0.);
\draw[line width=0.5pt] (1.,0.)-- (1.,3.);
\draw[line width=0.5pt] (1.,3.)-- (5.,3.);
\draw[line width=0.5pt] (5.,3.)-- (13.,3.);
\draw[line width=0.5pt] (13.,2.)-- (13.,3.);
\draw[line width=0.5pt] (13.,2.)-- (14.,2.);
\draw[line width=0.5pt] (14.,2.)-- (14.,1.);
\draw[line width=0.5pt] (14.,1.)-- (6.,1.);
\draw[line width=0.5pt] (6.,1.)-- (6.,0.);
\draw[line width=0.5pt] (6.,0.)-- (5.,0.);
\draw[line width=0.5pt] (5.,0.)-- (5.,-3.);
\draw[line width=0.5pt] (12.,1.)-- (12.,0.);
\draw[line width=0.5pt] (12.,0.)-- (11.,0.);
\draw[line width=0.5pt] (11.,0.)-- (11.,-4.);
\draw[line width=0.5pt] (11.,-4.)-- (7.,-4.);
\draw[line width=0.5pt] (7.,-4.)-- (7.,-4.);
\draw[line width=0.5pt] (7.,-4.)-- (5.,-4.);
\draw[line width=0.5pt] (5.,-4.)-- (5.,-3.);
\draw[-latex,thin] (0.,4.) -- (5.,4.);
\draw[-latex,thin] (5.,4.) -- (0.,4.);
\draw[-latex,thin] (5.,4.) -- (12.,4.);
\draw[-latex,thin] (12.,4.) -- (5.,4.);
\draw[-latex,thin] (12.,4.) -- (14.,4.);
\draw[-latex,thin] (14.,4.) -- (12.,4.);
\draw[-latex,thin] (-0.5,3.) -- (-0.5,1.);
\draw[-latex,thin] (-0.5,1.) -- (-0.5,3.);
\draw[-latex,thin] (-0.5,1.) -- (-0.5,-3.);
\draw[-latex,thin] (-0.5,-3.) -- (-0.5,1.);
\draw[-latex,thin] (-0.5,-3.) -- (-0.5,-4.);
\draw[-latex,thin] (-0.5,-4.) -- (-0.5,-3.);
\draw (2.4,5.1) node[anchor=north west] {$l_1$};
\draw (7.9,5.1) node[anchor=north west] {$l_2$};
\draw (-1.4,2.3) node[anchor=north west] {$h_1$};
\draw (12.6,5.1) node[anchor=north west] {$l_3$};
\draw (-1.5,-0.4) node[anchor=north west] {$h_2$};
\draw (-1.5,-3.05) node[anchor=north west] {$h_3$};
\end{tikzpicture}

\caption{Class $\mathcal{V}$}
\label{fig:classVbefore}
\end{figure}

We close this section with the observation that the five classes cover all possible cases up to isometries, reflections,  and changing roles of the types.

\section{Analysis of Class $\mathcal{I}$}\label{sec:reg1}

\subsection{Regularisation inside Class $\mathcal{I}$}

The goal of this section is to make the configuration  in Class~$\mathcal{I}$  more regular without increasing the energy. This regularisation will facilitate the computation of the minimal energy. We keep the notation as in the previous section, and begin with the following observation.

\begin{proposition}\label{prop:classIregularisationstep1}
Fix $N_A, N_B > 0$ and $ \beta  \in (0,1)$. Suppose that $(A,B) \in \mathcal{I}$ is an optimal configuration. Then, we either have $l_2 = 0$ or $l_2 = 1$.
\end{proposition}

Both cases can happen: take $N_A = N_B = 3$ and $ \beta  \in (0,1)$. Then, there are two optimal configurations, one with $l_2 = 0$ and the other one with $l_2 = 1$, see Figure \ref{fig:sixpointexample}.

\begin{proof}
The idea of the proof is the following: we  suppose by contradiction that $l_2 \geq 2$. We add more points to the configuration $(A,B)$, so that it becomes a full rectangle, keeping track of the change of the energy in the process. Then, we exchange a number of points, making the interface shorter and causing a drop in the energy. Finally, we remove the added points, again keeping track of the energy. This yields strictly smaller total energy, a contradiction.  The argument is presented in Figure \ref{fig:regularisationClassI}. 

To be exact, let us modify the configuration $(A,B)$ as follows. We add $N'_A$ $A$-points on the left and $N'_B$ $B$-points on the right such that  that $(A,B)$ becomes a full rectangle with sides $l_1 + l_2 + l_3$ and $h$. Notice that in this way we do not alter the surface energy. Meanwhile, the bulk energy changes by $- 2 (N_A' + N_B')$. Now, look at the rectangle in the middle with sides $l_2$ and $h$. If we exchange $A$-points from its rightmost column and $B$-points from its leftmost column (as many as we can), we will make one column (or two) full of points of one type. Hence, in the formula for the  energy, see \eqref{eq:formulaforenergyclassI},  we replace $l_2$ by $l_2 - 1$ (respectively $l_2 - 2$), and $l_1 + l_3$ by $l_1 + l_3 + 1$ (respectively $l_1 + l_3 + 2$). This causes a drop in the surface energy by  $ (1-\beta)$ or $2(1-\beta)$. 

Finally, we take care of the added points. We remove $N_A'$ $A$-points, starting from the leftmost column, going from top to bottom. In the process, the surface energy decreases or remains the same (since $l_1$ may decrease or remain the same). Similarly, we remove $N_B'$ $B$-points, starting from the rightmost column and going from top to bottom. 

In this way, we have obtained a configuration $(\hat{A},\hat{B})$ with
the same number of $A$- and $B$-points as $(A,B)$, but with energy
lower at least by $   (1-\beta) $. After this operation, we possibly end up with a shape of the interface   different  from the one in Class $\mathcal{I}$, but this does not matter since we only wanted to show that $(A,B)$ was not optimal. Hence, if $(A,B)$ is an optimal configuration, then $l_2 = 0$ or $l_2 = 1$.
\end{proof}

\begin{figure}[h]

\definecolor{zzttqq}{rgb}{0.6,0.2,0.}
\begin{tikzpicture}[line cap=round,line join=round,>=triangle 45,x=0.5cm,y=0.5cm]
\clip(-1.4,-9.5) rectangle (23.5,7.1);
\fill[line width=0.5pt,color=zzttqq,fill=zzttqq,fill opacity=0.1] (5.,6.) -- (8.,6.) -- (8.,4.) -- (9.,4.) -- (9.,2.) -- (8.,2.) -- (8.,0.) -- (3.,0.) -- (3.,1.) -- (4.,1.) -- (4.,3.) -- (5.,3.) -- cycle;
\fill[line width=0.5pt,color=zzttqq,fill=zzttqq,fill opacity=0.1] (19.,6.) -- (23.,6.) -- (23.,0.) -- (17.,0.) -- (17.,1.) -- (18.,1.) -- (18.,3.) -- (19.,3.) -- cycle;
\fill[line width=0.5pt,color=zzttqq,fill=zzttqq,fill opacity=0.1] (5.,-3.) -- (9.,-3.) -- (9.,-9.) -- (4.,-9.) -- (4.,-5.) -- (5.,-5.) -- cycle;
\fill[line width=0.5pt,color=zzttqq,fill=zzttqq,fill opacity=0.1] (19.,-3.) -- (22.,-3.) -- (22.,-7.) -- (23.,-7.) -- (23.,-9.) -- (18.,-9.) -- (18.,-5.) -- (19.,-5.) -- cycle;
\draw[line width=0.5pt] (0.,0.)-- (0.,4.);
\draw[line width=0.5pt] (0.,4.)-- (1.,4.);
\draw[line width=0.5pt] (1.,4.)-- (1.,5.);
\draw[line width=0.5pt] (1.,5.)-- (2.,5.);
\draw[line width=0.5pt] (2.,5.)-- (2.,6.);
\draw[line width=0.5pt] (2.,6.)-- (5.,6.);
\draw[line width=0.5pt] (5.,6.)-- (5.,3.);
\draw[line width=0.5pt] (5.,3.)-- (4.,3.);
\draw[line width=0.5pt] (4.,3.)-- (4.,1.);
\draw[line width=0.5pt] (4.,1.)-- (3.,1.);
\draw[line width=0.5pt] (3.,1.)-- (3.,0.);
\draw[line width=0.5pt] (3.,0.)-- (0.,0.);
\draw[line width=0.5pt] (5.,6.)-- (8.,6.);
\draw[line width=0.5pt] (8.,6.)-- (8.,4.);
\draw[line width=0.5pt] (8.,4.)-- (9.,4.);
\draw[line width=0.5pt] (9.,4.)-- (9.,2.);
\draw[line width=0.5pt] (8.,2.)-- (9.,2.);
\draw[line width=0.5pt] (8.,2.)-- (8.,1.);
\draw[line width=0.5pt] (8.,1.)-- (8.,0.);
\draw[line width=0.5pt] (8.,0.)-- (7.,0.);
\draw[line width=0.5pt] (7.,0.)-- (3.,0.);
\draw[line width=0.5pt] (14.,6.)-- (19.,6.);
\draw[line width=0.5pt] (19.,6.)-- (19.,3.);
\draw[line width=0.5pt] (19.,3.)-- (18.,3.);
\draw[line width=0.5pt] (18.,3.)-- (18.,1.);
\draw[line width=0.5pt] (18.,1.)-- (17.,1.);
\draw[line width=0.5pt] (17.,1.)-- (17.,0.);
\draw[line width=0.5pt] (17.,0.)-- (14.,0.);
\draw[line width=0.5pt] (14.,0.)-- (14.,6.);
\draw[line width=0.5pt] (19.,6.)-- (23.,6.);
\draw[line width=0.5pt] (23.,6.)-- (23.,0.);
\draw[line width=0.5pt] (23.,0.)-- (17.,0.);
\draw[line width=0.5pt] (0.,-3.)-- (5.,-3.);
\draw[line width=0.5pt] (5.,-3.)-- (5.,-5.);
\draw[line width=0.5pt] (5.,-5.)-- (4.,-5.);
\draw[line width=0.5pt] (4.,-5.)-- (4.,-9.);
\draw[line width=0.5pt] (4.,-9.)-- (0.,-9.);
\draw[line width=0.5pt] (0.,-9.)-- (0.,-3.);
\draw[line width=0.5pt] (5.,-3.)-- (9.,-3.);
\draw[line width=0.5pt] (9.,-3.)-- (9.,-9.);
\draw[line width=0.5pt] (9.,-9.)-- (4.,-9.);
\draw[line width=0.5pt] (15.,-3.)-- (15.,-6.);
\draw[line width=0.5pt] (14.,-6.)-- (15.,-6.);
\draw[line width=0.5pt] (14.,-6.)-- (14.,-9.);
\draw[line width=0.5pt] (14.,-9.)-- (18.,-9.);
\draw[line width=0.5pt] (18.,-9.)-- (18.,-5.);
\draw[line width=0.5pt] (18.,-5.)-- (19.,-5.);
\draw[line width=0.5pt] (19.,-5.)-- (19.,-3.);
\draw[line width=0.5pt] (19.,-3.)-- (15.,-3.);
\draw[line width=0.5pt] (19.,-3.)-- (22.,-3.);
\draw[line width=0.5pt] (22.,-3.)-- (22.,-7.);
\draw[line width=0.5pt] (22.,-7.)-- (23.,-7.);
\draw[line width=0.5pt] (23.,-7.)-- (23.,-9.);
\draw[line width=0.5pt] (23.,-9.)-- (18.,-9.);
\draw[-latex,line width=0.5pt] (10.,3.) -- (13.,3.);
\draw[-latex,line width=0.5pt] (13.,-0.5) -- (10.,-2.5);
\draw[-latex,line width=0.5pt] (10.,-6.) -- (13.,-6.);
\end{tikzpicture}

\caption{Regularisation of Class $\mathcal{I}$}
\label{fig:regularisationClassI}
\end{figure}
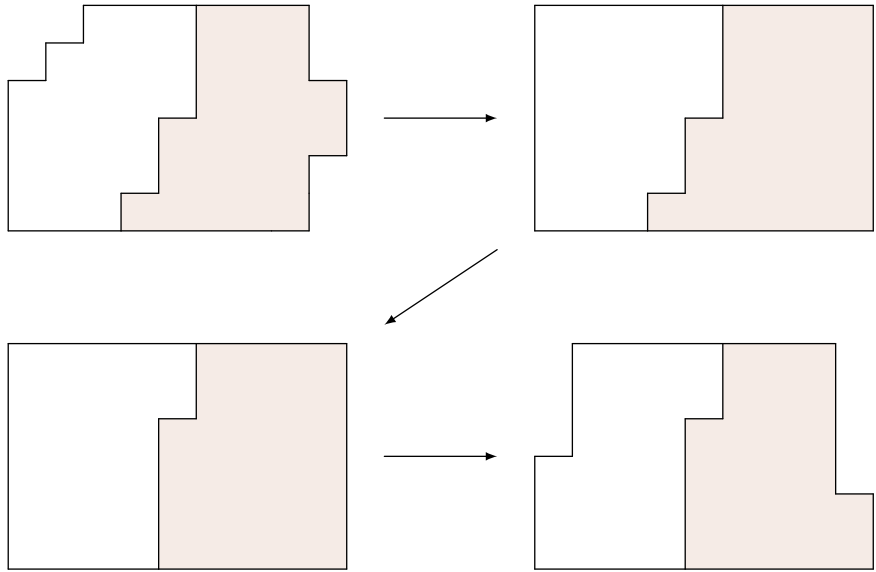

By performing the modification described in the proof, we get that we may assume that the configuration is as compact as possible: given $h$, the values of $l_1$ and $l_3$ are as small as possible, and all the columns except for the leftmost and rightmost ones are full (i.e., have $h$ points).  This is a property that we will use several times in the sequel. 

 We provide an exact formula for the minimal energy in Theorem
\ref{thm:classIexact}. This requires fixing  $N_A = N_B$, which will
be assumed throughout. Note however that some of the intermediate lemmas below
may be adapted for the case $N_A \not
= N_B$, as well.   Let us first prove that we may assume that $l_2 = 0$. To this end, let us first state the following technical lemma.

\begin{lemma}\label{lem:classIstructurelemma}
  Fix  $N := N_A = N_B>0$  and $\beta \in (0,1)$ . Suppose that $(A,B) \in \mathcal{I}$ is an optimal configuration such that $l_1 = l_3$ and $l_2 = 1$. Then, we have $l_1 = l_3 \geq h/2$.
\end{lemma}

\begin{proof}
Without restriction we assume that $(A,B)$ has the form described
before the statement of the lemma, see also the last picture in Figure
\ref{fig:regularisationClassI}. Let  $k = \lceil h/2 \rceil$. Suppose by contradiction that the statement does not hold, i.e., $l_1 = l_3 < k$ (in particular, $k \geq 2$). 

 Consider two cases: first, assume that $h$ is even, so that $h = 2k$. Then, the whole configuration fits into a rectangle with height $2k$ and width $2l_1 + 1$, where $l_1 \leq k-1$. Let us rearrange all the points so that the resulting configuration lies in a rectangle with height $2k-1$ and width $2l_1+2$.   We place the points by filling the columns from left to right, first with $A$-points and then with $B$-points, so that the resulting configuration lies in Class $\mathcal{I}$ and has $l_2 \leq 1$.  In fact,   all points may be placed in this rectangle since the assumption $l_1 \leq k-1$ implies
$$ (2k-1)(2l_1+2) \geq 2k(2l_1+1).$$ 
 But then the new configuration has strictly smaller energy since $h$ decreased by $1$, $l_2 \leq 1$, and $l_1+l_3$ grew by at most $1$. Hence, the original configuration was not optimal, a contradiction.

In the second case, $h$ is odd, so that $h = 2k-1$. Then, the whole configuration fits into a rectangle with height $2k -1$ and width $2l_1 + 1$, where $l_1 \leq k-1$. Let us again rearrange all the points using the procedure from the previous paragraph, so that the resulting configuration lies in a rectangle with height $2k-2$ and width $2l_1+2$ and satisfies $l_2 \leq 1$.  Indeed, if $l_1 \leq k-2$, all  points may be placed in this rectangle since  in this case we have 
\begin{align}\label{inequili}
 (2k-2)(2l_1+2) \geq (2k-1)(2l_1+1).
 \end{align}
  On the other hand, if  $l_1 = k-1$, we have
$$(2k-2)(2l_1+2) =   (2k-1)(2l_1+1)  - 1,$$
so the  inequality \eqref{inequili}  is not satisfied. In this case,
however,  $(2k-1)(2l_1+1)$ is  odd. Thus,  since the total number of
points $ 2 N $ is even, it is not possible that the entire rectangle  with height $2k$ and width $2l_1 + 1$  was full in the original configuration. Therefore, we can   still  place  all the points in the  rectangle  with height $2k-2$ and width $2l_1+2$.  As before, the new configuration has strictly smaller energy since $h$ decreased by $1$, $l_2 \leq 1$, and $l_1+l_3$ grew by at most $1$: a contradiction.
\end{proof}

%
%

Now, we proceed to prove the main result for Class $\mathcal{I}$, namely that for the purpose of the computation of the minimal energy we may assume that $l_2 = 0$.

\begin{proposition}\label{prop:classIregularisationstep2}
 Fix $N := N_A = N_B>0$ and  $ \beta  \in (0,1)$. Then, if $(A,B) \in \mathcal{I}$ is an optimal configuration, then there exists an optimal configuration $(\hat{A},\hat{B}) \in \mathcal{I}$ with $l_2 = 0$.
\end{proposition}

\begin{proof}
If $(A,B) \in \mathcal{I}$ is such that $l_2 = 0$, there is nothing  to prove. Suppose to the contrary that $l_2 > 0$. Then, by Proposition \ref{prop:classIregularisationstep1} we have that $l_2 = 1$. We introduce the following notation: again, $l_1$ is the number of columns with only $A$-points and $l_3$ is the number of columns with only $B$-points. We can assume that all columns except for the leftmost and rightmost ones are full, cf.\ last picture in Figure \ref{fig:regularisationClassI}. By $r_1 \in \{ 1,...,h \}$ we denote the number of  $A$-points in the leftmost column, and $r_4 \in \{ 1,...,h \}$ is the number of $B$-points in the rightmost column. By $r_2,r_3 \in \{ 1,...,h-1 \}$ we denote the numbers of $A$- and $B$-points, respectively, in the single column which contains points of both types.

Since $N_A = N_B$, we compute the number of points of each type and we get
\begin{equation*}
(l_1 - 1) h + r_1 + r_2 = (l_3 - 1) h + r_3 + r_4,
\end{equation*}
so
\begin{equation}\label{eq: numbering}
(l_1 - l_3) h = r_3 + r_4 - r_1 - r_2.
\end{equation}
Due to the range of $r_1,\ldots,r_4$, the left-hand side can take only values between $-2h+3$ and $2h-3$, so it needs to take values in the set $\{ -h,0,h\}$. Hence, up to exchanging the roles of the two types, we either have $l_1 = l_3$ or $l_1 = l_3 + 1$.

First, suppose that $l_1 = l_3 + 1$. Then, by \eqref{eq: numbering} we have $r_1 + r_2 + h = r_3 + r_4$. In particular, $r_1 + r_2 < h$ as $r_3+r_4 \le 2h-1$. Hence, we may move the $r_2$ $A$-points from the single column with both types to the leftmost column, and replace them by $r_2$ $B$-points from the rightmost column. In this way, the double-type column disappeared altogether. This process strictly decreases  the energy  \eqref{eq:formulaforenergyclassI}  since $l_1$ stays the same, $l_2$ decreases by $1$, and $l_3$ increases by $1$ or stays the same. This is a contradiction.
 
Now, suppose that $l_1 = l_3$. Then, by \eqref{eq: numbering} we have $r_1 + r_2 = r_3 + r_4$. If $r_1 + r_2 \leq h$, we proceed as in the previous paragraph. Suppose otherwise, i.e., $r_1 + r_2 = r_3+r_4  > h$. Without restriction we can suppose that $r_3 \ge r_2$. Let $k \in \mathbb{N}$ such that $k = \lceil h/2 \rceil$. Notice that we may modify the configuration so that $r_2 = \lfloor h/2 \rfloor$ and $r_3 = k$.  Indeed, otherwise  we move $\lfloor h/2 \rfloor - r_2$ ($=r_3 - k$) $B$-points from the double-type column to the rightmost column and move $\lfloor h/2 \rfloor - r_2$ $A$-points from the leftmost column to the double-type column, so that both types have $\lfloor h/2 \rfloor$ and $k$ points, respectively, in the double-type column. In this way, since 
$$r_4 + \lfloor h/2 \rfloor - r_2 = (r_1+r_2-r_3) + \lfloor h/2 \rfloor - r_2 = r_1 + \lfloor h/2 \rfloor -r_3 \le h,$$
where we used $r_1 \le h$ and $r_3 \ge \lfloor h/2 \rfloor$, we did not add any additional column on the right. Thus, the total energy did not   increase.   

%

As $l_1=l_3$ and $l_2 = 1$, by Lemma \ref{lem:classIstructurelemma} we have that $l_1 = l_3 \geq k$. Now, remove all the points in the double-type column and place them directly above the first row, $ \lfloor h/2 \rfloor  $ $A$-points directly above the $l_1$ $A$-points (starting from the right) and $k$ $B$-points directly above the $l_3$ $B$-points (starting from the left). Finally, we merge the two connected components of the resulting configuration by moving the connected component on the left by $(1,0)$. In this way, $h$ increased by 1, $l_2$ decreased by 1, and $l_1$ and $l_3$ remain unchanged, so  that  the energy remains the same, see  \eqref{eq:formulaforenergyclassI}.  Hence, the resulting configuration $(\hat{A},\hat{B})$ is minimal, lies in Class $\mathcal{I}$, and satisfies $l_2 = 0$. This concludes the proof.
\end{proof}


\subsection{Exact calculation for Class $\mathcal{I}$}

The regularisation procedure presented in the previous subsection
enables us to compute directly  the minimal energy for configurations
in Class $\mathcal{I}$ for any $ \beta  \in (0,1)$. In this subsection, we suppose that $N_A =
N_B$ and denote the common value by  $N$.   Later, in Section
\ref{sec:regularisation} we will show that there exists always a
minimiser in Class~$\mathcal{I}$ which induces that the energy
computed below coincides with the minimal energy.  

\begin{theorem}\label{thm:classIexact}
Fix $ N  := N_A = N_B>0$ and  $ \beta  \in (0,1)$. Suppose that a minimal configuration $(A,B)$ is in Class $\mathcal{I}$. Then, its energy is equal to  
$$-4N + \min_{h \in \mathbb{N}} \big(2\llceil N/h\rrceil + h( 2 - \beta )\big), $$
where all minimisers $h$ satisfy $|h- \sqrt{2  N /(2-\beta)}| \le C_\beta N^{1/4}$ for some  constant $C_\beta$      only  depending on $\beta$. For $\beta \in \mathbb{R} \setminus \mathbb{Q}$, there exists a unique minimiser.  
\end{theorem}

\begin{proof}
By Proposition \ref{prop:classIregularisationstep2}, for the purpose of the computation of the minimal energy, we may assume that $l_2 = 0$. Hence, we also have $l_1 = l_3$, and denote the common value by $\ell$. Notice that we may minimise the energy under the constraint
$${ h,\,\ell \in \mathbb{N}, \quad   N  =h \ell +  r  \quad \text{with}\
r\in \mathbb{N}, \ 0\leq r\leq h-1.}$$
This constraint is natural since for fixed $h$, the length $\ell$ is minimal whenever all the columns except for the leftmost and rightmost ones are full (i.e., have $h$ points).  We also refer to the configuration given in Theorem~\ref{thm:main}.v.  Under these assumptions, we may rewrite the energy \eqref{eq:formulaforenergyclassI} as
\begin{equation*}
E(A,B) = -4  N  + 2 (\ell+\min\lbrace r,1 \rbrace)  + h( 2- \beta ). 
\end{equation*}
In particular, one can express  the energy solely in terms of $h\in \mathbb{N}$ as 
\begin{equation}\label{eq:E}
E(h):=-4  N + 2\llceil \frac{N}{h}\rrceil + h( 2 - \beta ).
\end{equation}
It is clear that the minimum of $E$ over $\mathbb{N}$ is unique in
case that $\beta \in \mathbb{R} \setminus \mathbb{Q}$ as $E(h_1) -
E(h_2) \notin \mathbb{Q}$ for all  $h_1 ,\,h_2\in {\mathbb N}$
with  $h_1 \neq h_2$. 

It remains to check that minimisers $h$ satisfy $|h- \sqrt{2  N
  /(2-\beta)}| \le C_\beta N^{1/4}$ for some  constant $C_\beta$.  The
function $\bar{E}(h) =  -4  N  + 2 N /h + h ( 2 - \beta ) $ is
strictly convex and attains its minimum  at   $\bar{h}:=\sqrt{2  N/( 2 - \beta )}$ with $\bar{E}(\bar{h}) = - 4N  +  2\sqrt{2  N( 2 - \beta )}$. For $h_* = \lceil\bar{h} \rceil$ we get for $\beta \in (0,1)$
\begin{align}\label{eq: LLLLL}
E(h_*) &= -4  N + 2\llceil \frac{N}{h_*}\rrceil + h_*( 2 - \beta ) \le -4  N + 2  \frac{N}{h_*} + 2 + h_*( 2 - \beta ) \notag \\ 
&  \le  -4  N + 2  \frac{N}{\bar{h}}  + \bar{h}( 2 - \beta ) + 4 = \bar{E}(\bar{h}) + 4.
\end{align}
 Let us now determine those $h\in {\mathbb N}$ such that the
inequality $\bar{E}(h) \le \bar{E}(\bar{h}) + 4$ holds. 
By determining the roots of the quadratic equation $h\bar{E}(h) = h(\bar{E}(\bar{h}) + 4)$, one can check that $\bar{E}(h) \le \bar{E}(\bar{h}) + 4$ is equivalent to
$$h \in I_{N,\beta} :=  \frac{2}{2-\beta} + \sqrt{2N/(2-\beta)}  + \frac{2}{2-\beta} \Big[ - \sqrt{ 1 +\sqrt{2N(2-\beta)}}, \sqrt{ 1 +\sqrt{2N(2-\beta)}}    \Big].  $$
Note that $h \notin I_{N,\beta}$  cannot be a minimiser of $E$ since then by \eqref{eq: LLLLL} we have $E(h) \ge \bar{E}(h) > \bar{E}(\bar{h}) + 4 \ge E(h_*)$. Clearly, the definition of $I_{N,\beta}$ implies $|h- \sqrt{2  N /(2-\beta)}| \le C_\beta N^{1/4}$ for all $h \in I_{N,\beta}$ for some $C_\beta$ sufficiently large. This concludes the proof. 
 \end{proof}

 We close this section with the observation that,  once we have guaranteed the existence of a  minimiser in Class~$\mathcal{I}$ (see Theorem \ref{thm:classIexistence} below), Theorem \ref{thm:main}.iv follows from Theorem \ref{thm:classIexact} and \eqref{eq:eq}. The construction of the configuration in the previous proof  also yields the explicit solution in Theorem \ref{thm:main}.v.

\section{Analysis and regularisation of other classes}\label{sec:regularisation}\label{sec:reg2}

In this section, we show how to regularise configurations related to
classes $\mathcal{II}$--$\mathcal{V}$. Our main goal is to show that
for $N_A = N_B$, it is not possible that  a   minimiser lies in Class~$\mathcal{II}$ or Class~$\mathcal{III}$. While it is possible that a
minimiser lies in Class~$\mathcal{IV}$,  see Proposition \ref{prop:largeminimisersiv} below,   we will show that under the
constraint $ \beta \leq 1/2$ we can modify  an  optimal configuration
so that it lies in Class~$\mathcal{I}$.


\subsection{Class $\mathcal{II}$}

Since the definition of Class $\mathcal{II}$ already involved a very regular interface, namely a straight line, the situation here is much simpler with respect to Class~$\mathcal{I}$. In fact, the whole analysis of the problem boils down to the following simple result.

\begin{proposition}\label{prop:classIIregularisation}
 Fix $N:= N_A = N_B>0$  and $ \beta  \in (0,1)$. If $(A,B)$ is an optimal configuration, then $(A,B) \notin \mathcal{II}$.
\end{proposition}

\begin{proof}
Suppose otherwise. Then, recalling \eqref{eq: energy,class2}, notice that we may rewrite the energy as   
\begin{equation*}
E(A,B) = -4  N  + E_A + E_B  -  \beta  h_2,    
\end{equation*} 
where $E_A = l_1 + h_1 + h_2$ and $E_B = l_3 + h_2$ are the energy between the void and $A$ and $B$, respectively, and the last term corresponds to the interface energy. 

Suppose first that $E_A > E_B$. Then, we modify the configuration as follows: set $\hat{B} = B$ and let $\hat{A}$ be the symmetric image of $B$ under the reflection along the interface. In this way, we obtain 
\begin{equation*}
E(\hat{A},\hat{B}) = -4  N + 2E_B   -  \beta  h_2 < -4  N  + E_A + E_B   -  \beta  h_2 = E(A,B),
\end{equation*} 
a contradiction to minimality of $(A,B)$. Now, we suppose  $E_A \leq E_B$ instead. We modify the configuration as follows: set $\hat{A} = A$ and let $\hat{B}$ be the symmetric image of $A$ under the reflection along the interface. In this way,
the part of the energy corresponding to the shape of $A$ stays the
same, the part corresponding to $B$ drops or stays the same, and the
length $h_2$ of the interface increases at least by $1$. Hence, the
total energy decreases, so $(A,B)$ was not a minimal configuration.

 Let us remark that Proposition
\ref{prop:classIIregularisation} hoes not hold if $N_A \not = N_B$, a
counterexample  being  provided by Figure \ref{fig:sevenpointexample}. 
\end{proof}

\subsection{Class $\mathcal{III}$}

Using again the notation introduced in the previous section, our first
goal is to show that we can modify an admissible configuration in
Class~$\mathcal{III}$ such that we remain in Class~$\mathcal{III}$ and
$l_3 = h_3 = 0$ without increasing the energy. Then, we will prove
that such a configuration cannot be optimal if $N_A = N_B$.

\begin{proposition}\label{prop:classIIIregularisationstep1}
Fix $N_A, N_B > 0$ and $ \beta  \in (0,1)$.  Suppose that $(A,B) \in \mathcal{III}$ is a minimal configuration and $l_3 > 0$ (respectively $h_3 > 0$). Then, there exists a minimal configuration $(\hat{A},\hat{B}) \in \mathcal{III}$ with $l_3 = 0$ (respectively $h_3 = 0$). 
\end{proposition}

\begin{proof}
Assume that $l_3 > 0$ (the proof in the case $h_3 > 0$ is analogous). Our  construction is presented in Figure~\ref{fig:classIIIb}. We will modify the top $h_1$ rows of the configuration $(A,B)$ in the following way:  for every $1 \leq k \leq N_{\rm row}$, denote by $x_k$ the first coordinate in the rightmost point of $(A \cup B)_k^{\rm row}$. Then, for $k \leq h_1$, we set $\hat{A}_k := A^{\rm row}_k +  (\min\{x_{h_1+1} - x_k,0\},0)$, i.e., each row which has points further to the right than the rightmost point of $B_{h_1 + 1}$ is translated to the left, in such a way that its rightmost  point  aligns with the rightmost point of $B_{h_1 + 1}$. As we made no modifications inside rows,  $E^{\rm row}_k(\hat{A},\hat{B}) = E^{\rm row}_k(A,B)$ for all $k = 1,\ldots,N_{\rm row}$,  see \eqref{eq: row energy}.  Regarding $E_k^{\rm  inter }$, observe that for $k \geq h_1 + 1$ nothing changed in the configuration, so $E_k^{\rm  inter }(\hat{A},\hat{B}) = E_k^{\rm  inter }(A,B)$.  On the other hand, for $k < h_1$, we either left two adjacent rows intact (so the number of connections between them stayed the same); moved both of them to the left so that their rightmost points align (so the number of connections between them stayed the same or increased); or moved only one of them to the left, but because the rightmost point of the other one has first coordinate smaller  or equal to   the first coordinate of $B_{h_1+1}$, this shift did not destroy any bonds and possibly created new ones. In every case, all these connections are of type $A$-$A$, so we have $E_k^{\rm  inter }(\hat{A},\hat{B}) \leq E_k^{\rm  inter }(A,B)$.  Finally, for $k = h_1$, we did not change the number of $A$-$B$ connections and  possibly added some $A$-$A$ connections.  Thus,  $E_k^{\rm  inter }(\hat{A},\hat{B}) \leq E_k^{\rm  inter }(A,B)$.  
 Note that  after this  procedure all columns $A_l^{\rm col}$ for $l \leq l_1$ are still connected, as otherwise this  would  contradict Theorem \ref{thm:connected} and the  minimality of the original configuration. Hence, the resulting configuration lies in Class $\mathcal{III}$. 
\end{proof}

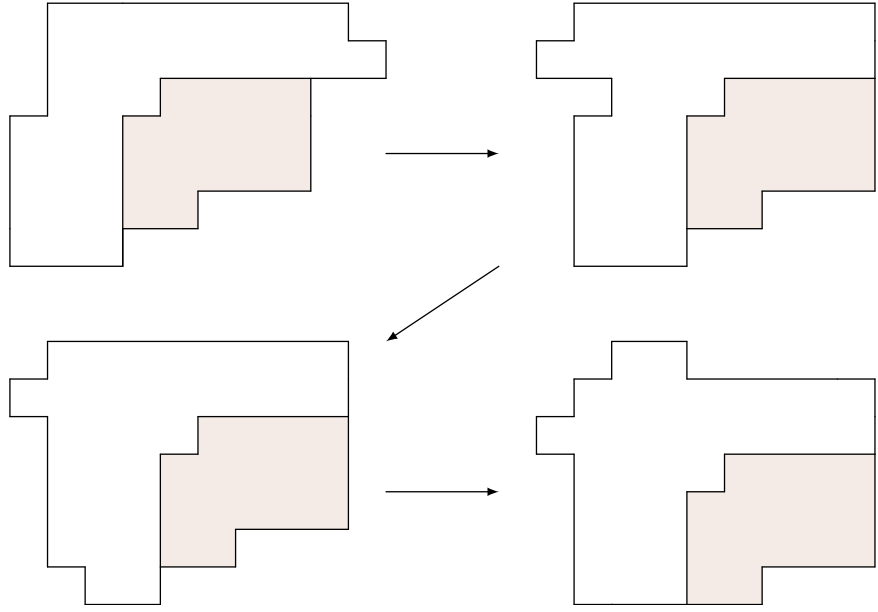
\begin{figure}[h]

\definecolor{zzttqq}{rgb}{0.6,0.2,0.}
\begin{tikzpicture}[line cap=round,line join=round,>=triangle 45,x=0.5cm,y=0.5cm]
\clip(1.2,-13.) rectangle (25.6,3.2);
\fill[line width=0.5pt,color=zzttqq,fill=zzttqq,fill opacity=0.1] (6.,1.) -- (10.,1.) -- (10.,-2.) -- (7.,-2.) -- (7.,-3.) -- (5.,-3.) -- (5.,0.) -- (6.,0.) -- cycle;
\fill[line width=0.5pt,color=zzttqq,fill=zzttqq,fill opacity=0.1] (21.,1.) -- (25.,1.) -- (25.,-2.) -- (22.,-2.) -- (22.,-3.) -- (20.,-3.) -- (20.,0.) -- (21.,0.) -- cycle;
\fill[line width=0.5pt,color=zzttqq,fill=zzttqq,fill opacity=0.1] (7.,-8.) -- (11.,-8.) -- (11.,-11.) -- (8.,-11.) -- (8.,-12.) -- (6.,-12.) -- (6.,-9.) -- (7.,-9.) -- cycle;
\fill[line width=0.5pt,color=zzttqq,fill=zzttqq,fill opacity=0.1] (21.,-9.) -- (25.,-9.) -- (25.,-12.) -- (22.,-12.) -- (22.,-13.) -- (20.,-13.) -- (20.,-10.) -- (21.,-10.) -- cycle;
\draw[line width=0.5pt] (5.,-4.)-- (2.,-4.);
\draw[line width=0.5pt] (2.,-3.)-- (2.,-4.);
\draw[line width=0.5pt] (3.,0.)-- (3.,3.);
\draw[line width=0.5pt] (3.,3.)-- (5.,3.);
\draw[line width=0.5pt] (5.,3.)-- (11.,3.);
\draw[line width=0.5pt] (11.,2.)-- (11.,3.);
\draw[line width=0.5pt] (11.,2.)-- (12.,2.);
\draw[line width=0.5pt] (12.,2.)-- (12.,1.);
\draw[line width=0.5pt] (12.,1.)-- (6.,1.);
\draw[line width=0.5pt] (6.,1.)-- (6.,0.);
\draw[line width=0.5pt] (6.,0.)-- (5.,0.);
\draw[line width=0.5pt] (5.,0.)-- (5.,-4.);
\draw[line width=0.5pt] (10.,1.)-- (10.,0.);
\draw[line width=0.5pt] (10.,0.)-- (10.,-2.);
\draw[line width=0.5pt] (10.,-2.)-- (7.,-2.);
\draw[line width=0.5pt] (7.,-2.)-- (7.,-3.);
\draw[line width=0.5pt] (7.,-3.)-- (5.,-3.);
\draw[line width=0.5pt] (5.,-3.)-- (5.,-4.);
\draw[line width=0.5pt] (3.,0.)-- (2.,0.);
\draw[line width=0.5pt] (2.,0.)-- (2.,-3.);
\draw[line width=0.5pt] (25.,1.)-- (25.,-2.);
\draw[line width=0.5pt] (25.,-2.)-- (22.,-2.);
\draw[line width=0.5pt] (22.,-2.)-- (22.,-3.);
\draw[line width=0.5pt] (22.,-3.)-- (20.,-3.);
\draw[line width=0.5pt] (20.,-3.)-- (20.,0.);
\draw[line width=0.5pt] (20.,0.)-- (21.,0.);
\draw[line width=0.5pt] (21.,0.)-- (21.,1.);
\draw[line width=0.5pt] (21.,1.)-- (25.,1.);
\draw[line width=0.5pt] (25.,1.)-- (25.,2.);
\draw[line width=0.5pt] (25.,3.)-- (17.,3.);
\draw[line width=0.5pt] (16.,1.)-- (18.,1.);
\draw[line width=0.5pt] (18.,1.)-- (18.,0.);
\draw[line width=0.5pt] (18.,0.)-- (17.,0.);
\draw[line width=0.5pt] (17.,0.)-- (17.,-4.);
\draw[line width=0.5pt] (17.,-4.)-- (20.,-4.);
\draw[line width=0.5pt] (20.,-4.)-- (20.,-3.);
\draw[line width=0.5pt] (11.,-8.)-- (11.,-11.);
\draw[line width=0.5pt] (11.,-11.)-- (8.,-11.);
\draw[line width=0.5pt] (8.,-11.)-- (8.,-12.);
\draw[line width=0.5pt] (8.,-12.)-- (6.,-12.);
\draw[line width=0.5pt] (6.,-12.)-- (6.,-9.);
\draw[line width=0.5pt] (6.,-9.)-- (7.,-9.);
\draw[line width=0.5pt] (7.,-9.)-- (7.,-8.);
\draw[line width=0.5pt] (7.,-8.)-- (11.,-8.);
\draw[line width=0.5pt] (11.,-8.)-- (11.,-7.);
\draw[line width=0.5pt] (11.,-6.)-- (3.,-6.);
\draw[line width=0.5pt] (2.,-8.)-- (3.,-8.);
\draw[line width=0.5pt] (3.,-8.)-- (3.,-12.);
\draw[line width=0.5pt] (3.,-12.)-- (4.,-12.);
\draw[line width=0.5pt] (4.,-12.)-- (4.,-13.);
\draw[line width=0.5pt] (4.,-13.)-- (6.,-13.);
\draw[line width=0.5pt] (6.,-13.)-- (6.,-12.);
\draw[line width=0.5pt] (25.,-9.)-- (25.,-12.);
\draw[line width=0.5pt] (25.,-12.)-- (22.,-12.);
\draw[line width=0.5pt] (22.,-12.)-- (22.,-13.);
\draw[line width=0.5pt] (22.,-13.)-- (20.,-13.);
\draw[line width=0.5pt] (20.,-13.)-- (20.,-10.);
\draw[line width=0.5pt] (20.,-10.)-- (21.,-10.);
\draw[line width=0.5pt] (21.,-10.)-- (21.,-9.);
\draw[line width=0.5pt] (21.,-9.)-- (25.,-9.);
\draw[line width=0.5pt] (25.,-9.)-- (25.,-8.);
\draw[line width=0.5pt] (24.,-7.)-- (20.,-7.);
\draw[line width=0.5pt] (16.,-8.)-- (16.,-9.);
\draw[line width=0.5pt] (16.,-9.)-- (17.,-9.);
\draw[line width=0.5pt] (17.,-9.)-- (17.,-13.);
\draw[line width=0.5pt] (17.,-13.)-- (18.,-13.);
\draw[line width=0.5pt] (18.,-13.)-- (18.,-13.);
\draw[line width=0.5pt] (18.,-13.)-- (20.,-13.);
\draw[-latex,line width=0.5pt] (12.,-1.) -- (15.,-1.);
\draw[-latex,line width=0.5pt] (15.,-4.) -- (12.,-6.);
\draw[-latex,line width=0.5pt] (12.,-10.) -- (15.,-10.);
\draw[line width=0.5pt] (17.,3.)-- (17.,2.);
\draw[line width=0.5pt] (17.,2.)-- (16.,2.);
\draw[line width=0.5pt] (16.,2.)-- (16.,1.);
\draw[line width=0.5pt] (25.,3.)-- (25.,2.);
\draw[line width=0.5pt] (3.,-6.)-- (3.,-7.);
\draw[line width=0.5pt] (3.,-7.)-- (2.,-7.);
\draw[line width=0.5pt] (2.,-7.)-- (2.,-8.);
\draw[line width=0.5pt] (11.,-7.)-- (11.,-6.);
\draw[line width=0.5pt] (24.,-7.)-- (25.,-7.);
\draw[line width=0.5pt] (25.,-8.)-- (25.,-7.);
\draw[line width=0.5pt] (16.,-8.)-- (17.,-8.);
\draw[line width=0.5pt] (17.,-8.)-- (17.,-7.);
\draw[line width=0.5pt] (17.,-7.)-- (18.,-7.);
\draw[line width=0.5pt] (18.,-7.)-- (18.,-6.);
\draw[line width=0.5pt] (18.,-6.)-- (20.,-6.);
\draw[line width=0.5pt] (20.,-6.)-- (20.,-7.);
\end{tikzpicture}

\caption{Regularisation of Class $\mathcal{III}$: part one}
\label{fig:classIIIb}
\end{figure}

In order to facilitate the proof that configurations in Class~$\mathcal{III}$ cannot be optimal, we further modify the configuration without increasing the energy.

\begin{lemma}\label{lem:classIIIregularisationstep2}
Fix $N_A, N_B > 0$ and $ \beta  \in (0,1)$. Suppose that $(A,B) \in \mathcal{III}$ is a minimal configuration. Then, there exists a minimal configuration $(\hat{A},\hat{B}) \in \mathcal{III}$ such that for every $k = 1,...,N_{\rm row}$ the rightmost point of $(\hat{A} \cup \hat{B})_k^{\rm row}$ has the same first coordinate and for every $k = 1,...,N_{\rm col}$ the lowest point of $(\hat{A} \cup \hat{B})_k^{\rm col}$ has the same second coordinate.  
\end{lemma}

\begin{proof}
By the previous proposition, we may assume that $l_3 = h_3 = 0$. We will use a version of the technique used for Class~$\mathcal{I}$, and refer to Figure \ref{fig:regularisationClassIII} for an illustration of the construction. Note that if we add $N_A'$ $A$-points on the top and on the left and $N_B'$ $B$-points on in the bottom right corner, so that the configuration $(A,B)$ becomes a full rectangle with sides $l_1 + l_2$ and $h_1 + h_2$, we do not alter the surface energy, but the bulk energy changes by $- 2  (N_A' + N_B')$.

 Having fixed $N_B' > 0$,  let us remove the topmost $A$-point  in the leftmost column and change the type of the topmost $B$-point in the leftmost column to $A$. In this way, we removed a $B$-point, without increasing the   energy \eqref{eq: nerg3}.   We repeat this procedure until we removed $N_B'$ $B$-points. Then, we remove $N_A'$ $A$-points, starting from the top of the leftmost column. Again, this cannot increase the energy. Moreover, the resulting configuration lies in Class $\mathcal{III}$ because if in this last step we removed a whole column or a point which lies next to the interface, we would decrease the energy. Hence, the resulting configuration is also minimal and satisfies the desired property. \end{proof}

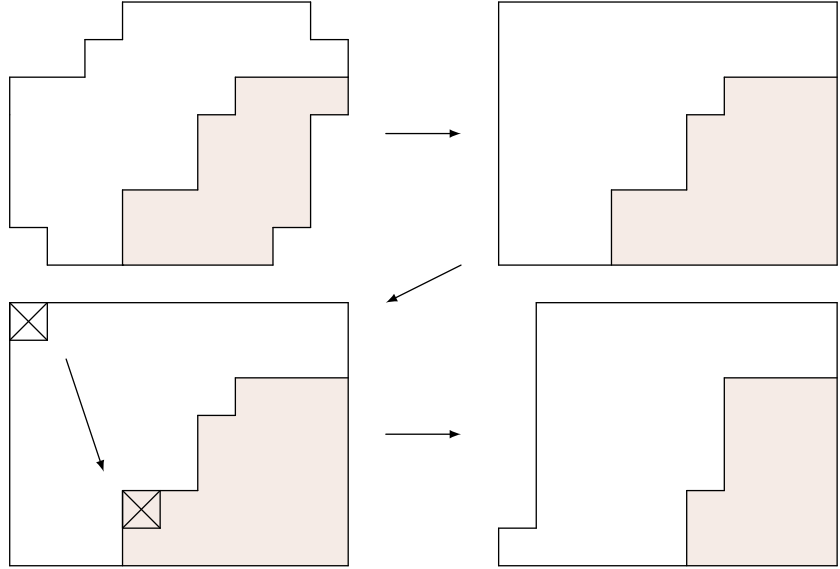
\begin{figure}[h]

\definecolor{zzttqq}{rgb}{0.6,0.2,0.}
\begin{tikzpicture}[line cap=round,line join=round,>=triangle 45,x=0.5cm,y=0.5cm]
\clip(1.2,-12.2) rectangle (24.6,3.3);
\fill[line width=0.5pt,color=zzttqq,fill=zzttqq,fill opacity=0.1] (8.,1.) -- (11.,1.) -- (11.,0.) -- (10.,0.) -- (10.,-3.) -- (9.,-3.) -- (9.,-4.) -- (5.,-4.) -- (5.,-2.) -- (7.,-2.) -- (7.,0.) -- (8.,0.) -- cycle;
\fill[line width=0.5pt,color=zzttqq,fill=zzttqq,fill opacity=0.1] (21.,1.) -- (24.,1.) -- (24.,-4.) -- (18.,-4.) -- (18.,-2.) -- (20.,-2.) -- (20.,0.) -- (21.,0.) -- cycle;
\fill[line width=0.5pt,color=zzttqq,fill=zzttqq,fill opacity=0.1] (8.,-7.) -- (11.,-7.) -- (11.,-12.) -- (5.,-12.) -- (5.,-10.) -- (7.,-10.) -- (7.,-8.) -- (8.,-8.) -- cycle;
\fill[line width=0.5pt,color=zzttqq,fill=zzttqq,fill opacity=0.1] (21.,-7.) -- (24.,-7.) -- (24.,-12.) -- (20.,-12.) -- (20.,-10.) -- (21.,-10.) -- cycle;
\draw[line width=0.5pt] (5.,-4.)-- (3.,-4.);
\draw[line width=0.5pt] (3.,-3.)-- (3.,-4.);
\draw[line width=0.5pt] (2.,0.)-- (2.,1.);
\draw[line width=0.5pt] (11.,2.)-- (11.,1.);
\draw[line width=0.5pt] (11.,1.)-- (8.,1.);
\draw[line width=0.5pt] (8.,1.)-- (8.,0.);
\draw[line width=0.5pt] (9.,-3.)-- (9.,-4.);
\draw[line width=0.5pt] (9.,-4.)-- (5.,-4.);
\draw[line width=0.5pt] (5.,-4.)-- (5.,-4.);
\draw[line width=0.5pt] (5.,3.)-- (10.,3.);
\draw[line width=0.5pt] (10.,3.)-- (10.,2.);
\draw[line width=0.5pt] (10.,2.)-- (11.,2.);
\draw[line width=0.5pt] (5.,3.)-- (5.,2.);
\draw[line width=0.5pt] (5.,2.)-- (4.,2.);
\draw[line width=0.5pt] (4.,2.)-- (4.,1.);
\draw[line width=0.5pt] (4.,1.)-- (2.,1.);
\draw[line width=0.5pt] (8.,0.)-- (7.,0.);
\draw[line width=0.5pt] (7.,0.)-- (7.,-2.);
\draw[line width=0.5pt] (7.,-2.)-- (5.,-2.);
\draw[line width=0.5pt] (5.,-2.)-- (5.,-4.);
\draw[line width=0.5pt] (9.,-3.)-- (10.,-3.);
\draw[line width=0.5pt] (10.,-3.)-- (10.,0.);
\draw[line width=0.5pt] (10.,0.)-- (11.,0.);
\draw[line width=0.5pt] (11.,0.)-- (11.,1.);
\draw[line width=0.5pt] (2.,0.)-- (2.,-3.);
\draw[line width=0.5pt] (2.,-3.)-- (3.,-3.);
\draw[line width=0.5pt] (15.,3.)-- (24.,3.);
\draw[line width=0.5pt] (24.,3.)-- (24.,-4.);
\draw[line width=0.5pt] (24.,-4.)-- (15.,-4.);
\draw[line width=0.5pt] (15.,-4.)-- (15.,3.);
\draw[line width=0.5pt] (18.,-4.)-- (18.,-2.);
\draw[line width=0.5pt] (18.,-2.)-- (20.,-2.);
\draw[line width=0.5pt] (20.,-2.)-- (20.,0.);
\draw[line width=0.5pt] (20.,0.)-- (21.,0.);
\draw[line width=0.5pt] (21.,0.)-- (21.,1.);
\draw[line width=0.5pt] (21.,1.)-- (24.,1.);
\draw[line width=0.5pt] (2.,-5.)-- (11.,-5.);
\draw[line width=0.5pt] (11.,-5.)-- (11.,-12.);
\draw[line width=0.5pt] (11.,-12.)-- (2.,-12.);
\draw[line width=0.5pt] (2.,-12.)-- (2.,-5.);
\draw[line width=0.5pt] (5.,-12.)-- (5.,-10.);
\draw[line width=0.5pt] (5.,-10.)-- (7.,-10.);
\draw[line width=0.5pt] (7.,-10.)-- (7.,-8.);
\draw[line width=0.5pt] (7.,-8.)-- (8.,-8.);
\draw[line width=0.5pt] (8.,-8.)-- (8.,-7.);
\draw[line width=0.5pt] (8.,-7.)-- (11.,-7.);
\draw[-latex,line width=0.5pt] (12.,-0.5) -- (14.,-0.5);
\draw[-latex,line width=0.5pt] (14.,-4.) -- (12.,-5.);
\draw[line width=0.5pt] (2.,-5.)-- (2.,-6.);
\draw[line width=0.5pt] (2.,-6.)-- (3.,-6.);
\draw[line width=0.5pt] (3.,-6.)-- (3.,-5.);
\draw[line width=0.5pt] (5.,-10.)-- (5.,-11.);
\draw[line width=0.5pt] (5.,-11.)-- (6.,-11.);
\draw[line width=0.5pt] (6.,-11.)-- (6.,-10.);
\draw[line width=0.5pt] (2.,-6.)-- (3.,-5.);
\draw[line width=0.5pt] (2.,-5.)-- (3.,-6.);
\draw[line width=0.5pt] (5.,-10.)-- (6.,-11.);
\draw[line width=0.5pt] (5.,-11.)-- (6.,-10.);
\draw[-latex,line width=0.5pt] (3.5,-6.5) -- (4.5,-9.5);
\draw[line width=0.5pt] (15.,-12.)-- (15.,-11.);
\draw[line width=0.5pt] (15.,-11.)-- (16.,-11.);
\draw[line width=0.5pt] (16.,-11.)-- (16.,-5.);
\draw[line width=0.5pt] (16.,-5.)-- (24.,-5.);
\draw[line width=0.5pt] (24.,-5.)-- (24.,-12.);
\draw[line width=0.5pt] (24.,-12.)-- (15.,-12.);
\draw[line width=0.5pt] (20.,-12.)-- (20.,-10.);
\draw[line width=0.5pt] (20.,-10.)-- (21.,-10.);
\draw[line width=0.5pt] (21.,-10.)-- (21.,-7.);
\draw[line width=0.5pt] (21.,-7.)-- (24.,-7.);
\draw[-latex,line width=0.5pt] (12.,-8.5) -- (14.,-8.5);
\end{tikzpicture}

\caption{Regularisation of Class $\mathcal{III}$: part two}
\label{fig:regularisationClassIII}
\end{figure}

These regularisation results imply that in the case when the numbers of points in the two types are equal, then the minimising configuration cannot lie in Class $\mathcal{III}$.

\begin{proposition}\label{prop:classIIIexcluded}
 Fix $ N_A = N_B>0$  and $ \beta  \in (0,1)$. Then, if $(A,B)$ is a minimal configuration, $(A,B) \notin \mathcal{III}$.
\end{proposition}

\begin{proof}
Suppose otherwise and let $(A,B) \in \mathcal{III}$ be a minimal configuration. Apply the regularisation procedure described in Proposition \ref{prop:classIIIregularisationstep1} and Lemma \ref{lem:classIIIregularisationstep2}. After these operations, $(A,B)$ lies in a rectangle $R$ with sides $h_1 + h_2$ and $l_1 + l_2$. Then, the length of the interface equals $l_2 + h_2$. Without loss of generality $h_1 + h_2 \leq l_1 + l_2$ (otherwise, this is true after applying a symmetry with respect to the line  $\mathbb{R}(-1,1)$).  Then, we compare $(A,B)$ with a configuration $(\hat{A},\hat{B}) \in \mathcal{I}$ which fits into the rectangle $R$, with $A$-points on the left and  $B$-points on the right such that the length of the interface is either $h_1 + h_2$ or $h_1 + h_2 + 1$,  depending on whether $l_2 =0$ or $l_2 = 1$.  Hence, by minimality of $(A,B)$, we have $l_2 + h_2 \leq h_1 + h_2 + 1$, i.e.,
\begin{align}\label{eq: l2h1}
l_2 \leq h_1 + 1.
\end{align}
This gives a contradiction with the assumption $N_A = N_B$.  To see
this, first recall that the configuration is {\it full}, in the sense that the construction in Lemma \ref{lem:classIIIregularisationstep2} ensures that all the columns except for the leftmost one have the same number of points. Therefore, we may  first estimate from above the number of $B$-points by 
\begin{equation*}
N_B \leq l_2 h_2 \leq h_1 h_2 + h_2
\end{equation*}
 and  the number of $A$-points from below by 
\begin{align*}
N_A & \geq h_1 l_2 + h_1 (l_1 - 1) + h_2 (l_1 - 1) = h_1 (l_1 + l_2) - h_1 + h_2 (l_1 - 1)\\
&\geq h_1 (h_1 + h_2) - h_1 + h_2 (l_1 - 1) = h_1 h_2 + h_1 (h_1 - 1) + h_2 (l_1 - 1),
\end{align*}
where we used the assumption that $h_1 + h_2 \leq l_1 + l_2$. Hence, whenever $h_1, l_1 \geq 2$ or $l_1 \geq 3$, we have $N_A > N_B$, which would contradict the assumption $N_A = N_B$. Moreover, we get that necessarily $h_1 \leq h_2$.

Finally, we have to take into consideration the case when $l_1 = 1$
(with $h_1$ arbitrary) or when $h_1 = 1$ and $l_1 = 2$. In the first
case, by \eqref{eq: l2h1} we have $h_1 + h_2 \leq l_2 + 1 \leq h_1 +
2$, so $h_2 \leq 2$. But then $h_1 \leq h_2 \leq 2$, and thus $l_2
\leq h_1 + 1 \leq 3$.  This leaves us with a finite (and small) number
of configurations to consider separately and it may be checked that
none of them  is  optimal.  In the second case, again by \eqref{eq:
  l2h1} we have $l_2 \leq h_1 + 1 = 2$. Furthermore, $l_1 + l_2 \geq
h_1 + h_2$, so $h_1 + h_2 \leq 4$, and hence $h_2 \leq 3$. Again, we
end up with a small number of configurations.  A direct exhaustive
analysis guarantees that   none of them is optimal. 
\end{proof}

\subsection{Class $\mathcal{IV}$, part one}\label{sec:classIVpartone}

The situation in Class~$\mathcal{IV}$ is not as clear-cut as in
Classes $\mathcal{II}$ and $\mathcal{III}$: whereas configurations in
Classes $\mathcal{II}$ and $\mathcal{III}$ are never optimal, the
problem is that, even for $N_A = N_B$ and $ \beta  = 1/2$, an optimal configuration may actually lie in Class~$\mathcal{IV}$, see Figure~\ref{fig:fivepoints}. Hence, the goal in this subsection is a bit different: we will prove that even though minimal configurations in Class $\mathcal{IV}$ may exist, there also exists an optimal configuration in Class $\mathcal{I}$. Moreover, the reasoning will also provide some further properties of optimal configurations in Class $\mathcal{IV}$.  In particular, a careful inspection of the forthcoming constructions will show a fluctuation estimate for minimisers in Class~$\mathcal{IV}$, see Section~\ref{sec:law} below. 

This goal is achieved as follows: in the first part, we regularise our
configuration such that $h_3 = 0$ and $h_1\le l_1$. This is achieved
in Proposition \ref{prop:classIVregularisationstep3}, with the key
part of the reasoning proved in Proposition
\ref{prop:classIVregularisationstep2}. These arguments are valid for
any $ \beta   \in (0,1)$. Then, in the second part, under the
restriction $ \beta \leq  1/2$, we regularise a configuration with $h_3 = 0$ and $h_1\le l_1$ to obtain a configuration in Class $\mathcal{I}$. This is achieved in Propositions  \ref{prop:regularisationofclassIVpart4}--\ref{prop:wemayrequireclassI}.  We break the reasoning into smaller pieces in order to highlight different techniques and different assumptions required at each point.


\begin{lemma}\label{lem:classIVregularisationstep1}
Fix $N_A, N_B > 0$ and $ \beta  \in (0,1)$. Suppose that $(A,B) \in \mathcal{IV}$ is a minimal configuration. Then, there exists  a minimal configuration  $(\hat{A},\hat{B}) \in \mathcal{I} \cup \mathcal{IV}$ such that  $l_2 \leq h_2$, $\min \lbrace h_1, h_1+h_2 -l_1-l_2\rbrace \le 0$, and $\min \lbrace h_3, h_2+h_3 -l_2-l_3\rbrace \le 0$. 
\end{lemma}

\begin{proof}
Choose a minimal configuration $(A,B)$ in Class~$\mathcal{IV}$.  Without loss of generality, we may assume that $l_2 \leq h_2$. Otherwise, consider a reflection of the original configuration with respect to the line $\mathbb{R}(-1,1)$. Then, we end up with a configuration of the same type with the roles of $h_i$ and $l_i$ reversed. We suppose that  $h_1 \ge 1$ as otherwise  the second condition in the statement of the  lemma  is satisfied.  We modify the configuration without increasing the energy such that $h_1=0$ or $l_1 + l_2 \ge h_1 + h_2$.  To see this, suppose that  $l_1 + l_2 < h_1 + h_2$. Then,  we remove all the points in the first row, and place them on the  left-hand  side starting from the second row, one in each row, possibly forming one additional column. The assumption guarantees that there was enough space to place all the points. In this way,  $h_1$  decreases by 1 and  $l_1$  increases possibly by 1, so the total energy decreases (in which case $(A,B)$ was not a minimal configuration) or stays the same, cf.\ \eqref{eq:classIVformula}. We repeat this procedure until $h_1 = 0$ or $l_1 + l_2 \geq h_1 + h_2$. 

In a similar fashion, we modify the configuration to obtain $\min \lbrace h_3, h_2+h_3 -l_2-l_3\rbrace \le 0$. Finally, if $h_1=h_3=0$, the configuration is in Class~$\mathcal{I}$. Otherwise, if $h_1 \ge 1$, the configuration is in Class~$\mathcal{IV}$, and if $h_1=0$, $h_3 \ge 1$, after a rotation by $\pi$ and interchanging the roles of the two types we obtain a configuration in Class~$\mathcal{IV}$.
\end{proof}

We continue the regularisation in the following proposition.

\begin{proposition}\label{prop:classIVregularisationstep2} 
Fix $N_A, N_B > 0$ and $ \beta  \in (0,1)$. Suppose that $(A,B) \in \mathcal{IV}$ is a minimal configuration.  Then, there exists a minimal configuration  $(\hat{A},\hat{B}) \in  \mathcal{I} \cup \mathcal{IV}$ which satisfies $l_2 \leq h_2$, $\min \lbrace h_1, h_1+h_2 -l_1-l_2\rbrace \le 0$,  $\min \lbrace h_3, h_2+h_3 -l_2-l_3\rbrace \le 0$, and at least one of the following two properties: 
$$(1) \ \ l_2 = 1, \quad \quad  \quad   (2)  \ \ h_3 = 0.$$
\end{proposition}

For the proof, we introduce the following notation specific for Class $\mathcal{IV}$. With the notation of Figure \ref{fig:ClassIVbefore}, we will refer to the nine rectangles with sides $l_i$ and $h_j$ as $l_i:h_j$. For instance, the rectangle in the middle with sides $l_2$ and $h_2$ will be referred to as rectangle $l_2:h_2$. A priori, some of these rectangles may be not full or even empty, for instance the rectangle $l_3:h_1$.

\begin{proof}
Let $(A,B)  \in \mathcal{IV}$ be a minimal configuration from Lemma \ref{lem:classIVregularisationstep1} which does not satisfy  the desired properties, i.e.,  $l_2 >1$ and $h_1,h_3 >0$ ($l_2=0$ is not possible as it would imply $(A,B)  \in \mathcal{II}$). Then, we first make a similar regularisation as we did for Class~$\mathcal{I}$. We add $N_A'$ $A$-points to the configuration $(A,B)$, so that the interface between $A$ and the void consists of four line segments (of lengths $l_1$, $h_1 + h_2$, $l_1 + l_2$ and $h_1$). This does not increase the surface energy. Then, we remove $N_A'$ $A$-points, column by column, starting from the leftmost column in $(A,B)$. If we removed a whole column, or if we removed a point which lies at the interface, the energy drops, so the original configuration $(A,B)$ was not minimal. Hence, the resulting configuration lies in Class~$\mathcal{IV}$. We proceed in a similar fashion for the $B$-points. In particular, the rectangle $l_2:h_2$ (in the middle) is full.

Now, let us look at the (full) rectangle $l_2:h_2$. It contains exactly $l_2 h_2$ points, $N_A''$ of them of type $A$ and $N_B''$ of them of type $B$. We rearrange them (i.e., remove all the points in $l_2:h_2$ and place them back in $l_2:h_2$) in the following way: we start with the leftmost column and we fill the columns one by one  with $A$-points  until we end up with less than $h_2$ points to place. Then, we place the remaining points in the next column, starting from the top. Similarly, we place the $B$-points starting from the rightmost column and we fill the columns one by one until we end up with less than $h_2$ points. We place the remaining points on the bottom of the next column. In this way, the resulting configuration has an interface with at most one step in $l_2:h_2$, and we did not change the energy. By Lemma  \ref{lem:classIVregularisationstep1} we also have
\begin{equation}\label{eq:tripleassumption}
{\rm (i)} \ \ l_2 \leq h_2, \quad \quad {\rm (ii)} \ \  l_1 \geq h_1, \quad \quad {\rm (iii) } \ \ l_3 \geq h_3.
\end{equation}
Indeed, (i) is clear. If  $h_1 = 0$, (ii) is obvious. Otherwise we have $h_1+h_2 - l_1-l_2 \le 0$ which along with (i) shows (ii). The proof of (iii) is similar.  The procedure described above is presented in Figure \ref{fig:regularisationClassIVpart1}. 

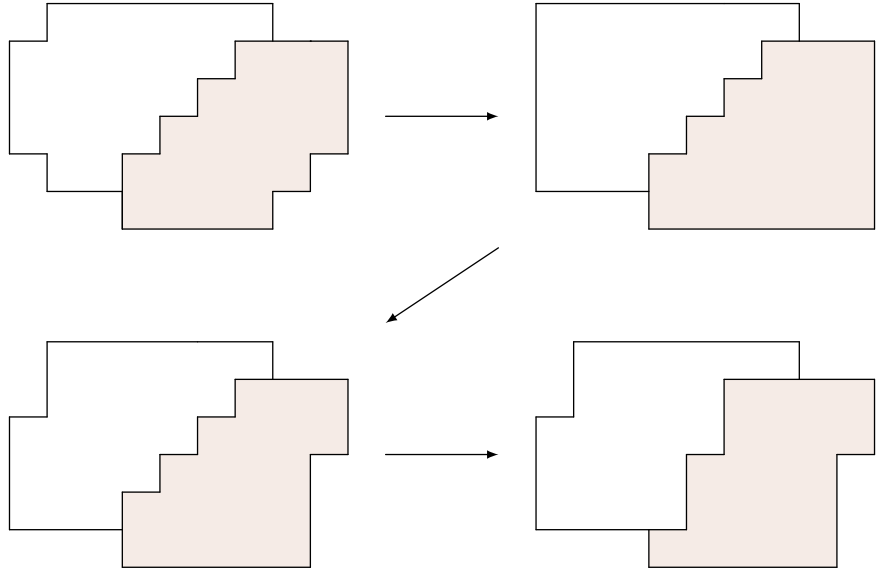
\begin{figure}[h!]

\definecolor{zzttqq}{rgb}{0.6,0.2,0.}
\begin{tikzpicture}[line cap=round,line join=round,>=triangle 45,x=0.5cm,y=0.5cm]
\clip(-1.1,-9.5) rectangle (23.5,6.5);
\fill[line width=2.pt,color=zzttqq,fill=zzttqq,fill opacity=0.1] (23.,-4.) -- (19.,-4.) -- (19.,-6.) -- (18.,-6.) -- (18.,-8.) -- (17.,-8.) -- (17.,-9.) -- (22.,-9.) -- (22.,-6.) -- (23.,-6.) -- cycle;
\fill[line width=2.pt,color=zzttqq,fill=zzttqq,fill opacity=0.1] (9.,5.) -- (6.,5.) -- (6.,4.) -- (5.,4.) -- (5.,3.) -- (4.,3.) -- (4.,2.) -- (3.,2.) -- (3.,0.) -- (7.,0.) -- (7.,1.) -- (8.,1.) -- (8.,2.) -- (9.,2.) -- cycle;
\fill[line width=2.pt,color=zzttqq,fill=zzttqq,fill opacity=0.1] (9.,-4.) -- (6.,-4.) -- (6.,-5.) -- (5.,-5.) -- (5.,-6.) -- (4.,-6.) -- (4.,-7.) -- (3.,-7.) -- (3.,-9.) -- (8.,-9.) -- (8.,-6.) -- (9.,-6.) -- cycle;
\fill[line width=2.pt,color=zzttqq,fill=zzttqq,fill opacity=0.1] (23.,5.) -- (20.,5.) -- (20.,4.) -- (19.,4.) -- (19.,3.) -- (18.,3.) -- (18.,2.) -- (17.,2.) -- (17.,0.) -- (23.,0.) -- cycle;
\draw[line width=0.5pt] (3.,1.)-- (1.,1.);
\draw[line width=0.5pt] (1.,1.)-- (1.,2.);
\draw[line width=0.5pt] (1.,2.)-- (0.,2.);
\draw[line width=0.5pt] (0.,2.)-- (0.,5.);
\draw[line width=0.5pt] (5.,3.)-- (4.,3.);
\draw[line width=0.5pt] (4.,3.)-- (4.,2.);
\draw[line width=0.5pt] (4.,2.)-- (3.,2.);
\draw[line width=0.5pt] (3.,2.)-- (3.,0.);
\draw[line width=0.5pt] (3.,0.)-- (3.,1.);
\draw[line width=0.5pt] (8.,5.)-- (8.,5.);
\draw[line width=0.5pt] (8.,5.)-- (9.,5.);
\draw[line width=0.5pt] (9.,5.)-- (9.,2.);
\draw[line width=0.5pt] (8.,2.)-- (9.,2.);
\draw[line width=0.5pt] (8.,2.)-- (8.,1.);
\draw[line width=0.5pt] (8.,1.)-- (7.,1.);
\draw[line width=0.5pt] (7.,1.)-- (7.,0.);
\draw[line width=0.5pt] (7.,0.)-- (3.,0.);
\draw[line width=0.5pt] (14.,6.)-- (19.,6.);
\draw[line width=0.5pt] (19.,3.)-- (18.,3.);
\draw[line width=0.5pt] (18.,3.)-- (18.,2.);
\draw[line width=0.5pt] (18.,2.)-- (17.,2.);
\draw[line width=0.5pt] (17.,2.)-- (17.,0.);
\draw[line width=0.5pt] (14.,1.)-- (14.,6.);
\draw[line width=0.5pt] (1.,-3.)-- (5.,-3.);
\draw[line width=0.5pt] (5.,-6.)-- (4.,-6.);
\draw[line width=0.5pt] (8.,-9.)-- (3.,-9.);
\draw[line width=0.5pt] (15.,-3.)-- (15.,-5.);
\draw[line width=0.5pt] (14.,-5.)-- (15.,-5.);
\draw[line width=0.5pt] (14.,-5.)-- (14.,-8.);
\draw[line width=0.5pt] (14.,-8.)-- (17.,-8.);
\draw[line width=0.5pt] (17.,-8.)-- (18.,-8.);
\draw[line width=0.5pt] (18.,-8.)-- (18.,-6.);
\draw[line width=0.5pt] (23.,-6.)-- (22.,-6.);
\draw[line width=0.5pt] (22.,-6.)-- (22.,-9.);
\draw[line width=0.5pt] (22.,-9.)-- (17.,-9.);
\draw[line width=0.5pt] (17.,-9.)-- (17.,-8.);
\draw [-latex,line width=0.5pt] (10.,3.) -- (13.,3.);
\draw [-latex,line width=0.5pt] (13.,-0.5) -- (10.,-2.5);
\draw [-latex,line width=0.5pt] (10.,-6.) -- (13.,-6.);
\draw[line width=0.5pt] (0.,5.)-- (1.,5.);
\draw[line width=0.5pt] (1.,5.)-- (1.,6.);
\draw[line width=0.5pt] (1.,6.)-- (7.,6.);
\draw[line width=0.5pt] (19.,6.)-- (21.,6.);
\draw[line width=0.5pt] (21.,6.)-- (21.,5.);
\draw[line width=0.5pt] (21.,5.)-- (23.,5.);
\draw[line width=0.5pt] (23.,5.)-- (23.,0.);
\draw[line width=0.5pt] (23.,0.)-- (17.,0.);
\draw[line width=0.5pt] (14.,1.)-- (17.,1.);
\draw[line width=0.5pt] (4.,-6.)-- (4.,-7.);
\draw[line width=0.5pt] (4.,-7.)-- (3.,-7.);
\draw[line width=0.5pt] (3.,-7.)-- (3.,-8.);
\draw[line width=0.5pt] (3.,-8.)-- (3.,-9.);
\draw[line width=0.5pt] (3.,-8.)-- (0.,-8.);
\draw[line width=0.5pt] (5.,-3.)-- (7.,-3.);
\draw[line width=0.5pt] (7.,-3.)-- (7.,-4.);
\draw[line width=0.5pt] (7.,-4.)-- (9.,-4.);
\draw[line width=0.5pt] (1.,-3.)-- (1.,-5.);
\draw[line width=0.5pt] (1.,-5.)-- (0.,-5.);
\draw[line width=0.5pt] (0.,-5.)-- (0.,-8.);
\draw[line width=0.5pt] (8.,-9.)-- (8.,-6.);
\draw[line width=0.5pt] (8.,-6.)-- (9.,-6.);
\draw[line width=0.5pt] (9.,-6.)-- (9.,-4.);
\draw[line width=0.5pt] (15.,-3.)-- (21.,-3.);
\draw[line width=0.5pt] (21.,-3.)-- (21.,-4.);
\draw[line width=0.5pt] (21.,-4.)-- (19.,-4.);
\draw[line width=0.5pt] (19.,-4.)-- (19.,-6.);
\draw[line width=0.5pt] (19.,-6.)-- (18.,-6.);
\draw[line width=0.5pt] (21.,-4.)-- (23.,-4.);
\draw[line width=0.5pt] (23.,-4.)-- (23.,-6.);
\draw[line width=0.5pt] (7.,6.)-- (7.,5.);
\draw[line width=0.5pt] (8.,5.)-- (6.,5.);
\draw[line width=0.5pt] (6.,5.)-- (6.,4.);
\draw[line width=0.5pt] (6.,4.)-- (5.,4.);
\draw[line width=0.5pt] (5.,4.)-- (5.,3.);
\draw[line width=0.5pt] (7.,-4.)-- (6.,-4.);
\draw[line width=0.5pt] (6.,-4.)-- (6.,-5.);
\draw[line width=0.5pt] (6.,-5.)-- (5.,-5.);
\draw[line width=0.5pt] (5.,-5.)-- (5.,-6.);
\draw[line width=0.5pt] (21.,5.)-- (20.,5.);
\draw[line width=0.5pt] (20.,5.)-- (20.,4.);
\draw[line width=0.5pt] (20.,4.)-- (19.,4.);
\draw[line width=0.5pt] (19.,4.)-- (19.,3.);
\end{tikzpicture}

\caption{Regularisation of Class $\mathcal{IV}$: part one}
\label{fig:regularisationClassIVpart1}
\end{figure}

As $l_2 \ge 2$  and the interface has at most one step,  we observe that at least one of the following cases holds true: (a) The rightmost column  of $l_2:h_2$ consists only of points of type $B$. (b)  The leftmost column  of $l_2:h_2$ consists only of points of type $A$.
Then, we do one of the two following procedures:

(a) We move the $A$-points from the rightmost column of the rectangle $l_2:h_1$ (in the upper right corner) to the rectangle $l_1:h_3$ (in the bottom left corner) and place them in its highest row (starting from the right). Here, we use  \eqref{eq:tripleassumption}(ii) and $h_3 \ge 1$. In this way, we do not increase the surface energy, see \eqref{eq:classIVformula}, since we have $h_2 \rightarrow h_2 + 1$, $l_2 \rightarrow l_2-1$, $h_3 \rightarrow h_3 - 1$, $l_3 \rightarrow l_3 + 1$, and $h_1$ and $l_1$ remain unchanged. Finally, we perform a rearrangement in the new rectangle $l_2:h_2$ as above.  An example of such construction is given by the top arrow in Figure \ref{fig:regularisationClassIVpart2}. 

\begin{figure}[h!]

\definecolor{zzttqq}{rgb}{0.6,0.2,0.}
\begin{tikzpicture}[line cap=round,line join=round,>=triangle 45,x=0.5cm,y=0.5cm]
\clip(-1.1,-9.5) rectangle (23.5,6.5);
\fill[line width=2.pt,color=zzttqq,fill=zzttqq,fill opacity=0.1] (9.,0.5) -- (5.,0.5) -- (5.,-1.5) -- (4.,-1.5) -- (4.,-3.5) -- (3.,-3.5) -- (3.018522391962444,-4.503093145874537) -- (7.998695591890927,-4.503093145874537) -- (8.,-1.5) -- (9.,-1.5) -- cycle;
\fill[line width=2.pt,color=zzttqq,fill=zzttqq,fill opacity=0.1] (23.,5.) -- (19.,5.) -- (19.,3.) -- (18.,3.) -- (18.,1.) -- (17.,1.) -- (17.,0.) -- (22.,0.) -- (22.,3.) -- (23.,3.) -- cycle;
\fill[line width=2.pt,color=zzttqq,fill=zzttqq,fill opacity=0.1] (23.,-4.) -- (22.,-4.) -- (22.,-3.) -- (21.,-3.) -- (21.,-4.) -- (19.,-4.) -- (19.,-6.) -- (18.,-6.) -- (18.,-9.) -- (22.,-9.) -- (22.,-6.) -- (23.,-6.) -- cycle;
\draw[line width=0.5pt] (5.,0.5)-- (7.,0.5);
\draw[line width=0.5pt] (9.,0.5)-- (7.,0.5);
\draw[line width=0.5pt] (9.,0.5)-- (9.,-1.5);
\draw[line width=0.5pt] (9.,-1.5)-- (8.,-1.5);
\draw[line width=0.5pt] (8.,-1.5)-- (8.,-4.5);
\draw[line width=0.5pt] (8.,-4.5)-- (3,-4.5);
\draw[line width=0.5pt] (15.,-3.)-- (15.,-5.);
\draw[line width=0.5pt] (14.,-5.)-- (15.,-5.);
\draw[line width=0.5pt] (14.,-5.)-- (14.,-8.);
\draw[line width=0.5pt] (14.,-8.)-- (18.,-8.);
\draw[line width=0.5pt] (18.,-8.)-- (18.,-8.);
\draw[line width=0.5pt] (18.,-8.)-- (18.,-6.);
\draw[line width=0.5pt] (23.,-6.)-- (22.,-6.);
\draw[line width=0.5pt] (22.,-6.)-- (22.,-9.);
\draw[line width=0.5pt] (22.,-9.)-- (18.,-9.);
\draw[line width=0.5pt] (18.,-9.)-- (18.,-8.);
\draw [-latex,line width=0.5pt] (10.,1.) -- (13.,2.);
\draw [-latex,line width=0.5pt] (10.,-3.) -- (13.,-4.);
\draw[line width=0.5pt] (15.,6.)-- (20.,6.);
\draw[line width=0.5pt] (20.,6.)-- (20.,5.);
\draw[line width=0.5pt] (20.,5.)-- (23.,5.);
\draw[line width=0.5pt] (15.,-3.)-- (21.,-3.);
\draw[line width=0.5pt] (19.,-4.)-- (19.,-6.);
\draw[line width=0.5pt] (19.,-6.)-- (18.,-6.);
\draw[line width=0.5pt] (23.,-4.)-- (23.,-6.);
\draw[line width=0.5pt] (20.,5.)-- (19.,5.);
\draw[line width=0.5pt] (19.,5.)-- (19.,3.);
\draw[line width=0.5pt] (7.,0.5)-- (7.,1.5);
\draw[line width=0.5pt] (7.,1.5)-- (1.,1.5);
\draw[line width=0.5pt] (1.,1.5)-- (1.,-0.5);
\draw[line width=0.5pt] (1.,-0.5)-- (0.,-0.5);
\draw[line width=0.5pt] (0.,-0.5)-- (0.,-3.5);
\draw[line width=0.5pt] (0.,-3.5)-- (3.,-3.5);
\draw[line width=0.5pt] (3.,-3.5)-- (3.,-4.5);
\draw[line width=0.5pt] (3.,-3.5)-- (4.,-3.5);
\draw[line width=0.5pt] (4.,-3.5)-- (4.,-1.5);
\draw[line width=0.5pt] (4.,-1.5)-- (5.,-1.5);
\draw[line width=0.5pt] (5.,-1.5)-- (5.,0.5);
\draw[line width=0.5pt] (23.,5.)-- (23.,3.);
\draw[line width=0.5pt] (23.,3.)-- (22.,3.);
\draw[line width=0.5pt] (22.,3.)-- (22.,0.);
\draw[line width=0.5pt] (22.,0.)-- (17.,0.);
\draw[line width=0.5pt] (18.,1.)-- (18.,3.);
\draw[line width=0.5pt] (18.,3.)-- (19.,3.);
\draw[line width=0.5pt] (15.,6.)-- (15.,4.);
\draw[line width=0.5pt] (15.,4.)-- (14.,4.);
\draw[line width=0.5pt] (14.,4.)-- (14.,1.);
\draw[line width=0.5pt] (14.,1.)-- (16.,1.);
\draw[line width=0.5pt] (16.,1.)-- (16.,0.);
\draw[line width=0.5pt] (16.,0.)-- (17.,0.);
\draw[line width=0.5pt] (17.,0.)-- (17.,1.);
\draw[line width=0.5pt] (17.,1.)-- (18.,1.);
\draw[line width=0.5pt] (19.,-4.)-- (21.,-4.);
\draw[line width=0.5pt] (21.,-4.)-- (21.,-3.);
\draw[line width=0.5pt] (21.,-3.)-- (22.,-3.);
\draw[line width=0.5pt] (22.,-3.)-- (22.,-4.);
\draw[line width=0.5pt] (22.,-4.)-- (23.,-4.);
\draw (10.8,2.6) node[anchor=north west] {(a)};
\draw (10.8,-2.4) node[anchor=north west] {(b)};
\end{tikzpicture}

\caption{Regularisation of Class $\mathcal{IV}$: part two}
\label{fig:regularisationClassIVpart2}
\end{figure}
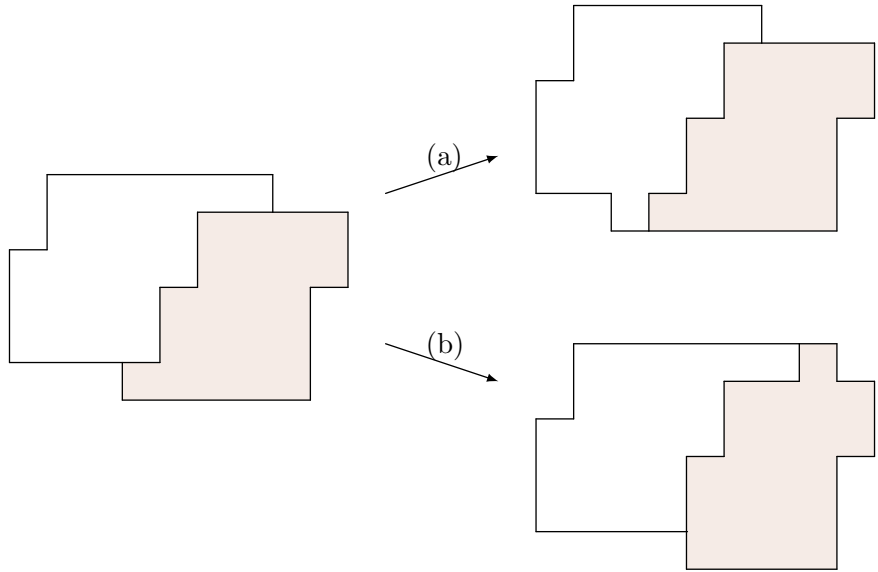

(b) We move the $B$-points from the leftmost column of the rectangle $l_2:h_3$ (in the bottom left corner) to the rectangle $l_3:h_1$ (in the upper right corner) and place them in its lowest row (starting from the left). Here, we use  \eqref{eq:tripleassumption}(iii) and $h_1 \ge 1$. In this way, we do not increase the surface energy since  we have $h_2 \rightarrow h_2 + 1$, $l_2 \rightarrow l_2-1$, $h_1 \rightarrow h_1 - 1$, $l_1 \rightarrow l_1 + 1$, and $h_3$ and $l_3$ remain unchanged. Finally, we perform a rearrangement in the new rectangle $l_2:h_2$ as above. An example of such construction is given by the bottom arrow in Figure \ref{fig:regularisationClassIVpart2}. 

In both cases, after applying the procedure, the condition \eqref{eq:tripleassumption} is still satisfied, so we may repeat it. We repeat it until $l_2 = 1$, $h_1 = 0$, or $h_3 = 0$.  Indeed, this follows after a finite number of steps since in each step $l_2$ decreases. If $l_2 = 1$ or $h_3 = 0$ hold, the proof is concluded. Otherwise, $h_3 = 0$ holds after a rotation by $\pi$ and interchanging the roles of the two types. 
\end{proof}

We now come to the main result of this subsection.

\begin{proposition}\label{prop:classIVregularisationstep3}
Fix  $ N_A = N_B>0$  and $ \beta  \in (0,1)$. Suppose that $(A,B) \in \mathcal{IV}$ is a minimal configuration.  Then, there exists a minimal configuration $({A},{B}) \in  \mathcal{I} \cup \mathcal{IV}$ such that $l_2 \leq h_2$, $h_1 \leq l_1$, and  $h_3 = 0$.  
\end{proposition}

\begin{proof}
Let $(A,B)$ be a configuration from Proposition \ref{prop:classIVregularisationstep2}.  Suppose by contradiction that $(A,B)$ (up to a rotation by $\pi$ and interchanging the roles of the two types) does not have the desired properties. Since \eqref{eq:tripleassumption} holds, we thus get that $l_2=1$ and $h_1, h_3 > 0$. By Proposition~\ref{prop:classIVregularisationstep2} and  $h_1, h_3 > 0$ we also have
\begin{equation}\label{eq:squareboundinthespecialcase}
l_1 + 1 \geq h_1 + h_2 \qquad \mbox{and} \qquad l_3 + 1 \geq h_2 + h_3. 
\end{equation}
As $h_2 \ge 1$, this particularly implies $h_1 \le l_1$. We can thus move the single column $l_2: h_1$ to the empty rectangle $l_1: h_3$ without increasing the energy. Note that $l_1 \ge h_1$ guarantees that there was enough space to place all the points. The resulting configuration has a straight interface with $h_1 >0$, i.e., lies in Class~$\mathcal{II}$. In view of Proposition~\ref{prop:classIIregularisation}, however, this contradicts optimality of the original configuration.
\end{proof}

Hence, for $N_A = N_B$ and any $ \beta  \in (0,1)$, we may
require that $h_3 = 0$ and $h_1 \le l_1$. We continue the analysis in
the next subsection, with an additional requirement on $ \beta$. 

\subsection{Class $\mathcal{IV}$, part two}\label{sec:classIVparttwo}

From now on, we will work with configurations which satisfy the
statement of  Proposition \ref{prop:classIVregularisationstep3}, i.e.,
$h_1 \leq l_1$ and $h_3 = 0$. Our goal is to perform a further
modification such that configurations  lie in Class $\mathcal{I}$. To
this end, we assume without restriction that configurations from
Proposition \ref{prop:classIVregularisationstep3} lie in
Class~$\mathcal{IV}$ and that $ N   := N_A = N_B$. In due course, we will introduce an additional assumption on $ \beta   \in (0,1)$.

As a first step of the regularisation procedure, we again straighten the interface such that it has at most one step.

\begin{lemma}\label{lemma: step lemma}
Fix $ N_A = N_B>0 $ and $ \beta   \in (0,1)$. Suppose that $(A,B) \in \mathcal{IV}$ is an optimal configuration with $h_3 = 0$. Then, there exists a minimal configuration with the same properties and at most one step in the interface.
\end{lemma}

\begin{proof}
We proceed similarly to our reasoning in Class $\mathcal{I}$, i.e., as in the proof of Proposition \ref{prop:classIregularisationstep1}. We add points to the configuration such that the rectangles $l_i:h_j$ for $i=1,2,3$ and $j=1,2$, except for $l_3:h_1$ are full. In this way, the surface part of the energy did not change. Then, we remove the same number of $A$- and $B$-points that we added, starting with the leftmost and rightmost column. If we removed a full column, then the energy would drop and the original configuration would not be minimal. Hence, the rectangle $l_2:h_2$ is necessarily full. Let us now reorganise it in the following way: we put all the $A$-points to the left and all the $B$-points to the right, so that the interface between them (inside $l_2:h_2$) is vertical except  for  a single possible step to the right. Its length did not change, so the resulting configuration is optimal.
\end{proof}

\begin{lemma}\label{lem:h1smallerthanh2}
Fix $ N_A = N_B>0 $ and $ \beta   \in (0,1)$. Suppose that $(A,B) \in \mathcal{IV}$ is an optimal configuration such that $h_1 \leq l_1$ and $h_3 = 0$. Then,  $h_1  \le  h_2$. 
\end{lemma}


\begin{proof}
Suppose otherwise, i.e.,  $h_1  >  h_2$.  First, we can assume that $l_3 \leq h_2$. Indeed, if not, we can remove the whole rectangle $l_3:h_2$, rotate it by $\pi/2$ and reattach it to the configuration, adding at least one additional bond: a contradiction to minimality of $(A,B)$. Moreover, we can assume that $l_1 \ge 2$ as $l_1 = 1$ implies also $h_1 = 1$, and the inequality $h_1 \leq h_2$ is automatically satisfied. Finally, we can suppose that the interface has at most one step, see Lemma \ref{lemma: step lemma}. The main step of the proof is to show that  $l_1 < l_3$. Indeed, then we obtain the contradiction
\begin{equation*}
 h_2 \leq h_1  \leq l_1 < l_3 \leq h_2.
\end{equation*}
Let us now prove $l_1 < l_3$. To this end, we will calculate the total number of points in two ways. Denote by $r_1$ the number of $A$-points in the leftmost column, by $h_1 + r_2$ the number of $A$-points in the leftmost double-type column, by $r_3$ the number of $B$-points in the leftmost double-type column, and by $r_4$ the number of $B$-points in the rightmost column. Then, we have
\begin{equation}
N_A = (l_1-1)(h_1 + h_2) + l_2 h_1 + r_1 + r_2
\end{equation}
and 
\begin{equation}
N_B = (l_2 + l_3 - 2) h_2 + r_3 + r_4.
\end{equation}
Now, we subtract one of these equations from the other. Since $r_1> 0, r_2 \geq 0$, and $r_3, r_4 \leq h_2$ we get
\begin{align*}
0 & = N_A - N_B = l_1 h_1 + l_1 h_2 + l_2 h_1 - h_1 - h_2 + r_1 + r_2 - l_2 h_2 - l_3 h_2 + 2h_2 - r_3 - r_4 \\
&>  (l_1 - 1) h_1 - h_2 + (l_1 - l_3) h_2 + l_2(h_1 - h_2)  \geq  (l_1 - l_3) h_2,
\end{align*}
where in the last step we used $l_1 \ge 2$ and the assumption (by contradiction) that  $h_1 \geq h_2$.  This shows $l_1 < l_3$ and concludes the proof.
\end{proof}

\begin{proposition}\label{prop:regularisationofclassIVpart4}
Fix $ N_A = N_B>0 $ and $ \beta \leq   1/2$. Suppose that $(A,B) \in \mathcal{IV}$ is an optimal configuration such that $h_1 \leq l_1$ and $h_3 = 0$. Then, there exists an optimal configuration $(\hat{A},\hat{B})$ such that $(\hat{A},\hat{B}) \in \mathcal{IV}$ with $h_3 = 0$ and $l_2 \in \{ 1,2 \}$.
\end{proposition}

\begin{proof}
Suppose that $(A,B)$ satisfies  $l_2 \geq 3$. By Lemma
\ref{lem:h1smallerthanh2} we have $h_1 \leq h_2$.  Then, let us remove
the rightmost two layers in $l_2:h_1$, and place the (at most $2h_1$)
$A$-points on the left of the configuration, at most one point in
every row. Since $h_1 \leq h_2$, there is enough space to place all
the points. In this way,  since the configuration can assumed to have only one step in the interface (see Lemma~\ref{lemma: step lemma}),   $l_1$ increases by at most 1, $l_2$ decreases
by 2, $l_3$ increases by 2, and all $h_i$ stay the same. Hence, by
formula \eqref{eq:classIVformula} we see that the energy stays the
same (for $ \beta  = 1/2$), so the resulting configuration is
optimal, or decreases (for $ \beta <  1/2$), so the original configuration was not optimal.  We repeat this procedure until $l_2 \in \lbrace 1,2 \rbrace$. 
\end{proof}

Hence, in order to prove existence of an optimal configuration in Class $\mathcal{I}$, we have two special cases to consider, depending on the value of $l_2$. We start with the case $l_2 = 1$.

\begin{proposition}\label{prop:regularisationofclassIVpart5}
Fix $ N_A = N_B>0 $ and $ \beta  \in (0,1)$. Suppose that $(A,B) \in \mathcal{IV}$ is an optimal configuration such that $h_3 = 0$ and $l_2 = 1$. Then, $h_1 = 1$. Furthermore, there exists an optimal configuration $(\hat{A},\hat{B}) \in \mathcal{I}$.
\end{proposition}

\begin{proof}
As in  \eqref{eq:formulafortheenergy}, let us write the energy as
\begin{equation}
E(A,B) = E_A + E_B - (h_2+1)\beta,
\end{equation}
where $E_A$ is  minus the number of bonds between points in $A$ and $E_B$ is minus the number of bonds between points in $B$.

We consider two cases. First, suppose that $E_B < E_A$. We do the following rearrangement of points: we separate $A$ and $B$ and suppose without restriction that  the leftmost column of $B$ is full as otherwise  we can move the points in this column to the right-hand side of $B$, without changing the $E_B$. We replace $A$ by $\hat{A}$, a reflection of $B$ along the vertical axis. Then we reconnect $\hat{A}$ and $B$ along the vertical line segment of length $h_2$. In this way, the resulting configuration has energy
\begin{equation}
E(\hat{A},B) = E_B + E_B - h_2 \beta.
\end{equation}
Hence, as $E_B \le E_A -1$, the energy drops by at least $1-\beta$, so the original configuration was not optimal, a contradiction.

Now, suppose that $E_A \leq E_B$. We do the following: we keep $A$ fixed    (or, as above,  we make $A$ flat on one side without changing $E_A$) and replace $B$ by $\hat{B}$, a reflection of $A$ along the vertical axis. Then, we join $A$ and $\hat{B}$ along the vertical line segment of length $h_1 + h_2$. In this way, the resulting configuration lies in Class $\mathcal{I}$, has a flat interface, and the energy is given by
\begin{equation}
E(A,\hat{B}) = E_A + E_A - (h_1+h_2) \beta.
\end{equation}
Therefore, the only way in which the energy does not decrease is that $E_A = E_B$ and $h_1 = 1$.
\end{proof}

We will employ another variant of the reflection argument to deal with the case $l_2 = 2$. This is formalised in the next proposition.

\begin{proposition}\label{prop:regularisationofclassIVstep6}
Fix $ N_A = N_B>0 $ and $ \beta \leq  1/2$. Suppose that $(A,B) \in \mathcal{IV}$ is an optimal configuration such that $h_3 = 0$ and $l_2 = 2$. Then,  $h_1 \le 2+1/\beta$ and there exists an optimal configuration $(\hat{A},\hat{B}) \in \mathcal{I}$.
\end{proposition}

\begin{proof}
  Again, as in  \eqref{eq:formulafortheenergy}, we write the energy as
\begin{equation*}
E(A,B) = E_A + E_B - (h_2 + 2)\beta.
\end{equation*}
We  consider three cases: first, suppose that either $E_B \le  E_A -2$ or $E_B =  E_A -1$ and $h_1 \le 2+1/\beta$. We do the following rearrangement of points: we keep $B$ fixed (up to making one side flat, as in the previous proof) and replace $A$ by $\hat{A}$, a reflection of $B$ along the vertical axis. Then, we join $\hat{A}$ and $B$ along the vertical line segment of length $h_2$. The resulting configuration  lies in Class~$\mathcal{I}$ and  satisfies 
\begin{equation}
E(\hat{A},B) = E_B + E_B - h_2 \beta.
\end{equation}
Hence, the energy drops by  $k -2\beta$, where $k = E_A - E_B \ge 1$. Thus, either the original configuration was not optimal (for $k \ge 2$ or $k=1$ and  $\beta < 1/2$) or the resulting configuration is optimal (for $k=1$ and  $\beta = 1/2$). Moreover, the resulting configuration lies in Class $\mathcal{I}$.

Now, suppose that either $E_B  = E_A - 1$ and $h_1 > 2+1/\beta$ or $E_B = E_A$ and $h_1 \ge 2$ or   $E_A < E_B$. This time, we keep $A$ fixed (up to making one side flat) and replace $B$ by $\hat{B}$, a reflection of $A$ along the vertical axis.  Then, we join $A$ and $\hat{B}$ along the vertical line segment of length $h_1 + h_2$. In this way, the resulting configuration lies in Class $\mathcal{I}$, has a flat interface, and the energy is given by
\begin{equation}
E(A,\hat{B}) = E_A + E_A - (h_1+h_2) \beta.
\end{equation}
Thus, the energy decreases by $k + (h_1-2)\beta$, where $k = E_B-E_A$. In particular, for $k=-1$ and $h_1  >  2+1/\beta$  or $k=0$ and $h_1 \ge 3$ or $k>0$ the energy drops. For $k=0$ and $h_1=2$ it stays the same, so the resulting configuration is optimal and  lies  in Class $\mathcal{I}$.

The only case left to consider is when $E_A = E_B$ and $h_1 = 1$.  We proceed as follows: we exchange the rightmost $A$-point (i.e., the rightmost point of the rectangle $l_2:h_1$) with the top $B$-point from column $C_{l_1+1}$, i.e., the point with two connections to points of type $A$ and two connections to points of type $B$.  If $C_{l_1+1}$ contains only one $B$-point, then the interface became shorter (without changing the overall shape of the configuration) and the energy actually drops. If it contains more then one $B$-point,  this  procedure did not change the energy. Moreover, the resulting configuration is in Class $\mathcal{I}$.  The construction is presented in Figure \ref{fig:exchangingonepoint}. 

\begin{figure}[h]

\definecolor{zzttqq}{rgb}{0.6,0.2,0.}
\begin{tikzpicture}[line cap=round,line join=round,>=triangle 45,x=0.5cm,y=0.5cm]
\clip(0.9,-5.1) rectangle (25.1,3.1);
\fill[line width=0.5pt,color=zzttqq,fill=zzttqq,fill opacity=0.1] (6.,-5.) -- (10.,-5.) -- (10.,-4.) -- (11.,-4.) -- (11.,1.) -- (7.,1.) -- (7.,-1.) -- (6.,-1.) -- cycle;
\fill[line width=0.5pt,color=zzttqq,fill=zzttqq,fill opacity=0.1] (21.,2.) -- (22.,2.) -- (22.,1.) -- (25.,1.) -- (25.,-4.) -- (24.,-4.) -- (24.,-5.) -- (20.,-5.) -- (20.,-2.) -- (21.,-2.) -- cycle;
\draw[line width=0.5pt] (5.,-5.)-- (3.,-5.);
\draw[line width=0.5pt] (8.,2.)-- (8.,1.);
\draw[line width=0.5pt] (10.,-4.)-- (10.,-5.);
\draw[line width=0.5pt] (10.,-5.)-- (6.,-5.);
\draw[line width=0.5pt] (6.,-5.)-- (5.,-5.);
\draw[line width=0.5pt] (8.,2.)-- (8.,2.);
\draw[line width=0.5pt] (5.,3.)-- (5.,3.);
\draw[line width=0.5pt] (3.,-3.)-- (2.,-3.);
\draw[line width=0.5pt] (6.,-2.)-- (6.,-2.);
\draw[line width=0.5pt] (6.,-2.)-- (6.,-5.);
\draw[line width=0.5pt] (10.,-4.)-- (11.,-4.);
\draw[line width=0.5pt] (11.,-4.)-- (11.,0.);
\draw[line width=0.5pt] (11.,0.)-- (11.,1.);
\draw[line width=0.5pt] (2.,-5.)-- (3.,-5.);
\draw[-latex,line width=0.5pt] (13.,-1.) -- (15.,-1.);
\draw[line width=0.5pt] (3.,-3.)-- (3.,2.);
\draw[line width=0.5pt] (3.,2.)-- (8.,2.);
\draw[line width=0.5pt] (7.,1.)-- (8.,1.);
\draw[line width=0.5pt] (7.,1.)-- (7.,0.);
\draw[line width=0.5pt] (8.,1.)-- (11.,1.);
\draw[line width=0.5pt] (7.,0.)-- (7.,-1.);
\draw[line width=0.5pt] (7.,-1.)-- (6.,-1.);
\draw[line width=0.5pt] (6.,-1.)-- (6.,-2.);
\draw[line width=0.5pt] (2.,-3.)-- (2.,-5.);
\draw[line width=0.5pt] (16.,-5.)-- (20.,-5.);
\draw[line width=0.5pt] (20.,-5.)-- (24.,-5.);
\draw[line width=0.5pt] (24.,-5.)-- (24.,-4.);
\draw[line width=0.5pt] (24.,-4.)-- (25.,-4.);
\draw[line width=0.5pt] (16.,-5.)-- (16.,-3.);
\draw[line width=0.5pt] (16.,-3.)-- (17.,-3.);
\draw[line width=0.5pt] (17.,-3.)-- (17.,2.);
\draw[line width=0.5pt] (25.,-4.)-- (25.,1.);
\draw[line width=0.5pt] (20.,-5.)-- (20.,-2.);
\draw[line width=0.5pt] (20.,-2.)-- (21.,-2.);
\draw[line width=0.5pt] (21.,-2.)-- (21.,2.);
\draw[line width=0.5pt] (21.,2.)-- (22.,2.);
\draw[line width=0.5pt] (22.,2.)-- (22.,1.);
\draw[line width=0.5pt] (22.,1.)-- (25.,1.);
\draw[line width=0.5pt] (21.,2.)-- (17.,2.);
\end{tikzpicture}

\caption{Final step of modification into Class $\mathcal{I}$}
\label{fig:exchangingonepoint}
\end{figure}
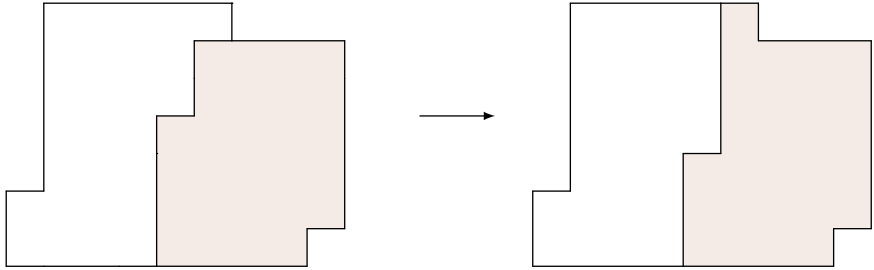

Summarising, we have shown that $h_1 \le 2+1/\beta$ and that there exists an optimal configuration in Class~$\mathcal{I}$.
\end{proof}

We summarise the reasoning from this subsection in the following result.

\begin{proposition}\label{prop:wemayrequireclassI}
Fix $ N_A = N_B>0 $ and $ \beta \leq  1/2$. Suppose that $(A,B) \in \mathcal{IV}$ is an optimal configuration. Then, there exists an optimal configuration $(\hat{A},\hat{B}) \in \mathcal{I}$.
\end{proposition}

\begin{proof}
By Proposition \ref{prop:classIVregularisationstep3}, there exists an optimal configuration with  $h_1 \leq l_1$ and $h_3 = 0$. Since $ \beta \leq  1/2$, by Proposition \ref{prop:regularisationofclassIVpart4} one may require additionally that $l_2 = 1$ or $l_2 = 2$. In both cases, existence of an optimal configuration in Class $\mathcal{I}$ is guaranteed by Proposition \ref{prop:regularisationofclassIVpart5} and by Proposition \ref{prop:regularisationofclassIVstep6}, respectively.
\end{proof}

\subsection{Class $\mathcal{V}$}

Finally, we show that we can modify optimal configurations in Class~$\mathcal{V}$ to optimal configuration in Class $\mathcal{IV}$. Along with Proposition~\ref{prop:wemayrequireclassI} this shows that there always exists a minimiser in Class~$\mathcal{I}$. This is done in the following proposition which employs a similar technique to the one used for Class $\mathcal{III}$.

\begin{proposition}\label{prop:classVregularisation}
Fix $N_A, N_B > 0$ and $ \beta  \in (0,1)$. Suppose that $(A,B) \in \mathcal{V}$. Then, there exists $(\hat{A},\hat{B}) \in \mathcal{IV}$ with $E(\hat{A},\hat{B}) \leq E(A,B)$. 
\end{proposition}

\begin{proof}
We will modify the top $h_1$ rows of the configuration $(A,B)$ in a similar fashion to the proof of Proposition \ref{prop:classIIIregularisationstep1}. For every $k \leq h_1$, we set $\hat{A}_k := A^{\rm row}_k + (-1,0)$. This translation implies $E^{\rm row}_k(\hat{A},\hat{B}) = E^{\rm row}_k(A,B)$ for all $k = 1,...,N_{\rm row}$. Regarding $E_k^{\rm  inter }$, a change is possible at most for $k = h_1$, where we did not change the number of $A$-$B$ connections and added zero or one $A$-$A$ connections, so $E_k^{\rm  inter }(\hat{A},\hat{B}) \leq E_k^{\rm  inter }(A,B)$. Hence, the total energy did not increase. We repeat this procedure for all rows  wit index $k \le h_1$   until the rightmost point of all $A^{\rm row}_{k}$ with $k \leq h_1$  does not lie right to  the rightmost point of $B^{\rm row}_{h_1+1}$.  We thus get a configuration which lies in Class~$\mathcal{IV}$.  
\end{proof}

\subsection{Conclusion}

Finally, we are in  the  position to state another of the main results, which together with Theorem \ref{thm:classIexact} gives the exact formula for the minimal energy,  see  Theorem \ref{thm:main}.iv. 

\begin{theorem}\label{thm:classIexistence}
Fix $ N_A = N_B>0 $ and $ \beta \leq  1/2$. Then, there exists an optimal configuration $(A,B)$ which lies in Class $\mathcal{I}$ and has a straight interface.
\end{theorem}

\begin{proof}
Since the number of points is finite, there exists an optimal configuration. By Theorem \ref{thm:connected} and the discussion below it, it lies in one of the five classes. However, it cannot lie in Class $\mathcal{II}$ by Proposition \ref{prop:classIIregularisation}. It also cannot lie in Class $\mathcal{III}$ by Proposition \ref{prop:classIIIexcluded}. If it lies in Class $\mathcal{V}$, then there exists a minimal configuration in Class  $\mathcal{IV}$  by virtue of Proposition \ref{prop:classVregularisation}. If it lies in Class $\mathcal{IV}$, then by Proposition \ref{prop:wemayrequireclassI} there exists a minimal configuration in Class $\mathcal{I}$. Finally, since there is an optimal configuration in Class $\mathcal{I}$, by Proposition \ref{prop:classIregularisationstep2} we may  suppose  that it has a flat interface.
\end{proof}

Let us note that in the above  theorem   we only state that a solution in Class $\mathcal{I}$ exists and that we cannot fully exclude existence of solutions in other classes. In particular, the following result shows that there exist arbitrarily large optimal configurations in Class $\mathcal{IV}$.

\begin{proposition}\label{prop:largeminimisersiv}
Let $\beta \in (0,1/2]\cap {\mathbb Q}$, $r,\, s
  \in {\mathbb N}$ with $r/s=1-\beta/2$, and $k \in {\mathbb N}$. Then, the
  Class-${\mathcal IV}$ configuration $(A,B)$ with 
  \begin{align*}
    A&=\{(x,y)\in {\mathbb Z}^2 \, \colon \, x \in [-kr+1,0], \ y \in
       [1,ks]\}\} \cup (1,ks),\\
B&=\{(x,y)\in {\mathbb Z}^2 \, \colon \, x \in [1,kr], \ y \in
   [0,ks-1]\}\cup (0,0)
  \end{align*}
is optimal. 
  \end{proposition}

\begin{proof}
Using  \eqref{eq:eq} and  formula \eqref{eq:classIVformula},  one can directly compute 
\begin{equation}
P(A,B) = 4kr + 2 (ks+1)  + 2(1-\beta)(ks+1).\label{eq:to_compare}
\end{equation}
To prove optimality, it hence suffices to check that
$P(A,B)=\min\{P_*,P^*\}$, where  $P_*$ and $P^*$ are defined in 
Theorem \ref{thm:main}.iv for $ N :=  N_A=N_B=k^2rs+1$.   From  $\beta \in (0,1/2]$ we get that $s/r = 2/(2-\beta) \in
(1,4/3]$. This in particular entails that $s>r \geq 2$, which in turn
allows to prove that
\begin{align*}
  \sqrt{\frac{2   N }{2-\beta}}   = \sqrt{\frac{k^2rs + 1}{r/s}} =   \sqrt{{k^2s^2 + s/r}} \in (ks,ks+1).
\end{align*}
In particular, we have checked that 
\begin{align*}
  \left\lfloor\sqrt{\frac{2   N }{2-\beta}}\right\rfloor=ks \quad\text{and}\quad \left\lceil
  \sqrt{\frac{2   N }{2-\beta}} \right\rceil=ks+1.
\end{align*}
One can hence compute 
\begin{align*}
  P_*&= 4 \left\lceil
       \frac{  N }{\left\lfloor\sqrt{\frac{2   N }{2-\beta}}\right\rfloor }
       \right\rceil + 2
       \left\lfloor\sqrt{\frac{2  N }{2-\beta}}\right\rfloor (2-\beta)\\
&=4 \left\lceil
       \frac{k^2rs +1}{ks }
       \right\rceil + 2 ks (2-\beta) = 4 \left\lceil
       kr + 1/ks 
       \right\rceil + 2 ks (2-\beta) =4 kr + 4 + 2 ks (2-\beta).
\end{align*}
On the other hand, using again the fact that  for  $s> r \geq 2$ we get that  
$$\frac{k^2rs +1}{ks+1} \in (kr-1,kr]$$ and we can compute
\begin{align*}
P^*&= 4 \left\lceil
       \frac{  N }{\left\lceil\sqrt{\frac{2   N }{2-\beta}}\right\rceil }
       \right\rceil + 2
       \left\lceil\sqrt{\frac{2   N }{2-\beta}}\right\rceil (2-\beta)\\
&=4 \left\lceil
       \frac{k^2rs +1}{ks+1}
       \right\rceil + 2(ks+1) (2-\beta) = 4kr + 2(ks+1) (2-\beta).
\end{align*}
We conclude that 
$$\min\{P_*,P^*\} = P^* \stackrel{\eqref{eq:to_compare}}{=}P(A,B)$$
which proves that $(A,B)$ is optimal.
  \end{proof}

\section{$N^{1/2}$-law and  $N^{3/4}$-law   for minimisers}\label{sec:law}

In this section, we give a quantitative upper bound on the difference of two optimal configurations,  see Theorem \ref{thm:main}.vi.   The goal is to prove that, even though in general there is no uniqueness of the optimal configurations and some of them may even not be in Class $\mathcal{I}$, they all have the same approximate shape.  In the following, an isometry $T\colon \mathbb{Z}^2 \rightarrow \mathbb{Z}^2$ indicates a composition of the translations $x \mapsto x + \tau$ for $\tau \in  \mathbb{Z}^2$, the rotation $(x_1,x_2) \mapsto  (-x_2, x_1) $ by the angle $\pi/2$, and the reflections $(x_1,x_2) \mapsto (x_1,-x_2)$, $(x_1,x_2) \mapsto (-x_1,x_2)$.

\begin{theorem}[$N^{1/2}$-law]\label{thm:nonehalf}
 Fix $N :=  N_A = N_B>0$  and $ \beta \leq 
\frac12$. Then, there exists a constant $ C_\beta$  only
depending on $ \beta $ such that for each two  optimal configurations  $(A,B)$ and $(A',B')$ it holds that 
\begin{equation}\label{eq: N12}
\min \bigg\{ \#(A \triangle   T(A')) +  \#(B \triangle T(B')) \colon \,
T\colon \mathbb{Z}^2 \rightarrow \mathbb{Z}^2 \mbox{ is an isometry}
\bigg\} \leq  C_\beta   N^{ \gamma(\beta)},
\end{equation}
 where $\gamma(\beta) = 1/2$ if $\beta \in \mathbb{R}\setminus \mathbb{Q}$ and    $\gamma(\beta) = 3/4$ if $\beta \in \mathbb{Q}$. 
\end{theorem}

\begin{proof}
Throughout the proof, $ C_\beta  $ is a constant which depends
only on $ \beta $ whose value may vary from line to line. We start the proof by mentioning that it suffices to check the assertion only for $N \ge N_0$ for some $N_0 \in \mathbb{N}$ depending only on $ \beta $. As observed in the proof of Theorem~\ref{thm:classIexistence}, every optimal configuration lies in the  Classes $\mathcal{I}$, $\mathcal{IV}$, $\mathcal{V}$. In Step 1, we show \eqref{eq: N12} for two optimal configurations in Class $\mathcal{I}$. Afterwards, in Step 2 we show that for each optimal configuration $(A,B)$ in Class $\mathcal{IV}$ there exists $(A',B')$ in Class $\mathcal{I}$ such that \eqref{eq: N12} holds. Eventually, in Step 3 we check that for each optimal configuration $(A,B)$ in Class $\mathcal{V}$ there exists $(A',B')$ in Class $\mathcal{IV}$ such that \eqref{eq: N12} holds. The combination of these three steps yields the statement.

{\bf Step 1: Class $\mathcal{I}$.} Let first $(A,B)$ be an optimal
configuration in Class $\mathcal{I}$ such that $l_2 = 0$.  Then by Theorem \ref{thm:main}.iv   
we find $\bar{h} \in \mathbb{N}$ with $|\bar{h}- \sqrt{2  N /(2-\beta)}| \le C_\beta N^{1/4}$ such that   
\begin{equation}\label{eq: fixing hhh} 
h \sim   \bar{h},   \quad \quad \quad
l_1=l_3 \sim  \sqrt{N/\bar{h}}, 
\end{equation}
 where here and in the following $\sim$ indicates that equality holds
up to a  constant  $C_\beta N^{\gamma(\beta) - 1/2}$.  Indeed, for
$\beta \in \mathbb{R}\setminus \mathbb{Q}$, the value  $\bar{h}$
 is unique, whereas for $\beta \in \mathbb{Q}$ it lies in an
interval whose  diameter is at most  of order $N^{1/4}$.   Consequently, two optimal configurations in Class $\mathcal{I}$ with $l_2= 0$ clearly  satisfy \eqref{eq: N12}.  Also, notice that since the interface is straight, reflection along the interface exchanges the roles of the sets $A$ and $B$.  Now, consider an optimal configuration $(A,B)$ in Class~$\mathcal{I}$ with $l_2>0$. Then we get $l_2=1$ by Proposition \ref{prop:classIregularisationstep1}. The regularisation of Proposition \ref{prop:classIregularisationstep2} shows that $(A,B)$ can be modified to a configuration $(A',B')$ in Class $\mathcal{I}$ with $l_2 = 0$ such that \eqref{eq: N12} holds. Indeed, in this regularisation we only alter the configurations involving the single column containing points of both types and possibly merge two connected components by moving one connected component by $(1,0)$. This concludes Step 1  of the proof.

{\bf Step 2: Class $\mathcal{IV}$.} We now consider an optimal
configuration in Class $\mathcal{IV}$ and show that it can be modified
to a configuration in Class $\mathcal{I}$ such that \eqref{eq: N12}
holds.  We will work through the proofs in Subsections
\ref{sec:classIVpartone} and \ref{sec:classIVparttwo} in reverse
order. Our strategy is as follows: we use the knowledge of the
structure of the final step of the regularisation procedure, obtain
some a posteriori bounds on the size of $l_i$ and $h_i$, and go back
to see how these can change at every step of the regularisation
procedure. Eventually, this will allow us to show that already after
the first modification described  in  Lemma
\ref{lem:classIVregularisationstep1} we obtain an optimal
configuration in Class~$\mathcal{I}$, by moving at most $ C_\beta  N^{1/2}$ many points.  This will conclude Step 2 of the proof.

{\bf Step 2.1.} Our starting points are Propositions \ref{prop:regularisationofclassIVpart5} and \ref{prop:regularisationofclassIVstep6}: recall that applying all the intermediate steps, in the end we have $h_3 = 0$ and we land with an alternative $l_1 = 1$ (which is covered in Proposition \ref{prop:regularisationofclassIVpart5}) or $l_2 = 2$ (which is covered by Proposition \ref{prop:regularisationofclassIVstep6}). In both cases, before applying these propositions, we have
\begin{align}\label{eq: 2.1}
l_2 \leq 2, \quad
h_1 \leq 2 + \frac{1}{\beta}, \quad h_3 = 0, \quad h_2 \sim
  \bar{h},  \quad  l_1,l_3 \sim
   \sqrt{N/\bar{h}}.      
\end{align}
In fact, the last  conditions follow from \eqref{eq: fixing hhh} (for $h=h_2$) and the reflection procedure described in the propositions. 

{\bf Step 2.2.} Now, we go  a step back in the regularisation
procedure. In Proposition~\ref{prop:regularisationofclassIVpart4}, for
$ \beta <  1/2$ nothing changes and the same bounds hold. For
$ \beta  = 1/2$, 
\eqref{eq: 2.1} yields  that $\frac{h_1}{h_2}\to 0$ as $N \rightarrow
\infty$. This implies that in
Proposition~\ref{prop:regularisationofclassIVpart4}, for sufficiently
large $N$, we move at most two layers. In fact, if we moved at least
three layers, the energy would strictly decrease since all of them fit
into a single column. Hence, for sufficiently big $N$ (depending only
on $ \beta$),   we have the following bounds
\begin{align}\label{eq: 22}
l_2 \leq 4, \quad h_1 \leq 2 + \frac{1}{\beta}, \quad h_3 = 0,  \quad h_2 
   \sim \bar{h},  \quad  l_1,l_3 \sim
   \sqrt{N/\bar{h}}.     
\end{align}
Finally, let us take one more step back in the regularisation procedure. In Lemma~\ref{lemma: step lemma}, we actually modify the configuration only slightly inside the rectangle $l_2:h_2$. 
In this way, $h_i$ and $l_i$ were not altered, so that the bounds \eqref{eq: 22} still holds.

{\bf Step 2.3.} Now we come to the main part of the regularisation procedure, i.e., Proposition~\ref{prop:classIVregularisationstep2}. In its proof, we apply an iterative procedure, and at every step one of the following changes happens:
\begin{align*}
 {\rm (a)}   \ \ \ h_2 \rightarrow h_2 + 1, \quad l_2 \rightarrow l_2-1, \quad h_3 \rightarrow h_3 - 1, \quad l_3 \rightarrow l_3 + 1, \quad h_1 \rightarrow h_1, \quad l_1 \rightarrow l_1
\end{align*}
or
\begin{align*}
 {\rm (b)}   \ \ \  h_2 \rightarrow h_2 + 1, \quad l_2 \rightarrow l_2-1, \quad h_1 \rightarrow h_1 - 1, \quad l_1 \rightarrow l_1 + 1, \quad h_3 \rightarrow h_3, \quad l_3 \rightarrow l_3.
\end{align*}
  Notice that in both cases $l_1$ and $l_3$ cannot decrease during this procedure, and exactly one of them increases at every step. The procedure can end in two ways:  $h_3 = 0$ (or equivalently $h_1 = 0$) or $l_2 = 1$. In the latter case, however, the proof of Proposition \ref{prop:classIVregularisationstep3} implies that the original configuration was not optimal, so we only need to examine the former case.

Consider the last step of the regularisation procedure in the proof of Proposition~\ref{prop:classIVregularisationstep2}, i.e., the one before we reach $h_3 = 0$. Denote by $\hat{h}_1$ the value of $h_1$ at the end of the regularisation procedure, and note that $\hat{h}_1 \leq 2 + \frac{1}{\beta}$ by \eqref{eq: 22}. There are two possible situations: either
\begin{equation}
 \hat{l}_1 \leq 2  \hat{h}_1 \qquad \mbox{or} \qquad  \hat{l}_1 > 2  \hat{h}_1.
\end{equation}
In the second case, notice that we cannot  have applied  the construction from case (a) twice as  otherwise   a slightly modified procedure would give the following:  we move the $A$-points from the rightmost two columns of the rectangle $l_2:h_1$ to the rectangle $l_1:h_3$, but we place them in a single row. In this way, we have
\begin{equation}
h_2 \rightarrow h_2 + 1, \quad l_2 \rightarrow l_2-2, \quad h_3 \rightarrow h_3 - 1, \quad l_3 \rightarrow l_3 + 2, \quad h_1 \rightarrow h_1, \quad l_1 \rightarrow l_1.
\end{equation}
This shows that the energy \eqref{eq:classIVformula} strictly decreases as the length of the interface is decreased.  Hence, the original configuration was not optimal, so either $ \hat{l}_1 \leq 2  \hat{h}_1$ or we have applied a step of type (a) at most once.

Similarly, since $\hat{h}_3$ at the end of the procedure equals zero, we consider the alternative
\begin{equation*}
 \hat{l}_3 \leq 2 \qquad \mbox{or} \qquad \hat{l}_3 > 2. 
\end{equation*}
We apply a similar argument to conclude that either $ \hat{l}_3 \leq 2 $ or that we have  applied a step of type (b) at most once.

In view of \eqref{eq: 22},  and because $l_1$ and $l_3$ can only
increase during the regularisation procedure,  we see that  $l_1
\leq 2 \hat{h}_1$ and $l_3 \leq 2$  lead to contradictions for $N$
sufficiently large depending only $ \beta $.  This implies that there can be at most one step of type (a) and (b), respectively.
Therefore, using again  \eqref{eq: 22} we see that before the application of Proposition~\ref{prop:classIVregularisationstep2} it holds that 
\begin{align}\label{eq: 23}
l_2 \leq 6, \quad h_1 \leq 4 + \frac{1}{\beta}, \quad h_3 \le  2,
  \quad   \quad  
   h_2 \sim \bar{h},  \quad  l_1,l_3 \sim
   \sqrt{N/\bar{h}}.        
\end{align}

{\bf Step 2.4.} Finally, we consider the modification in  Lemma
\ref{lem:classIVregularisationstep1}. For simplicity, we only address
the modification leading to $\min \lbrace h_1, h_1+h_2 -l_1-l_2\rbrace
\le 0$. Note that each step of the procedure consists in  $h_1
\rightarrow h_1-  1$ and   $l_1 \rightarrow l_1 +1$. As after the
application of Lemma \ref{lem:classIVregularisationstep1} we have
$\hat{h}_2/\hat{l}_1 \ge 2/( 2 - \beta   ) + {\rm O}(1/\sqrt{N})$, see \eqref{eq: 23},  and during its application $h_2$ does not change and $l_1$ can only increase,  at each step of the procedure it holds that ${h}_2/{l}_1 \ge 2/( 2 - \beta ) + {\rm O}(1/\sqrt{N})$. In view of \eqref{eq: 23},  in particular the fact that $l_2 \leq 6$,  for $N$ sufficiently large depending only on  $\beta$  we have 
\begin{align}\label{eq: good con}
(h_1+  h_2) / (l_1+l_2)  \geq h_2/(l_1 + l_2)   \ge  c_\beta 
\end{align} 
at each step of the procedure,  for some constant $ c_\beta
>1$ only depending on $ \beta $.  This ensures that at the beginning we have $h_1 \le M$ for $M \in \mathbb{N}$ such that $(M+1)/M < c_\beta $ since otherwise $M+1$  rows could be moved to $M$ columns leading to a strictly smaller energy. This along with \eqref{eq: 23} shows that at most $ C_\beta  N^{1/2}$ are moved. Moreover, the modifications stops once $h_1=0$ or $h_1 + h_2 \le l_1 + l_2$ as been obtained. By \eqref{eq: good con} we see that it necessarily holds $h_1 = 0$. In a similar fashion, one gets  $h_3 = 0$. This shows that directly after the application of Lemma \ref{lem:classIVregularisationstep1} we obtain a configuration in Class~$\mathcal{I}$. This concludes the proof as we have seen that in the modification of Lemma \ref{lem:classIVregularisationstep1} only $ C_\beta  N^{1/2}$ points are moved.

{\bf Step 3: Class $\mathcal{V}$.} We now consider an optimal
configuration in Class $\mathcal{V}$ and show that it can be modified
to a configuration in Class $\mathcal{IV}$ such that \eqref{eq: N12}
holds. The modification in Proposition~\ref{prop:classVregularisation}
consists in moving at most $h_1$ rows to  the left. 
By Step 2 we know that $h_1\le  C_\beta $ which implies that we have moved at most $ C_\beta  N^{1/2}$ many points. This concludes the proof of Step 3.  
\end{proof}

Let us highlight that in the proof of Theorem \ref{thm:nonehalf} we
have not only shown the $N^{1/2}$-law  and  $N^{3/4}$-law   for
minimisers, but we also get explicit estimates on the shape of the
configuration, written as  a separate    statement here below.
The following corollary  is a consequence of equations \eqref{eq: 23}, \eqref{eq: good con},  and the procedure from Step 3 of the proof of Theorem \ref{thm:nonehalf}. 

\begin{corollary}
Suppose that $(A,B) \in \mathcal{IV}  \cup \mathcal{V} $ is an optimal configuration. Then, 
\begin{align*}
l_2, h_1, h_3 \leq  C_\beta , \quad   h_2 \sim \bar{h},  \quad  l_1,l_3 \sim
   \sqrt{N/\bar{h}},       
\end{align*} 
 where $\bar{h}$ is a minimiser of \eqref{eq: periper}. 
\end{corollary}

 Recall that for $\beta \in \mathbb{R} \setminus \mathbb{Q}$ the
minimiser $\bar{h}$ is unique. Thus, in this case   the
quantitative bound given in Theorem \ref{thm:nonehalf} is sharp: the
optimal configuration in Class $\mathcal{IV}$ given by Proposition
\ref{prop:largeminimisersiv} differs from the one given in Theorem
\ref{thm:main}.v by a number of points of exactly this order.  In
the case  $\beta \in  \mathbb{Q}$, the  $N^{3/4}$-law  can
again be checked to be  sharp. This will be addressed  in a forthcoming paper.

\section{Proofs in the continuum setting}\label{sec:wulff}

 We conclude by providing the proofs of Corollaries \ref{cor: wulff}
and \ref{cor: cryst-db} from the Introduction.

\begin{proof}[Proof of Corollary \ref{cor: wulff}]
For the explicit solution $(A'_{ N },B'_{ N })$ in Theorem \ref{thm:main}.v with $N_A=N_B=:  N $, one can directly verify that  $\mu_{A'_{ N }} \stackrel{\ast}{\rightharpoonup}  {\mathcal         L} \mres {\mathcal A}$ and $\mu_{B'_{ N }} \stackrel{\ast}{\rightharpoonup}  {\mathcal L} \mres {\mathcal B}$, where $\mathcal{A}$ and $\mathcal{B}$ are given in \eqref{eq:wulff}. For  a general sequence of solutions $(A_{ N },B_{ N })$  of \eqref{eq:dbp},   the statement follows from the fluctuation estimate in Theorem \ref{thm:main}.vi.  
\end{proof}

\begin{proof}[Proof of Corollary \ref{cor: cryst-db}]
We start by relating point configurations with sets of finite
perimeter: given $(A_{ N },B_{ N })$ with $N_A = N_B
=:  N  $, we define the sets 
\begin{align}\label{eq: sofp}
A^{ N }:= \frac{1}{\sqrt{ N }} {\rm int}\Big(\bigcup_{p \in A_{ N }}  p + [-\tfrac{1}{2}, \tfrac{1}{2}]^2\Big), \quad \quad \quad B^{ N } := \frac{1}{\sqrt{ N }}   {\rm int}\Big( \bigcup_{p \in B_{ N }}  p + [-\tfrac{1}{2}, \tfrac{1}{2}]^2\Big).
\end{align}
Clearly, $A^{ N }$ and $B^{ N }$ satisfy $A^{ N } \cap B^{ N } =
 \emptyset$ and $\mathcal L(A^{ N })=\mathcal L(B^{N})=1$. It is an
  elementary matter to check that \eqref{eq:dbp} and \eqref{eq:dbp2}
  coincide in this case up to normalisation, i.e., 
\begin{align}\label{eq: scal-en-es}
  N ^{-1/2} P(A_{ N },B_{ N }) =  P_{\rm cont}(A^{ N },B^{ N }) := {\rm Per} (A^{ N }) + {\rm Per} (B^{ N }) - 2\beta {\rm L} (\partial^* A^{ N }
  \cap \partial^* B^{ N }).
\end{align}

Now, consider any pair of sets of finite perimeter with  $A \cap B = \emptyset$ and $\mathcal L(A)=\mathcal L(B)=1$. Given $\varepsilon>0$, by the density result \cite[Theorem 2.1 and Corollary 2.4]{BraidesDensity} (for $\mathcal{Z}$ consisting of three values representing $A$, $B$, and the emptyset) we can find $A'$ and $B'$ with polygonal boundary such that $A' \cap B' = \emptyset$, $\mathcal L(A')=\mathcal L(B')=1$, and 
$$P_{\rm cont}(A',B') \le P_{\rm cont}(A,B) + \varepsilon. $$
(Strictly speaking,  the constraint $\mathcal L(A')=\mathcal L(B')=1$
has not been addressed there. However, possibly after scaling one can
assume that $\mathcal L(A')\le 1$, $\mathcal L(B') \le 1$, and then it
suffices to add a disjoint  squares  of small volume and surface to satisfy the constraint.)
We define a point configuration related to $A'$ and $B'$ by setting
$$A_{ N } = \lbrace  p \in \mathbb{Z}^2 \colon \,  p/\sqrt{ N } \in A' \rbrace, \quad \quad \quad   B_{ N } = \lbrace  p \in \mathbb{Z}^2 \colon \,  p/\sqrt{ N } \in B' \rbrace.  $$  
By $A^{ N }$ and $B^{ N }$ we denote the corresponding sets of finite
perimeter defined in \eqref{eq: sofp}.  Note that the sets  $A^{ N }$
and $B^{ N }$ may have different cardinalities, although $\mathcal
L(A')=\mathcal L(B')=1$. Still, equal cardinalities can be restored by
adding points to one of the two sets. This can be achieved at the price of making a
small error in the perimeter, which goes to $0$ with $ N $ after
rescaling.     
The fact that $(A', B')$ have polygonal boundary along with the properties of $\Vert \cdot \Vert_1$ implies that 
$$\lim_{ N  \to \infty} P_{\rm cont}(A^{ N },B^{ N }) = P_{\rm cont}(A',B').$$ 
 In fact, each segment of the polygonal boundary of $(A',B')$ is approximated by a path consisting of horizontal and vertical segments which is contained in the boundary of the squares forming $(A^N,B^N)$, see \eqref{eq: sofp}. The $l^1$-norm of the segment and of the path coincide, up to an error of order $\frac{1}{{\sqrt{N}}}$.   
 This along with \eqref{eq: scal-en-es} and Theorem \ref{thm:main}.iv yields
 \begin{align*}
 P_{\rm cont}(A,B) & \ge \liminf_{ N \to \infty} P_{\rm cont}(A^{ N },B^{ N }) - \varepsilon \ge   \liminf_{ N  \to \infty}  N ^{-1/2}  \min_{h \in \mathbb{N}} \big(4\llceil N/h\rrceil +2 h( 2 - \beta )\big)   -\varepsilon \\
 &  = \liminf_{ N  \to \infty}  \min_{h \in \mathbb{N}} \big(4  \sqrt{N}/h  + 2\frac{h}{\sqrt{N}}( 2 - \beta )\big)   -\varepsilon. 
 \end{align*}
  Optimisation with respect to $h$ yields 
 $$  P_{\rm cont}(A,B) \ge     4  \frac{1}{\sqrt{\frac{2}{2-\beta}} } +2 
  \sqrt{\frac{2}{2-\beta}} (2-\beta) -\varepsilon = 4\sqrt{2}\sqrt{2-\beta} -\varepsilon. $$
  We directly compute  $P_{\rm cont}(\mathcal A,\mathcal B) =
  4\sqrt{2}\sqrt{2-\beta}$. As $\varepsilon>0$  is  arbitrary, we conclude that   the pair $(\mathcal A,\mathcal B)$ is a solution of \eqref{eq:dbp2}. 
 \end{proof}


{\flushleft Acknowledgments.} The authors are indebted to Frank Morgan for pointing out many relevant references. 

{\flushleft Funding.} MF  acknowledges support of the DFG project FR 4083/3-1.  This work was supported by the Deutsche Forschungsgemeinschaft (DFG, German Research Foundation) under Germany's Excellence Strategy EXC 2044-390685587, Mathematics M\"unster: Dynamics--Geometry--Structure.  WG acknowledges support of the FWF grant I4354, the OeAD-WTZ project CZ
01/2021, and the grant 2017/27/N/ST1/02418 funded by the National Science Centre, Poland. US acknowledges support of the FWF grants I4354, F65, I5149, and P\,32788, and by the OeAD-WTZ project CZ 01/2021.



\begin{thebibliography}{99}







  
\bibitem{Ambrosio-Fusco-Pallara}
L.~Ambrosio, N.~Fusco, D.~Pallara. {\it Functions of bounded variation
  and free discontinuity problems}. Oxford Mathematical
Monographs. The Clarendon Press, Oxford University Press, New York,
2000. 

 
\bibitem{Barrett}
J.~W.~Barrett, H.~Garcke, R.~N\"urnberg.
Numerical approximation of anisotropic geometric evolution equations
in the plane. {\it 
IMA J. Numer. Anal.} {\bf 28} (2008), no. 2, 292--330. 


\bibitem{Betermin}
L.~B\'etermin, H.~Kn\"upfer, F.~Nolte. Note on crystallization for
alternating particle chains. {\it J. Stat. Phys.} {\bf 181} (2020), no. 3,
803--815. 


\bibitem{Bezrukov0}
S.~L.~Bezrukov. Isoperimetric problems in discrete spaces. {\it
  Extremal problems for finite sets (Visegr\'ad, 1991)}, 59--91,
Bolyai Soc. Math. Stud., 3, J\'anos Bolyai Math. Soc., Budapest, 1994.


  \bibitem{Bezrukov}
S.~L.~Bezrukov.
Edge isoperimetric problems on graphs, in: Graph theory and combinatorial biology (Balatonlelle, 1996).
{\it Bolyai Soc. Math. Stud.}  {\bf 7} (1999), 157--197. 


\bibitem{Biskup}
M.~Biskup, O.~Louidor, E.~B.~Procaccia, R.~Rosenthal.
Isoperimetry in two-dimensional percolation. {\it
Comm. Pure Appl. Math.} {\bf 68} (2015), no. 9, 1483--1531.

\bibitem{Bobkov}
S.~G.~Bobkov, F.~G\"otze. Discrete isoperimetric and Poincar\'e-type
inequalities. {\it Probab. Theory Related Fields}, {\bf 114} (1999), no. 2,
245--277.

\bibitem{Bollobas}
B.~Bollob\'as, I.~Leader. Edge-isoperimetric inequalities in the
grid. {\it Combinatorica}, {\bf 11} (1991), no. 4, 299--314.



\bibitem{BraidesDensity}
A.~Braides, S.~Conti, A.~Garroni. 
Density of polyhedral partitions. 
{\it Calc.\ Var.\ Partial Differential Equations} 
{\bf 56} (2017),  Paper No.~28.   


\bibitem{Boyer}
W.~Boyer, B.~Brown, A.~Loving, S.~Tammen. Double bubbles in hyperbolic
surfaces. {\it Involve}, {\bf 11} (2018), no. 2, 207--217. 

\bibitem{Carrion}
  M.~Carri\'on \'Alvarez, J.~Corneli, G.~Walsh, S.~Beheshti. Double
  bubbles in the three-torus. {\it Experiment. Math.} {\bf 12} (2003),
  no. 1, 79--89.

\bibitem{Cerf}
  R.~Cerf. {\it The Wulff crystal in Ising and percolation models}.  Lecture Notes in Mathematics, 1878. Springer-Verlag, Berlin, 2006.

\bibitem{Cerf2}
 R.~Cerf, \'A.~Pisztora. On the Wulff crystal in the Ising model. {\it
   Ann. Probab.} {\bf 28} (2000), no. 3, 947--1017. 
  
\bibitem{Cicalese}
 M.~Cicalese, G.~P.~Leonardi. Maximal fluctuations on periodic
 lattices: an approach via quantitative Wulff inequalities. {\it
   Comm. Math. Phys.} {\bf 375} (2020), no. 3, 1931--1944.

\bibitem{Cicalese2} M.~Cicalese, G.~P.~Leonardi, F.~Maggi. Sharp
  stability inequalities for planar double bubbles. {\it Interfaces
    Free Bound.} {\bf 19} (2017), no. 3, 305--350. 

 
\bibitem{Corneli2}
J.~Corneli,I.~Corwin, S.~Hurder,V.~Sesum, Y.~Xu, E.~Adams, D.~Davis, M.~Lee, R.~ Visocchi, N.~Hoffman.
Double bubbles in Gauss space and spheres. 
{\it Houston J. Math.} {\bf 34} (2008), no. 1, 181--204.



\bibitem{Corneli0}
J.~Corneli, P.~Holt, G.~Lee, N.~Leger, E.~Schoenfeld, B.~Steinhurst, The double bubble problem on the flat
two-torus. {\it Trans. Amer. Math. Soc.} {\bf 356} (2004), no. 9, 3769--3820.


  \bibitem{Corneli}
J.~Corneli, N.~Hoffman, P.~Holt, G.~Lee, N.~Leger, S.~Moseley,
E.~Schoenfeld. Double bubbles in ${\mathbb S}^3$ and ${\mathbb
  H}^3$. {\it J. Geom. Anal.} {\bf 17} (2007), no. 2, 189--212.


\bibitem{Cotton}
 A.~Cotton, D.~Freeman. The double bubble problem in spherical space
 and hyperbolic space. {\it Int. J. Math. Math. Sci.} {\bf 32} (2002),
 no. 11, 641--699.

\bibitem{Davoli}
E.~Davoli, P.~Piovano, U.~Stefanelli.  Wulff shape emergence in
graphene. {\it Math. Models Methods Appl. Sci.} {\bf 26} (2016), no. 12,
2277--2310.


\bibitem{Duncan0}
P.~Duncan, R.~O'Dwyer, E.~B.~Procaccia.
An elementary proof for the double bubble problem in $\ell^1$
norm.  {\it J. Geom. Anal.}, {\bf 33} (2023), no. 1, Paper No. 31, 26
pp.  

\bibitem{Duncan}
P.~Duncan, R.~O'Dwyer, E.~B.~Procaccia. 
Discrete $\ell^1$ double bubble solution is at most ceiling plus two of the
continuous solution. {\it Discrete Comput. Geom.} (2023), doi.org/10.1007/s00454-023-00501-4. 



\bibitem{Foisy}
   J.~Foisy, M.~Alfaro, J.~Brock, N.~Hodges, J.~Zimba. The standard
   double soap bubble in $\mathbb R^2$
uniquely minimizes perimeter. {\it Pacific J. Math.} {\bf 159} (1993), no. 1, 47--59.

\bibitem{Franceschi}
  V.~Franceschi, G.~Stefani. Symmetric double bubbles in
   the Grushin plane. {\it ESAIM Control Optim. Calc. Var.} {\bf 25} (2019), Paper
   No. 77, 37 pp.


\bibitem{Futer}
D.~Futer, A.~Gnepp, D.~McMath, B.~A.~Munson, T.~Ng,
S.-H.~Pahk, C.~Yoder.
Cost-minimizing networks among immiscible fluids in $\mathbb{R}^2$. {\it Pacific
J. Math.} {\bf 196} (2000) 395--414.



   
\bibitem{kreutz}
{M.~Friedrich, L. Kreutz}. 
\newblock {Crystallization in the hexagonal lattice for ionic
  dimers}. {\it Math. Models Meth. Appl. Sci.}  {\bf 29} (2019),
1853--1900. 



\bibitem{kreutz2}
{M.~Friedrich, L.~Kreutz}. 
\newblock  {Finite crystallization and Wulff shape emergence for ionic compounds
    in the square lattice}.  {\it Nonlinearity}, {\bf 33} (2020),
  1240--1296.  


 \bibitem{periodic}
    {M.~Friedrich, U.~Stefanelli}.  Crystallization in a
    one-dimensional periodic landscape. {\it J. Stat. Phys.} {\bf 179} (2020), no. 2, 485--501.
    
\bibitem{Harper}
L.~H.~Harper.
{\it Global methods for combinatorial isoperimetric problems.}
Cambridge Studies in Advanced Mathematics, 90. Cambridge University
Press, Cambridge, 2004.

 \bibitem{Heitmann}
   R.~C.~Heitmann, C.~Radin. The ground states for sticky discs. {\it
     J. Stat. Phys.} {\bf 22} (1980), no. 3, 281--287.
  
\bibitem{Hutchings}
   M.~Hutchings, F.~Morgan, M.~Ritor\'e, A.~Ros.
   Proof of the double bubble conjecture. {\it Ann. of Math. (2)}, {\bf 155}
   (2002), no. 2, 459--489.

   

 \bibitem{Lopez}
 R.~Lopez, T.~Borawski~Baker. The double bubble problem on the
 cone. {\it New York J. Math.} {\bf 12} (2006), 157--167.

\bibitem{Maggi}
F.~Maggi. {\it
Sets of finite perimeter and geometric variational problems}.
An introduction to geometric measure theory. Cambridge Studies in Advanced Mathematics, 135. Cambridge University Press, Cambridge, 2012.

\bibitem{Mainini}
 E.~Mainini, P.~Piovano, U.~Stefanelli. Finite crystallization in the
 square lattice. {\it Nonlinearity}, {\bf 27} (2014), no. 4, 717--737.

  
 \bibitem{Mainini2}
E.~Mainini, B.~Schmidt. Maximal fluctuations around the Wulff
shape for edge-isoperimetric sets in ${\mathbb Z}^d$: a sharp scaling
law. {\it Comm. Math. Phys.} {\bf 380} (2020), no. 2, 947--971.


\bibitem{Masters}
J.~D.~Masters.
The perimeter-minimizing enclosure of two areas in $S^2$. 
{\it Real Anal. Exchange}, {\bf 22} (1996/97), no. 2, 645--654.

 
\bibitem{McCoy}
B.~M.~McCoy, T.~T.~Wu. {\it The two-dimensional Ising model. Second
edition}. Dover Publications, Inc., Mineola, NY,
2014.


\bibitem{Milman}
E.~Milman, J.~Neeman.
The Gaussian double-bubble and multi-bubble conjectures.  
{\it Ann. of Math. (2)}, {\bf 195} (2022), no. 1, 89--206.



   \bibitem{Morgan}
F.~Morgan. Area-minimizing surfaces in cones. {\it Comm. Anal. Geom.}
{\bf 10} (2002), no. 5, 971--983.


\bibitem{Morgan98}
F.~Morgan, C.~French, S.~Greenleaf. Wulff clusters in $\mathbb{R}^2$. {\it J. Geom. Anal.} {\bf 8} (1998), 97--115.



\bibitem{Radin86} C.~Radin. Crystals and quasicrystals: a continuum model. {\it Comm. Math. Phys.} {\bf 105} (1986), 385--390.

\bibitem{Nagamochi}
H.~Nagamochi, T.~Ibaraki. {\it Algorithmic aspects of graph
connectivity}.  Encyclopedia of Mathematics and its Applications, 123. Cambridge University Press, Cambridge, 2008.
  
 \bibitem{Reichardt}
   B.~W.~Reichardt. Proof of the double bubble conjecture in ${\mathbb
     R}^n$. {\it J. Geom. Anal.} {\bf 18} (2008), no. 1, 172--191.

 \bibitem{Schmidt}
B.~Schmidt. Ground states of the 2D sticky disc model: fine properties
and $N^{3/4}$ law for the deviation from the asymptotic Wulff
shape. {\it J. Stat. Phys.} {\bf 153} (2013), no. 4, 727--738.

\bibitem{Tasaki}
  H.~Tasaki. {\it Physics and mathematics of quantum many-body
    systems}. Graduate Texts in Physics. Springer, Cham, 2020. 


\bibitem{Wang}
D.~L.~Wang, P.~Wang. Discrete isoperimetric problems. {\it SIAM
  J. Appl. Math.} {\bf 32} (1977), no. 4, 860--870.

\bibitem{Wecht}
B.~Wecht, M.~Barber, J.~Tice. Double crystals. {\it Acta
   Cryst. Sect. A}, {\bf 56} (2000), no. 1, 92--95.

 
\end{thebibliography}
\end{document}